\documentclass[10pt,a4paper,dvipsnames]{article}

\usepackage{fullpage,amssymb,amsmath,amsfonts,amsthm,paralist,graphicx,color,float, mathtools,tikz,xfrac}
\usepackage[utf8]{inputenc}
\usepackage[T1]{fontenc}
\usepackage[english]{babel}
\usepackage[colorlinks=true,linkcolor=BrickRed,citecolor=blue]{hyperref}
\usepackage[normalem]{ulem}
\usepackage{centernot}
\usepackage{ifthen,xkeyval, tikz, calc, graphicx}
\usepackage{xcolor}
\usepackage{framed}
\usepackage[textsize=scriptsize, backgroundcolor=blue!20!white,bordercolor=red]{todonotes}
\usepackage{dsfont}
\usepackage{mathrsfs}
\usepackage{comment}
\usepackage{stmaryrd}
\usepackage{enumitem}
\usepackage{caption}
\usepackage{subcaption}
\usepackage{thmtools}
\usetikzlibrary{plotmarks}
\usetikzlibrary{shapes.misc}

\newtheorem{theorem}{Theorem}[section]
\newtheorem{lemma}[theorem]{Lemma}
\newtheorem{prop}[theorem]{Proposition}
\newtheorem{corollary}[theorem]{Corollary}

\newtheorem{conjecture}[theorem]{Conjecture}

\theoremstyle{definition}
\newtheorem{definition}[theorem]{Definition}

\newtheorem{assumption}{Assumption}
\newtheorem{assumptionalt}{Assumption}[assumption]
\newenvironment{assumptionp}[1]{
  \renewcommand\theassumptionalt{#1}
  \assumptionalt
}{\endassumptionalt}

\theoremstyle{remark}
\newtheorem{remark}[theorem]{Remark}

\definecolor{shadecolor}{named}{GreenYellow}

\makeatletter
\newcommand{\pushright}[1]{\ifmeasuring@#1\else\omit\hfill$\displaystyle#1$\fi\ignorespaces}
\newcommand{\pushleft}[1]{\ifmeasuring@#1\else\omit$\displaystyle#1$\hfill\fi\ignorespaces}
\makeatother

\newcommand{\pla}{\mathbb P_\lambda}

\newcommand{\Fcal}{\mathcal{F}}
\newcommand{\Acal}{\mathcal{A}}
\newcommand{\Gcal}{\mathcal{G}}
\newcommand{\Hcal}{\mathcal{H}}

\newcommand{\E}{\mathbb E}

\newcommand{\R}{\mathbb R}

\newcommand{\Rd}{\mathbb R^d}
\newcommand{\Z}{\mathbb Z}

\newcommand{\N}{\mathbb N}

\newcommand{\dd}{\mathrm{d}} 
\newcommand{\C}{\mathscr {C}}

\newcommand{\thinn}[1]{\langle #1 \rangle}
\newcommand{\piv}[1]{\textsf {Piv}(#1)}

\DeclareMathOperator*{\esssup}{ess\,sup}

\newcommand{\LandauBigO}[1]{\mathcal{O}\left(#1\right)}

\newcommand{\psibar}{\overline{\psi}}

\newcommand{\conn}[3]{#1 \longleftrightarrow #2\textrm { in } #3}
\newcommand{\adja}[3]{#1 \sim #2\textrm { in } #3}
\newcommand{\notadja}[3]{#1 \not\sim #2\textrm { in } #3}

\newcommand{\dconn}[3]{#1 \Longleftrightarrow #2\textrm { in } #3}

\newcommand{\xconn}[4]{#1 \xleftrightarrow{\,\,#4\,\,} #2\textrm { in } #3}

\newcommand{\sqconn}[3]{#1 \leftrightsquigarrow #2\textrm { in } #3 }
\newcommand{\sqarrow}{\leftrightsquigarrow}

\newcommand{\tlam}{\tau_\lambda}
\newcommand{\tlamC}{\tau_{\lambda_c}}
\newcommand{\tlamo}{\tau_\lambda^\circ}

\newcommand{\ftlam}{\widehat\tau_\lambda}

\newcommand{\Lacelam}{\Pi_\lambda}
\newcommand{\fLacelam}{\widehat\Pi_\lambda}
\newcommand{\LacelamC}{\Pi_{\lambda_c}}
\newcommand{\fLacelamC}{\widehat\Pi_{\lambda_c}}

\newcommand{\fgmu}{\widehat G_{\mulam}}

\newcommand{\Id}{\mathds 1}

\DeclarePairedDelimiter\abs{\lvert}{\rvert}
\DeclarePairedDelimiter\norm{\lVert}{\rVert}

\DeclarePairedDelimiter\ceil{\lceil}{\rceil}

\DeclarePairedDelimiterX{\inner}[2]{\langle}{\rangle}{#1, #2}

\newcommand{\phiint}{q_\connf}

\newcommand{\mulam}{{\mu_\lambda}}
\newcommand{\orig}{\mathbf{0}}
\newcommand{\e}{\text{e}}

\newcommand{\connf}{\varphi}
\newcommand{\Exconnf}[1]{\connf^{[#1]}}

\newcommand{\fconnf}{\widehat\connf}

\definecolor{darkorange}{RGB}{255,165,0}
\definecolor{altviolet}{RGB}{139,0,139}
\definecolor{turquoise}{RGB}{64,224,208}
\definecolor{lblue}{RGB}{173,216,230}
\definecolor{violet}{RGB}{238,130,238}
\definecolor{darkgreen}{RGB}{0,100,0}
\definecolor{lgreen}{RGB}{144,238,144}

\tikzset{cross/.style={cross out, draw=black, minimum size=2*(#1-\pgflinewidth), inner sep=0pt, outer sep=0pt},
cross/.default={1pt}}

\newcommand{\IntegralDot}{\raisebox{1pt}{\tikz{\filldraw (0,0) circle (2pt)}}}

\newcommand{\OriginDot}{\raisebox{1pt}{\tikz{\draw (0,0) circle (2pt)}}}

\newcommand{\connfline}{\raisebox{2pt}{\tikz{\draw (0,0) -- (1,0);}}}

\newcommand{\tlamline}{\raisebox{2pt}{\tikz{\draw[green] (0,0) -- (1,0);}}}

\newcommand{\Notconnfline}{\raisebox{2pt}{\tikz{\draw[red] (0,0) -- (1,0);}}}

\newcommand{\loopthree}{\raisebox{-8pt}{
    \begin{tikzpicture}[scale=0.5]
        \filldraw (1,0.6) circle (3pt);
        \filldraw (1,-0.6) circle (3pt);
        \draw (0,0) -- (1,0.6) -- (1,-0.6) -- (0,0);
        \filldraw[fill=white] (0,0) circle (3pt);
    \end{tikzpicture}
}}

\newcommand{\loopfour}{\raisebox{-8pt}{
    \begin{tikzpicture}[scale=0.5]
        \filldraw (1,0.6) circle (3pt);
        \filldraw (1,-0.6) circle (3pt);
        \filldraw (2,0) circle (3pt);
        \draw (0,0) -- (1,0.6) -- (2,0) -- (1,-0.6) -- (0,0);
        \filldraw[fill=white] (0,0) circle (3pt);
    \end{tikzpicture}
}}

\newcommand{\loopfive}{\raisebox{-8pt}{
    \begin{tikzpicture}[scale=0.5]
        \filldraw (1,0.6) circle (3pt);
        \filldraw (1,-0.6) circle (3pt);
        \filldraw (2,0.6) circle (3pt);
        \filldraw (2,-0.6) circle (3pt);
        \draw (0,0) -- (1,0.6) -- (2,0.6) -- (2,-0.6) -- (1,-0.6) -- (0,0);
        \filldraw[fill=white] (0,0) circle (3pt);
    \end{tikzpicture}
}}

\newcommand{\loopsix}{\raisebox{-8pt}{
    \begin{tikzpicture}[scale=0.5]
        \filldraw (1,0.6) circle (3pt);
        \filldraw (1,-0.6) circle (3pt);
        \filldraw (2,0.6) circle (3pt);
        \filldraw (2,-0.6) circle (3pt);
        \filldraw (3,0) circle (3pt);
        \draw (0,0) -- (1,0.6) -- (2,0.6) -- (3,0) -- (2,-0.6) -- (1,-0.6) -- (0,0);
        \filldraw[fill=white] (0,0) circle (3pt);
    \end{tikzpicture}
}}

\newcommand{\loopthreeempty}{\raisebox{-8pt}{
    \begin{tikzpicture}[scale=0.5]
        \draw (0,0) -- (1,0.6) -- (1,-0.6) -- (0,0);
        \filldraw[fill=black] (1,-0.6) circle (3pt);
        \filldraw[fill=black] (1,0.6) circle (3pt);
        \filldraw[fill=white] (0,0) circle (3pt);
    \end{tikzpicture}
}}

\newcommand{\loopfourempty}{\raisebox{-8pt}{
    \begin{tikzpicture}[scale=0.5]
        \draw (0,0) -- (1,0.6) -- (2,0) -- (1,-0.6) -- (0,0);
        \filldraw[fill=black] (1,0.6) circle (3pt);
        \filldraw[fill=black] (1,-0.6) circle (3pt);
        \filldraw[fill=black] (2,0) circle (3pt);
        \filldraw[fill=white] (0,0) circle (3pt);
    \end{tikzpicture}
}}

\newcommand{\loopfiveempty}{\raisebox{-8pt}{
    \begin{tikzpicture}[scale=0.5]
        \draw (0,0) -- (1,0.6) -- (2,0.6) -- (2,-0.6) -- (1,-0.6) -- (0,0);
        \filldraw[fill=black] (1,0.6) circle (3pt);
        \filldraw[fill=black] (1,-0.6) circle (3pt);
        \filldraw[fill=black] (2,0.6) circle (3pt);
        \filldraw[fill=black] (2,-0.6) circle (3pt);
        \filldraw[fill=white] (0,0) circle (3pt);
    \end{tikzpicture}
}}

\newcommand{\loopsixempty}{\raisebox{-8pt}{
    \begin{tikzpicture}[scale=0.5]
        \draw (0,0) -- (1,0.6) -- (2,0.6) -- (3,0) -- (2,-0.6) -- (1,-0.6) -- (0,0);
        \filldraw[fill=black] (1,0.6) circle (3pt);
        \filldraw[fill=black] (1,-0.6) circle (3pt);
        \filldraw[fill=black] (2,0.6) circle (3pt);
        \filldraw[fill=black] (2,-0.6) circle (3pt);
        \filldraw[fill=black] (3,0) circle (3pt);
        \filldraw[fill=white] (0,0) circle (3pt);
    \end{tikzpicture}
}}

\newcommand{\fourcrossone}{\raisebox{-8pt}{
    \begin{tikzpicture}[scale=0.5]
        \filldraw (1,0.6) circle (3pt);
        \filldraw (1,-0.6) circle (3pt);
        \filldraw (2,0) circle (3pt);
        \draw (0,0) -- (1,0.6) -- (1,-0.6) -- (0,0);
        \draw (1,0.6) -- (2,0) -- (1,-0.6);
        \filldraw[fill=white] (0,0) circle (3pt);
    \end{tikzpicture}
}}

\newcommand{\fourcrossoneempty}{\raisebox{-8pt}{
    \begin{tikzpicture}[scale=0.5]
        \draw (0,0) -- (1,0.6) -- (1,-0.6) -- (0,0);
        \draw (1,0.6) -- (2,0) -- (1,-0.6);
        \filldraw[fill=black] (1,0.6) circle (3pt);
        \filldraw[fill=black] (1,-0.6) circle (3pt);
        \filldraw[fill=black] (2,0) circle (3pt);
        \filldraw[fill=white] (0,0) circle (3pt);
    \end{tikzpicture}
}}

\newcommand{\fourcrosstwo}{\raisebox{-8pt}{
    \begin{tikzpicture}[scale=0.5]
        \filldraw (1,0.6) circle (3pt);
        \filldraw (1,-0.6) circle (3pt);
        \filldraw (2,0) circle (3pt);
        \draw (0,0) -- (1,0.6) -- (2,0) -- (1,-0.6) -- (0,0);
        \draw (0,0) -- (2,0);
        \filldraw[fill=white] (0,0) circle (3pt);
    \end{tikzpicture}
}}

\newcommand{\phiThreeTwoOne}{\raisebox{-8pt}{
    \begin{tikzpicture}[scale=0.5]
        \filldraw (1,0.6) circle (3pt);
        \filldraw (1,-0.6) circle (3pt);
        \filldraw (2,0.6) circle (3pt);
        \filldraw (2,-0.6) circle (3pt);
        \draw (0,0) -- (1,0.6) -- (1,-0.6) -- (0,0);
        \draw (1,0.6) -- (2,0.6) -- (2,-0.6) -- (1,-0.6);
        \filldraw[fill=white] (0,0) circle (3pt);
    \end{tikzpicture}
}}

\newcommand{\phiTwoTwoTwo}{\raisebox{-8pt}{
    \begin{tikzpicture}[scale=0.5]
        \filldraw (1,0.6) circle (3pt);
        \filldraw (1,-0.6) circle (3pt);
        \filldraw (2,0) circle (3pt);
        \filldraw (1,0) circle (3pt);
        \draw (0,0) -- (1,0.6) -- (1,-0.6) -- (0,0);
        \draw (1,0.6) -- (2,0) -- (1,-0.6);
        \filldraw[fill=white] (0,0) circle (3pt);
    \end{tikzpicture}
}}

\newcommand{\phiThreeTwoOneempty}{\raisebox{-8pt}{
    \begin{tikzpicture}[scale=0.5]
        \draw (0,0) -- (1,0.6) -- (1,-0.6) -- (0,0);
        \draw (1,0.6) -- (2,0.6) -- (2,-0.6) -- (1,-0.6);
        \filldraw[fill=black] (1,0.6) circle (3pt);
        \filldraw[fill=black] (1,-0.6) circle (3pt);
        \filldraw[fill=black] (2,0.6) circle (3pt);
        \filldraw[fill=black] (2,-0.6) circle (3pt);
        \filldraw[fill=white] (0,0) circle (3pt);
    \end{tikzpicture}
}}

\newcommand{\phiTwoTwoTwoempty}{\raisebox{-8pt}{
    \begin{tikzpicture}[scale=0.5]
        \draw (0,0) -- (1,0.6) -- (1,-0.6) -- (0,0);
        \draw (1,0.6) -- (2,0) -- (1,-0.6);
        \filldraw[fill=black] (1,0.6) circle (3pt);
        \filldraw[fill=black] (1,-0.6) circle (3pt);
        \filldraw[fill=black] (2,0) circle (3pt);
        \filldraw[fill=black] (1,0) circle (3pt);
        \filldraw[fill=white] (0,0) circle (3pt);
    \end{tikzpicture}
}}

\newcommand{\FtwoPtOne}{\raisebox{-20pt}{
    \begin{tikzpicture}[scale=1]
        \filldraw (1,-0.7) circle (0pt);
        \draw (0.5,0) -- (1,0.6) -- (1.5,0) -- (1,-0.6) -- (0.5,0);
        \draw[red] (1,0.6) -- (1,-0.6);
        \filldraw[fill=white] (1,-0.6) circle (2pt);
        \filldraw (0.5,0) circle (2pt);
        \filldraw (1.5,0) circle (2pt);
        \filldraw (1,0.6) circle (2pt);
    \end{tikzpicture}
}}

\newcommand{\FtwoPtTwo}{\raisebox{-20pt}{
    \begin{tikzpicture}[scale=1]
        \filldraw (1,-0.7) circle (0pt);
        \draw (0.5,0) -- (1,0.6) -- (1.5,0) -- (1,-0.6) -- (0.5,0);
        \draw (1,0.6) -- (2,0) -- (1,-0.6);
        \draw[red] (1,0.6) -- (1,-0.6);
        \filldraw[fill=white] (1,-0.6) circle (2pt);
        \filldraw (0.5,0) circle (2pt);
        \filldraw (1.5,0) circle (2pt);
        \filldraw (2,0) circle (2pt);
        \filldraw (1,0.6) circle (2pt);
    \end{tikzpicture}
}}

\newcommand{\FthreePtOne}[2]{\raisebox{-20pt}{
    \begin{tikzpicture}[scale=1]
        \filldraw (1,-0.7) circle (0pt);
        \draw (1,0.6) -- (1.5,0) -- (1,-0.6);
        \draw[red] (1,0.6) -- (1,-0.6);
        \filldraw[fill=white] (1,-0.6) circle (2pt) node[left]{$#1$};
        \filldraw (1.5,0) circle (2pt);
        \filldraw[fill=white] (1,0.6) circle (2pt) node[left]{$#2$};
    \end{tikzpicture}
}}

\newcommand{\FthreePtTwo}[2]{\raisebox{-20pt}{
    \begin{tikzpicture}[scale=1]
        \filldraw (1,-0.7) circle (0pt);
        \draw (0.5,0) -- (1,0.6) -- (1.5,0) -- (1,-0.6) -- (0.5,0);
        \draw[red] (1,0.6) -- (1,-0.6);
        \filldraw[fill=white] (1,-0.6) circle (2pt) node[left]{$#1$};
        \filldraw (0.5,0) circle (2pt);
        \filldraw (1.5,0) circle (2pt);
        \filldraw[fill=white] (1,0.6) circle (2pt) node[left]{$#2$};
    \end{tikzpicture}
}}

\newcommand{\FthreePtThree}[3]{\raisebox{-20pt}{
    \begin{tikzpicture}[scale=1]
        \filldraw (1,-0.7) circle (0pt);
        \draw (1,0.6) -- (1.5,0) -- (1,0);
        \draw[red] (1.5,0) -- (1,-0.6);
        \filldraw[fill=white] (1,-0.6) circle (2pt) node[left]{$#1$};
        \filldraw[fill=white] (1,0) circle (2pt) node[left]{$#3$};
        \filldraw (1.5,0) circle (2pt);
        \filldraw[fill=white] (1,0.6) circle (2pt) node[left]{$#2$};
    \end{tikzpicture}
}}

\newcommand{\FthreePtFour}[3]{\raisebox{-20pt}{
    \begin{tikzpicture}[scale=1]
        \filldraw (1,-0.7) circle (0pt);
        \draw (1,0.6) -- (1.5,0) -- (1,0) -- (0.5,0) -- (1,0.6);
        \draw[red] (1.5,0) -- (1,-0.6) -- (0.5,0);
        \filldraw[fill=white] (1,-0.6) circle (2pt) node[left]{$#1$};
        \filldraw[fill=white] (1,0) circle (2pt) node[above]{$#3$};
        \filldraw (1.5,0) circle (2pt);
        \filldraw (0.5,0) circle (2pt);
        \filldraw[fill=white] (1,0.6) circle (2pt) node[left]{$#2$};
    \end{tikzpicture}
}}

\newcommand{\FthreePtFive}[2]{\raisebox{-20pt}{
    \begin{tikzpicture}[scale=1]
        \filldraw (1,-0.7) circle (0pt);
        \draw (1,0.6) -- (1.5,0) -- (1,0) -- (1,-0.6);
        \draw[red] (1.5,0) -- (1,-0.6);
        \draw[red] (1,0) -- (1,0.6);
        \filldraw[fill=white] (1,-0.6) circle (2pt) node[left]{$#1$};
        \filldraw (1,0) circle (2pt);
        \filldraw (1.5,0) circle (2pt);
        \filldraw[fill=white] (1,0.6) circle (2pt) node[left]{$#2$};
    \end{tikzpicture}
}}

\newcommand{\FthreePtSix}[2]{\raisebox{-20pt}{
    \begin{tikzpicture}[scale=1]
        \filldraw (1,-0.7) circle (0pt);
        \draw (1,0.6) -- (1.5,0) -- (1,0) -- (0.5,0) -- (1,0.6);
        \draw (1,0)--(1,-0.6);
        \draw[red] (1.5,0) -- (1,-0.6) -- (0.5,0);
        \draw[red] (1,0) -- (1,0.6);
        \filldraw[fill=white] (1,-0.6) circle (2pt) node[left]{$#1$};
        \filldraw (1,0) circle (2pt);
        \filldraw (1.5,0) circle (2pt);
        \filldraw (0.5,0) circle (2pt);
        \filldraw[fill=white] (1,0.6) circle (2pt) node[left]{$#2$};
    \end{tikzpicture}
}}

\newcommand{\FthreePtSeven}[2]{\raisebox{-20pt}{
    \begin{tikzpicture}[scale=1]
        \filldraw (1,-0.7) circle (0pt);
        \draw (1,-0.6) -- (0.5,0) -- (0.85,0) -- (1,0.6) -- (1.5,0) -- (1.15,0) -- (1,-0.6);
        \draw[red] (1,-0.6) -- (0.85,0);
        \draw[red] (1,-0.6) -- (1.5,0);
        \draw[red] (1,0.6) -- (0.5,0);
        \draw[red] (1,0.6) -- (1.15,0);
        \filldraw[fill=white] (1,-0.6) circle (2pt) node[left]{$#1$};
        \filldraw (1.15,0) circle (2pt);
        \filldraw (1.5,0) circle (2pt);
        \filldraw (0.5,0) circle (2pt);
        \filldraw (0.85,0) circle (2pt);
        \filldraw[fill=white] (1,0.6) circle (2pt) node[left]{$#2$};
    \end{tikzpicture}
}}

\newcommand{\FthreePtEight}[2]{\raisebox{-20pt}{
    \begin{tikzpicture}[scale=1]
        \filldraw (1,-0.7) circle (0pt);
        \draw (1,-0.6) -- (0.5,0) -- (0.85,0) -- (1,0.6) -- (1.85,0) -- (1.5,0) -- (1,-0.6);
        \draw (1,0.6) -- (1.15,0) -- (1.5,0);
        \draw[red] (1,-0.6) -- (0.85,0);
        \draw[red] (1,-0.6) -- (1.15,0);
        \draw[red] (1,-0.6) -- (1.85,0);
        \draw[red] (1,0.6) -- (0.5,0);
        \draw[red] (1,0.6) -- (1.5,0);
        \filldraw[fill=white] (1,-0.6) circle (2pt) node[left]{$#1$};
        \filldraw (1.15,0) circle (2pt);
        \filldraw (1.5,0) circle (2pt);
        \filldraw (0.5,0) circle (2pt);
        \filldraw (0.85,0) circle (2pt);
        \filldraw (1.85,0) circle (2pt);
        \filldraw[fill=white] (1,0.6) circle (2pt) node[left]{$#2$};
    \end{tikzpicture}
}}

\newcommand{\FthreePtNine}[2]{\raisebox{-20pt}{
    \begin{tikzpicture}[scale=1]
        \filldraw (1,-0.7) circle (0pt);
        \draw (1,-0.6) -- (0.5,0) -- (0.15,0) -- (1,0.6) -- (1.85,0) -- (1.5,0) -- (1,-0.6);
        \draw (1,0.6) -- (1.15,0) -- (1.5,0);
        \draw (1,0.6) -- (0.85,0) -- (0.5,0);
        \draw[red] (1,-0.6) -- (0.85,0);
        \draw[red] (1,-0.6) -- (1.15,0);
        \draw[red] (1,-0.6) -- (1.85,0);
        \draw[red] (1,-0.6) -- (0.15,0);
        \draw[red] (1,0.6) -- (0.5,0);
        \draw[red] (1,0.6) -- (1.5,0);
        \filldraw[fill=white] (1,-0.6) circle (2pt) node[left]{$#1$};
        \filldraw (1.15,0) circle (2pt);
        \filldraw (1.5,0) circle (2pt);
        \filldraw (0.5,0) circle (2pt);
        \filldraw (0.15,0) circle (2pt);
        \filldraw (0.85,0) circle (2pt);
        \filldraw (1.85,0) circle (2pt);
        \filldraw[fill=white] (1,0.6) circle (2pt) node[left]{$#2$};
    \end{tikzpicture}
}}

\newcommand{\FthreePtTen}{\raisebox{-20pt}{
    \begin{tikzpicture}[scale=1]
        \filldraw (1,-0.7) circle (0pt);
        \draw (1,-0.6) -- (0.5,0) -- (1,0.6) -- (1.5,0) -- (1.25,0) -- (1,-0.6);
        \draw[red] (1,-0.6) -- (1,0.6);
        \draw[red] (1,-0.6) -- (1.5,0);
        \draw[red] (1,0.6) -- (1.25,0);
        \filldraw[fill=white] (1,-0.6) circle (2pt);
        \filldraw (1.25,0) circle (2pt);
        \filldraw (1.5,0) circle (2pt);
        \filldraw (0.5,0) circle (2pt);
        \filldraw (1,0.6) circle (2pt);
    \end{tikzpicture}
}}

\newcommand{\GtwoPtOne}{\raisebox{-18pt}{
    \begin{tikzpicture}[scale=1]
        \filldraw (1,-0.7) circle (0pt);
        \draw (0,0) -- (1,0.6) -- (2,0) -- (1,-0.6) -- (0,0);
        \draw[red] (1,0.6) -- (1,-0.6);
        \filldraw (1,-0.6) circle (2pt);
        \filldraw[fill=white] (0,0) circle (2pt);
        \filldraw (2,0) circle (2pt);
        \filldraw (1,0.6) circle (2pt);
    \end{tikzpicture}
}}

\newcommand{\GtwoPtTwo}{\raisebox{-18pt}{
    \begin{tikzpicture}[scale=1]
        \filldraw (1,-0.7) circle (0pt);
        \draw (0,0) -- (1,0.6) -- (1.5,0) -- (1,-0.6) -- (0,0);
        \draw (1,0.6) -- (2,0) -- (1,-0.6);
        \draw[red] (1,0.6) -- (1,-0.6);
        \filldraw (1,-0.6) circle (2pt);
        \filldraw[fill=white] (0,0) circle (2pt);
        \filldraw (1.5,0) circle (2pt);
        \filldraw (1,0.6) circle (2pt);
        \filldraw (2,0) circle (2pt);
    \end{tikzpicture}
}}

\newcommand{\GthreePtOne}{\raisebox{-18pt}{
    \begin{tikzpicture}[scale=1]
        \filldraw (1,-0.7) circle (0pt);
        \draw (0,0) -- (1,0.6) -- (2,0) -- (1,-0.6) -- (0,0);
        \draw[red] (1,0.6) -- (1,-0.6);
        \filldraw[fill=white] (1,-0.6) circle (2pt);
        \filldraw (0,0) circle (2pt);
        \filldraw (2,0) circle (2pt);
        \filldraw (1,0.6) circle (2pt);
    \end{tikzpicture}
}}

\newcommand{\GthreePtTwo}{\raisebox{-18pt}{
    \begin{tikzpicture}[scale=1]
        \filldraw (1,-0.7) circle (0pt);
        \draw (0,0) -- (1,0.6) -- (1.5,0) -- (1,-0.6) -- (0,0);
        \draw (1,0.6) -- (2,0) -- (1,-0.6);
        \draw[red] (1,0.6) -- (1,-0.6);
        \filldraw[fill=white] (1,-0.6) circle (2pt);
        \filldraw (0,0) circle (2pt);
        \filldraw (1.5,0) circle (2pt);
        \filldraw (1,0.6) circle (2pt);
        \filldraw (2,0) circle (2pt);
    \end{tikzpicture}
}}

\newcommand{\GfourPtOne}{\raisebox{-20pt}{
    \begin{tikzpicture}[scale=1]
        \filldraw (1,-0.7) circle (0pt);
        \draw (1,-0.6) -- (0.5,0) -- (1,0.6) -- (1.5,0) -- (1.25,0) -- (1,-0.6);
        \draw[red] (1,-0.6) -- (1,0.6);
        \draw[red] (1,-0.6) -- (1.5,0);
        \draw[red] (1,0.6) -- (1.25,0);
        \filldraw (1,-0.6) circle (2pt);
        \filldraw (1.25,0) circle (2pt);
        \filldraw (1.5,0) circle (2pt);
        \filldraw [fill=white] (0.5,0) circle (2pt);
        \filldraw (1,0.6) circle (2pt);
    \end{tikzpicture}
}}

\newcommand{\GfourPtTwo}{\raisebox{-20pt}{
    \begin{tikzpicture}[scale=1]
        \filldraw (1,-0.7) circle (0pt);
        \draw (1,-0.6) -- (0.5,0) -- (1,0.6) -- (1.95,0) -- (1.6,0) -- (1,-0.6);
        \draw (1,0.6) -- (1.25,0) -- (1.6,0);
        \draw[red] (1,-0.6) -- (1.25,0);
        \draw[red] (1,-0.6) -- (1.95,0);
        \draw[red] (1,0.6) -- (1.6,0);
        \draw[red] (1,-0.6) -- (1,0.6);
        \filldraw (1,-0.6) circle (2pt);
        \filldraw (1.25,0) circle (2pt);
        \filldraw (1.6,0) circle (2pt);
        \filldraw[fill=white] (0.5,0) circle (2pt);
        \filldraw (1.95,0) circle (2pt);
        \filldraw (1,0.6) circle (2pt);
    \end{tikzpicture}
}}

\newcommand{\GfourPtThree}{\raisebox{-20pt}{
    \begin{tikzpicture}[scale=1]
        \filldraw (1,-0.7) circle (0pt);
        \draw (1,-0.6) -- (0.5,0) -- (0.85,0) -- (1,0.6) -- (1.5,0) -- (1.15,0) -- (1,-0.6);
        \draw (1,-0.6) to [out=180,in=305] (0.15,0) to [out=55,in=180] (1,0.6);
        \draw[red] (1,-0.6) -- (0.85,0);
        \draw[red] (1,-0.6) -- (1.5,0);
        \draw[red] (1,0.6) -- (0.5,0);
        \draw[red] (1,0.6) -- (1.15,0);
        \draw[red] (1,-0.6) -- (1,0.6);
        \filldraw (1,-0.6) circle (2pt);
        \filldraw (1.15,0) circle (2pt);
        \filldraw (1.5,0) circle (2pt);
        \filldraw (0.5,0) circle (2pt);
        \filldraw (0.85,0) circle (2pt);
        \filldraw[fill=white] (0.15,0) circle (2pt);
        \filldraw (1,0.6) circle (2pt);
    \end{tikzpicture}
}}

\newcommand{\GfourPtFour}{\raisebox{-20pt}{
    \begin{tikzpicture}[scale=1]
        \filldraw (1,-0.7) circle (0pt);
        \draw (1,-0.6) -- (0.5,0) -- (0.85,0) -- (1,0.6) -- (1.85,0) -- (1.5,0) -- (1,-0.6);
        \draw (1,0.6) -- (1.15,0) -- (1.5,0);
        \draw (1,-0.6) to [out=180,in=305] (0.15,0) to [out=55,in=180] (1,0.6);
        \draw[red] (1,-0.6) -- (0.85,0);
        \draw[red] (1,-0.6) -- (1.15,0);
        \draw[red] (1,-0.6) -- (1.85,0);
        \draw[red] (1,0.6) -- (0.5,0);
        \draw[red] (1,0.6) -- (1.5,0);
        \draw[red] (1,0.6) -- (1,-0.6);
        \filldraw (1,-0.6) circle (2pt);
        \filldraw (1.15,0) circle (2pt);
        \filldraw (1.5,0) circle (2pt);
        \filldraw (0.5,0) circle (2pt);
        \filldraw (0.85,0) circle (2pt);
        \filldraw (1.85,0) circle (2pt);
        \filldraw[fill=white] (0.15,0) circle (2pt);
        \filldraw (1,0.6) circle (2pt);
    \end{tikzpicture}
}}

\newcommand{\GfourPtFive}{\raisebox{-20pt}{
    \begin{tikzpicture}[scale=1]
        \filldraw (1,-0.7) circle (0pt);
        \draw (1,-0.6) -- (0.5,0) -- (0.15,0) -- (1,0.6) -- (1.85,0) -- (1.5,0) -- (1,-0.6);
        \draw (1,-0.6) to [out=180,in=315] (-0.2,0) to [out=45,in=180] (1,0.6);
        \draw (1,0.6) -- (1.15,0) -- (1.5,0);
        \draw (1,0.6) -- (0.85,0) -- (0.5,0);
        \draw[red] (1,-0.6) -- (0.85,0);
        \draw[red] (1,-0.6) -- (1.15,0);
        \draw[red] (1,-0.6) -- (1.85,0);
        \draw[red] (1,-0.6) -- (0.15,0);
        \draw[red] (1,0.6) -- (0.5,0);
        \draw[red] (1,0.6) -- (1.5,0);
        \draw[red] (1,-0.6) -- (1,0.6);
        \filldraw (1,-0.6) circle (2pt);
        \filldraw (1.15,0) circle (2pt);
        \filldraw (1.5,0) circle (2pt);
        \filldraw (0.5,0) circle (2pt);
        \filldraw (0.15,0) circle (2pt);
        \filldraw (0.85,0) circle (2pt);
        \filldraw (1.85,0) circle (2pt);
        \filldraw[fill=white] (-0.2,0) circle (2pt);
        \filldraw (1,0.6) circle (2pt);
    \end{tikzpicture}
}}

\newcommand{\GfivePtOne}{\raisebox{-18pt}{
    \begin{tikzpicture}[scale=1]
        \filldraw (1,-0.7) circle (0pt);
        \draw (0,-0.6) -- (1,-0.6) -- (1.5,0) -- (0.5,0.6) -- (-0.5,0) -- (0,-0.6);
        \draw[red] (0,-0.6) -- (0.5,0.6) -- (1,-0.6);
        \filldraw[fill=white] (0,-0.6) circle (2pt);
        \filldraw (1,-0.6) circle (2pt);
        \filldraw (1.5,0) circle (2pt);
        \filldraw (0.5,0.6) circle (2pt);
        \filldraw (-0.5,0) circle (2pt);
    \end{tikzpicture}
}}

\newcommand{\GfivePtTwo}{\raisebox{-18pt}{
    \begin{tikzpicture}[scale=1]
        \filldraw (1,-0.7) circle (0pt);
        \draw (0,-0.6) -- (1,-0.6) -- (1.5,0) -- (0.5,0.6) -- (-0.5,0) -- (0,-0.6);
        \draw (0,-0.6) -- (-0.1,0) -- (0.5,0.6);
        \draw[red] (0,-0.6) -- (0.5,0.6) -- (1,-0.6);
        \filldraw[fill=white] (0,-0.6) circle (2pt);
        \filldraw (1,-0.6) circle (2pt);
        \filldraw (1.5,0) circle (2pt);
        \filldraw (0.5,0.6) circle (2pt);
        \filldraw (-0.1,0) circle (2pt);
        \filldraw (-0.5,0) circle (2pt);
    \end{tikzpicture}
}}

\newcommand{\GfivePtThree}{\raisebox{-18pt}{
    \begin{tikzpicture}[scale=1]
        \filldraw (1,-0.7) circle (0pt);
        \draw (0,-0.6) -- (1,-0.6) -- (1.5,0) -- (0.5,0.6) -- (-0.5,0) -- (0,-0.6);
        \draw (1,-0.6) -- (1.1,0) -- (0.5,0.6);
        \draw[red] (0,-0.6) -- (0.5,0.6) -- (1,-0.6);
        \filldraw[fill=white] (0,-0.6) circle (2pt);
        \filldraw (1,-0.6) circle (2pt);
        \filldraw (1.5,0) circle (2pt);
        \filldraw (0.5,0.6) circle (2pt);
        \filldraw (-0.5,0) circle (2pt);
        \filldraw (1.1,0) circle (2pt);
    \end{tikzpicture}
}}

\newcommand{\GfivePtFour}{\raisebox{-18pt}{
    \begin{tikzpicture}[scale=1]
        \filldraw (1,-0.7) circle (0pt);
        \draw (0,-0.6) -- (1,-0.6) -- (1.5,0) -- (0.5,0.6) -- (-0.5,0) -- (0,-0.6);
        \draw (0,-0.6) -- (-0.1,0) -- (0.5,0.6);
        \draw (1,-0.6) -- (1.1,0) -- (0.5,0.6);
        \draw[red] (0,-0.6) -- (0.5,0.6) -- (1,-0.6);
        \filldraw[fill=white] (0,-0.6) circle (2pt);
        \filldraw (1,-0.6) circle (2pt);
        \filldraw (1.5,0) circle (2pt);
        \filldraw (0.5,0.6) circle (2pt);
        \filldraw (-0.1,0) circle (2pt);
        \filldraw (-0.5,0) circle (2pt);
        \filldraw (1.1,0) circle (2pt);
    \end{tikzpicture}
}}

\newcommand{\GsixPtOne}{\raisebox{-20pt}{
    \begin{tikzpicture}[scale=1]
        \filldraw (1,-0.7) circle (0pt);
        \draw (1,-0.6) -- (0.5,0) -- (1,0.6) -- (1.5,0) -- (1.25,0) -- (1,-0.6);
        \draw[red] (1,-0.6) -- (1,0.6);
        \draw[red] (1,-0.6) -- (1.5,0);
        \draw[red] (1,0.6) -- (1.25,0);
        \filldraw[fill=white] (1,-0.6) circle (2pt);
        \filldraw (1.25,0) circle (2pt);
        \filldraw (1.5,0) circle (2pt);
        \filldraw (0.5,0) circle (2pt);
        \filldraw (1,0.6) circle (2pt);
    \end{tikzpicture}
}}

\newcommand{\GsixPtTwo}{\raisebox{-20pt}{
    \begin{tikzpicture}[scale=1]
        \filldraw (1,-0.7) circle (0pt);
        \draw (1,-0.6) -- (0.5,0) -- (1,0.6) -- (1.95,0) -- (1.6,0) -- (1,-0.6);
        \draw (1,0.6) -- (1.25,0) -- (1.6,0);
        \draw[red] (1,-0.6) -- (1.25,0);
        \draw[red] (1,-0.6) -- (1.95,0);
        \draw[red] (1,0.6) -- (1.6,0);
        \draw[red] (1,-0.6) -- (1,0.6);
        \filldraw[fill=white] (1,-0.6) circle (2pt);
        \filldraw (1.25,0) circle (2pt);
        \filldraw (1.6,0) circle (2pt);
        \filldraw (0.5,0) circle (2pt);
        \filldraw (1.95,0) circle (2pt);
        \filldraw (1,0.6) circle (2pt);
    \end{tikzpicture}
}}

\newcommand{\GsixPtThree}{\raisebox{-20pt}{
    \begin{tikzpicture}[scale=1]
        \filldraw (1,-0.7) circle (0pt);
        \draw (1,-0.6) -- (0.5,0) -- (0.85,0) -- (1,0.6) -- (1.5,0) -- (1.15,0) -- (1,-0.6);
        \draw (1,-0.6) to [out=180,in=305] (0.15,0) to [out=55,in=180] (1,0.6);
        \draw[red] (1,-0.6) -- (0.85,0);
        \draw[red] (1,-0.6) -- (1.5,0);
        \draw[red] (1,0.6) -- (0.5,0);
        \draw[red] (1,0.6) -- (1.15,0);
        \draw[red] (1,-0.6) -- (1,0.6);
        \filldraw[fill=white] (1,-0.6) circle (2pt);
        \filldraw (1.15,0) circle (2pt);
        \filldraw (1.5,0) circle (2pt);
        \filldraw (0.5,0) circle (2pt);
        \filldraw (0.85,0) circle (2pt);
        \filldraw (0.15,0) circle (2pt);
        \filldraw (1,0.6) circle (2pt);
    \end{tikzpicture}
}}

\newcommand{\GsixPtFour}{\raisebox{-20pt}{
    \begin{tikzpicture}[scale=1]
        \filldraw (1,-0.7) circle (0pt);
        \draw (1,-0.6) -- (0.5,0) -- (0.85,0) -- (1,0.6) -- (1.85,0) -- (1.5,0) -- (1,-0.6);
        \draw (1,0.6) -- (1.15,0) -- (1.5,0);
        \draw (1,-0.6) to [out=180,in=305] (0.15,0) to [out=55,in=180] (1,0.6);
        \draw[red] (1,-0.6) -- (0.85,0);
        \draw[red] (1,-0.6) -- (1.15,0);
        \draw[red] (1,-0.6) -- (1.85,0);
        \draw[red] (1,0.6) -- (0.5,0);
        \draw[red] (1,0.6) -- (1.5,0);
        \draw[red] (1,0.6) -- (1,-0.6);
        \filldraw[fill=white] (1,-0.6) circle (2pt);
        \filldraw (1.15,0) circle (2pt);
        \filldraw (1.5,0) circle (2pt);
        \filldraw (0.5,0) circle (2pt);
        \filldraw (0.85,0) circle (2pt);
        \filldraw (1.85,0) circle (2pt);
        \filldraw (0.15,0) circle (2pt);
        \filldraw (1,0.6) circle (2pt);
    \end{tikzpicture}
}}

\newcommand{\GsixPtFive}{\raisebox{-20pt}{
    \begin{tikzpicture}[scale=1]
        \filldraw (1,-0.7) circle (0pt);
        \draw (1,-0.6) -- (0.5,0) -- (0.15,0) -- (1,0.6) -- (1.85,0) -- (1.5,0) -- (1,-0.6);
        \draw (1,-0.6) to [out=180,in=315] (-0.2,0) to [out=45,in=180] (1,0.6);
        \draw (1,0.6) -- (1.15,0) -- (1.5,0);
        \draw (1,0.6) -- (0.85,0) -- (0.5,0);
        \draw[red] (1,-0.6) -- (0.85,0);
        \draw[red] (1,-0.6) -- (1.15,0);
        \draw[red] (1,-0.6) -- (1.85,0);
        \draw[red] (1,-0.6) -- (0.15,0);
        \draw[red] (1,0.6) -- (0.5,0);
        \draw[red] (1,0.6) -- (1.5,0);
        \draw[red] (1,-0.6) -- (1,0.6);
        \filldraw[fill=white] (1,-0.6) circle (2pt);
        \filldraw (1.15,0) circle (2pt);
        \filldraw (1.5,0) circle (2pt);
        \filldraw (0.5,0) circle (2pt);
        \filldraw (0.15,0) circle (2pt);
        \filldraw (0.85,0) circle (2pt);
        \filldraw (1.85,0) circle (2pt);
        \filldraw (-0.2,0) circle (2pt);
        \filldraw (1,0.6) circle (2pt);
    \end{tikzpicture}
}}

\newcommand{\PiOnePtOne}{\raisebox{-17pt}{
    \begin{tikzpicture}[scale=1]
        \filldraw (1,-0.7) circle (0pt);
        \draw[green] (0,0) -- (1,0.6) -- (1,-0.6) -- (0,0);
        \draw[green] (1,0.6) -- (2,0.6) -- (2,-0.6) -- (1,-0.6);
        \draw[green] (2,0.6) -- (3,0) -- (2,-0.6);
        \draw (1.5,-0.6) circle (0pt) node[above]{$\circ$};
        \draw (0.5,-0.3) circle (0pt) node[rotate=-31,above]{\tiny $\geq 2$};
        \draw (2.5,-0.3) circle (0pt) node[rotate=31,above]{\tiny $\geq 2$};
        \filldraw[fill=white] (0,0) circle (2pt);
        \filldraw (1,0.6) circle (2pt);
        \filldraw (1,-0.6) circle (2pt);
        \filldraw (2,0.6) circle (2pt);
        \filldraw (2,-0.6) circle (2pt);
        \filldraw (3,0) circle (2pt);
    \end{tikzpicture}
}}

\newcommand{\PiOnePtTwo}{\raisebox{-17pt}{
    \begin{tikzpicture}[scale=1]
        \filldraw (1,-0.7) circle (0pt);
        \draw[green] (0,0) -- (1,0.6) -- (1,-0.6) -- (0,0);
        \draw[green] (1,0.6) -- (2,0) -- (1,-0.6);
        \draw (0.5,-0.3) circle (0pt) node[rotate=-31,above]{\tiny $\geq 2$};
        \filldraw[fill=white] (0,0) circle (2pt);
        \filldraw (1,0.6) circle (2pt);
        \filldraw (1,-0.6) circle (2pt);
        \filldraw (2,0) circle (2pt);
    \end{tikzpicture}
}}

\newcommand{\PiOnePtThree}{\raisebox{-17pt}{
    \begin{tikzpicture}[scale=1]
        \filldraw (1,-0.7) circle (0pt);
        \draw[green] (0,0) -- (1,0.6) -- (1,-0.6) -- (0,0);
        \draw[green] (1,0.6) -- (2,0.6) -- (2,-0.6) -- (1,-0.6);
        \draw[green] (2,0.6) -- (3,0) -- (2,-0.6);
        \draw (1.5,-0.6) circle (0pt) node[above]{$\circ$};
        \draw (1,0) circle (0pt) node[rotate=90,above]{\tiny $\geq 2$};
        \draw (2.5,-0.3) circle (0pt) node[rotate=31,above]{\tiny $\geq 2$};
        \filldraw (0,0) circle (2pt);
        \filldraw[fill=white] (1,0.6) circle (2pt);
        \filldraw (1,-0.6) circle (2pt);
        \filldraw (2,0.6) circle (2pt);
        \filldraw (2,-0.6) circle (2pt);
        \filldraw (3,0) circle (2pt);
    \end{tikzpicture}
}}

\newcommand{\PiOnePtFour}{\raisebox{-17pt}{
    \begin{tikzpicture}[scale=1]
        \filldraw (1,-0.7) circle (0pt);
        \draw[green] (0,0) -- (1,0.6) -- (1,-0.6) -- (0,0);
        \draw[green] (1,0.6) -- (2,0) -- (1,-0.6);
        \draw (1,0) circle (0pt) node[rotate=90,above]{\tiny $\geq 2$};
        \filldraw (0,0) circle (2pt);
        \filldraw[fill=white] (1,0.6) circle (2pt);
        \filldraw (1,-0.6) circle (2pt);
        \filldraw (2,0) circle (2pt);
    \end{tikzpicture}
}}

\newcommand{\PiOnePtFive}{\raisebox{-17pt}{
    \begin{tikzpicture}[scale=1]
        \filldraw (1,-0.7) circle (0pt);
        \draw (1,0.6) -- (1,-0.6);
        \draw[green] (1,0.6) -- (2,0.6) -- (2,-0.6) -- (1,-0.6);
        \draw[green] (2,0.6) -- (3,0) -- (2,-0.6);
        \draw (1.5,-0.6) circle (0pt) node[above]{$\circ$};
        \draw (2.5,-0.3) circle (0pt) node[rotate=31,above]{\tiny $\geq 2$};
        \filldraw[fill=white] (1,0.6) circle (2pt);
        \filldraw (1,-0.6) circle (2pt);
        \filldraw (2,0.6) circle (2pt);
        \filldraw (2,-0.6) circle (2pt);
        \filldraw (3,0) circle (2pt);
    \end{tikzpicture}
}}

\newcommand{\PiOnePtSix}{\raisebox{-17pt}{
    \begin{tikzpicture}[scale=1]
        \filldraw (1,-0.7) circle (0pt);
        \draw (1,0.6) -- (1,-0.6);
        \draw[green] (1,0.6) -- (2,0) -- (1,-0.6);
        \filldraw[fill=white] (1,0.6) circle (2pt);
        \filldraw (1,-0.6) circle (2pt);
        \filldraw (2,0) circle (2pt);
    \end{tikzpicture}
}}

\numberwithin{equation}{section}
\allowdisplaybreaks

\title{Expansion of the Critical Intensity for the Random Connection Model}
\author{Matthew Dickson\footnote{University of British Columbia, Department of Mathematics, Vancouver, BC, Canada, V6T 1Z2; Email: dickson@math.ubc.ca; \includegraphics[height=1em]{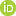}~https://orcid.org/0000-0002-8629-4796} \and Markus Heydenreich\footnote{Universität Augsburg, Institut für Mathematik, Universitätsstr.\ 2, 86135 Augsburg, Germany; Email: markus.heydenreich@uni-a.de; \includegraphics[height=1em]{ORCIDiD_icon16x16.png}~https://orcid.org/0000-0002-3749-7431}
}
 
\date{}

\begin{document}
\maketitle

\vspace{-1em}

{\centering{ \today}\par}

\vskip-3em

\begin{abstract}
We derive an asymptotic expansion for the critical percolation density of the random connection model as the dimension of the encapsulating space
tends to infinity. 
We calculate rigorously the first expansion terms for the Gilbert disk model, the hyper-cubic model, the Gaussian connection kernel, and a coordinate-wise Cauchy kernel. 
\end{abstract}


\smallskip


{\footnotesize
}

\section{Introduction}

    \subsection{Motivation} We study percolative systems, and address the question: \emph{What is the value of the critical percolation threshold?}
    A specific answer is only possible in very exceptional cases.  
    We are pursuing a different route instead, namely an asymptotic expression of the critical threshold as a function of the dimension $d$ of the encapsulating space in the $d\to\infty$ limit. This has been solved for percolation on the hypercubic lattice $\Z^d$: for \emph{bond} percolation on the hypercubic lattice it is known that 
    \begin{equation}\label{eq:pc_bond}
        p_c^\text{bond}(\mathbb Z^d)=\frac1{2d}+\frac{1}{\left(2d\right)^2} +\frac{7}{2}\frac{1}{\left(2d\right)^3} +\LandauBigO{\frac1{d^4}}\quad \text{as $d\to\infty$,}
    \end{equation} 
    cf.\ \cite{HarSla95,HofSla05}, whereas for hypercubic \emph{site} percolation on $\Z^d$ it is 
    \begin{equation}\label{eq:pc_site}
        p_c^\text{site}(\mathbb Z^d)= \frac1{2d}+\frac{5}{2}\frac{1}{\left(2d\right)^2} +\frac{31}{4}\frac{1}{\left(2d\right)^3} +\LandauBigO{\frac1{d^4}}\quad \text{as $d\to\infty$,}
    \end{equation}
    cf.\ \cite{HeyMat20}. 
    Mertens and Moore \cite{mertens2018series} use involved numerical enumeration to identify a few more terms (without a rigorous bound on the error). 
    In the present work, we address a corresponding question for continuum percolation. Interestingly, our analysis establishes an exponentially decaying series rather then an algebraic decay as on lattices. We shall discuss this point further in the discussion section. 

    
    \subsection{The Model} 
    To this end, we are considering the random connection model (henceforth abbreviated RCM), a spatial random graph model whose points are given as a homogeneous Poison process $\eta$ on $\R^d$ with intensity measure $\lambda\operatorname{Leb}$, and we refer to $\lambda>0$ as the \emph{intensity} of the model. Each pair of vertices $x,y$ in the support of $\eta$ are connected independently with probability 
    $\connf(x-y)$, where 
    \[\connf\colon \R^d\to[0,1]\]
    is integrable and symmetric (i.e. $\connf(x)=\connf(-x)$ for all $x\in\Rd$). 
    The classical example is the \emph{Gilbert disk model} \cite{Gil61} or hyper-sphere random connection model with 
    \[ \connf(x) = \Id_{\left\{\abs*{x}< R\right\}}\] 
    for some $R>0$: two vertices are connected whenever their (Euclidean) distance is at most $R$. 

    We are interested in the percolation phase transition of the RCM, that is, the critical intensity $\lambda_c$ given as the infimum of those values of $\lambda$ such that the resulting random graph has an infinite connected component: 
    \[ \lambda_c=\inf\{\lambda \mid \text{the RCM with intensity $\lambda$ has an infinite component}\}.\]
    See \cite[Section 2]{HeyHofLasMat19} for a more formal definition. 

    Penrose \cite{Pen91} uses the `method of generations' to show that for all dimensions $d\geq 1$ the critical intensity is strictly positive. In particular he derives the lower bound
    \begin{equation}
        \phiint\lambda_c \geq 1,
    \end{equation}
    where $\phiint= \int \connf(x)\dd x$. He also uses a coarse-graining argument to show that for $d\geq 2$ the critical intensity is finite if $\phiint>0$.
    Meester, Penrose and Sarkar \cite{MeePenSar97} prove the $0^{th}$ order asymptotics of $\lambda_c$ for radial non-increasing $\connf$ (with uniform bounds on the variance of the jumps taken by random walk with jump intensity proportional to $\connf$). Specifically they prove that for such models
    \begin{equation}
        \phiint\lambda_c \to 1
    \end{equation}
    as $d\to\infty$. In the present work we significantly expand their result by identifying several additional terms.

\subsection{Results}

We shall now make a couple of assumptions before formulating our main result. Throughout this paper we will denote the convolution of two non-negative functions $f,g\colon \Rd\to \R_{\geq 0}$ to be
    \begin{equation}
        f\star g(x) := \int f(x-u)g(u)\dd u,
    \end{equation}
and $f^{\star n}(x)$ to be the convolution of $n$ copies of $f$. In particular, $f^{\star 1}\equiv f$. We will also denote the Fourier transform of an integrable function $f\colon \Rd\to\R$ by
\begin{equation}
    \widehat{f}(k) := \int\e^{ik\cdot x}f(x)\dd x,
\end{equation}
for all $k\in\Rd$.

Our first set of assumptions are exactly those that allow us to use the results relating to lace expansion arguments.
\begin{assumptionp}{A}
\label{Assumption}
    We require $\connf$ to satisfy the following two properties: 
\begin{enumerate}[label=\textbf{(\ref{Assumption}.\arabic*)}]

\item\label{Assump:DecayBound} There exists a function $g\colon\mathbb N \to \R_{\geq 0}$ with the following three properties. Firstly, that $g(d)\to 0$ as $d\to\infty$. Secondly, that for $m\geq 3$, the $m$-fold convolution $\connf^{\star m}$ of $\connf$ satisfies
\begin{equation}
    \frac{1}{\phiint^{m-1}}\sup_{x\in\Rd}\connf^{\star m}(x) \leq g(d).
\end{equation}
Thirdly, that the Lebesgue volume
\begin{equation}\label{eqn:2-convolutionAssumption}
    \frac{1}{\phiint}\abs*{\left\{x\in\Rd\colon \frac{1}{\phiint} \connf\star\connf(x)  > g\left(d\right)\right\}} \leq g(d).
\end{equation}
\item\label{Assump:QuadraticBound} There are constants $b,c_1,c_2>0$ (independent of $d$) such that the Fourier transform $\fconnf$ satisfies
\begin{equation}
    \inf_{\abs*{k}\leq b}\frac{1}{\abs*{k}^2}\left(1-\frac{1}{\phiint}\fconnf(k)\right) > c_1,\qquad \inf_{\abs*{k}> b}\left(1-\frac{1}{\phiint}\fconnf(k)\right) > c_2.
\end{equation}
\end{enumerate}
\end{assumptionp}

\begin{remark}
    We believe that the condition \eqref{eqn:2-convolutionAssumption} is not necessary for our results. In \cite{DicHey2022triangle} it was required by the lace expansion argument to provide the skeleton of an argument that would work for ``spread out'' models in dimensions $d=7,8$ (in addition to $d\geq 9$). However, we are concerned here with taking $d\to\infty$, and so it should not be required. \hfill $\diamond$
\end{remark}

It will sometimes be more natural to work with a parameter $\beta(d)$ that is related to $g(d)$. From \cite{DicHey2022triangle} it is defined by
\begin{equation}
    \label{eqn:betafromgfunction}
        \beta(d) := \begin{cases}
            g(d)^{\frac{1}{4} -\frac{3}{2d}}d^{-\frac{3}{2}} &\colon \lim_{d\to\infty}g(d)\rho^{-d}\Gamma\left(\frac{d}{2}+1\right)^2 = 0 \qquad\forall\rho>0,\\
            g(d)^\frac{1}{4}&\colon \text{otherwise}.
        \end{cases}
    \end{equation}
Note that the Assumption~\ref{Assump:ExponentialDecay} below implies that $\beta(d)=g(d)^{\frac{1}{4}}$.

Our second set of assumptions allow us to keep suitable control of asymptotic properties. Let us define $h\colon \N\to\R_{\geq 0}$ and $N\colon \N\to\N$ by
    \begin{align}
        h(d) := & \frac{1}{\phiint^5}\connf^{\star 6}\left(\orig\right) + \frac{1}{\phiint^4}\int \connf(x)\connf^{\star 2}(x)\connf^{\star 3}(x)\dd x + \frac{1}{\phiint^4}\int \left(\connf^{\star 2}(x)\right)^3\dd x,\\
        N(d) := & \ceil*{\frac{\log h(d)}{\log \beta(d)}}.
    \end{align}
\begin{assumptionp}{B}
    \label{AssumptionBeta}
    We require that:
    \begin{enumerate}[label=\textbf{(\ref{AssumptionBeta}.\arabic*)}]
    \item\label{Assump:ExponentialDecay} There exists $\rho>0$ such that $\liminf_{d\to\infty}\rho^{-d}\phiint^{-5}\connf^{\star 6}\left(\orig\right)>0$.
    \item\label{Assump:NumberBound} $\limsup_{d\to\infty}N(d) < \infty$.
    \end{enumerate}    
\end{assumptionp}
Mind that Assumption~\ref{Assump:NumberBound} is in practice a lower bound on $h(d)$ because $\beta(d)<1$ for large $d$.
 \begin{remark}
     The factor $\connf^{\star 6}\left(\orig\right)$ appears in Assumption~\ref{Assump:ExponentialDecay} only because $\connf^{\star 6}\left(\orig\right)$ is the precision at which we stop our expansion. If we wished to proceed up to the $\connf^{\star m}\left(\orig\right)$ term, then we would need a version of ~\ref{Assump:ExponentialDecay} with $\connf^{\star m}\left(\orig\right)$ replacing $\connf^{\star 6}\left(\orig\right)$.  Assumption~\ref{Assump:ExponentialDecay} appears in our proof via Lemma~\ref{lem:tailbound}. A close inspection of the proof would reveal that it is a slightly stronger condition than is needed there. However the version presented is more concise and sufficient for the models we consider here. 
     
     The requirement that \ref{Assump:NumberBound} holds becomes apparent through Proposition~\ref{prop:PiBounds}. We take great care in describing the asymptotics of the first few terms in the expansion of $\fLacelamC(0)$ because they dictate the behaviour of $\lambda_c$ that we are interested in. On the other hand we can utilise pre-existing bounds for the tail of the expansion to show that it can be neglected in our calculations. If we fix a cut-off $N\geq 1$ in this expansion then these pre-existing bounds are of order $\beta^N$. Assumption~\ref{Assump:NumberBound} ensures that we can choose a fixed $N$ such that this tail error is smaller than the error terms arising elsewhere in the expansion. If this was not the case, we may try to let $N\to\infty$ as $d\to\infty$, but then we would be summing a diverging number of ``small'' terms prior to the cut-off and we would not have a good control on this. \hfill $\diamond$
 \end{remark}

\begin{definition}
    In addition to using the convolution operation to combine two non-negative functions $f,g\colon \Rd\to \R_{\geq 0}$, we will also find it convenient to use $f\cdot g$ to denote the pointwise multiplication of $f$ and $g$:
    \begin{equation}
        f\cdot g(x) := f(x)g(x).
    \end{equation}
    Furthermore, for $n_1,n_2,n_3\geq 1$, we will denote
    \begin{equation}
        \connf^{\star n_1\star n_2 \cdot n_3}\left(\orig\right) := \connf^{\star n_1}\star\left(\connf^{\star n_2}\cdot\connf^{\star n_3}\right)\left(\orig\right) = \int \connf^{\star n_1}(x)\connf^{\star n_2}(x)\connf^{\star n_3}(x)\dd x.
    \end{equation}
    This expression shows that $\connf^{\star n_1\star n_2 \cdot n_3}\left(\orig\right)$ is invariant under the permutation of $n_1$, $n_2$, and $n_3$.
\end{definition}

\begin{theorem}
\label{thm:CriticalIntensityExpansion}
    Suppose Assumptions~\ref{Assumption} and \ref{AssumptionBeta} are satisfied. Then as $d\to\infty$,
    \begin{multline}\label{eq:ExpansionMain}
        \lambda_c = \frac{1}{\phiint} + \frac{1}{\phiint^3}\connf^{\star 3}\left(\orig\right) + \frac{3}{2}\frac{1}{\phiint^4}\connf^{\star 4}\left(\orig\right) + 2\frac{1}{\phiint^5}\connf^{\star 5}\left(\orig\right) - \frac{5}{2}\frac{1}{\phiint^4}\connf^{\star 1\star 2\cdot2}\left(\orig\right) + 2\frac{1}{\phiint^5}\left(\connf^{\star 3}\left(\orig\right)\right)^2 \\+ \LandauBigO{\frac{1}{\phiint^6}\connf^{\star 3}\left(\orig\right)\connf^{\star 4}\left(\orig\right) + \frac{1}{\phiint^7}\left(\connf^{\star 3}\left(\orig\right)\right)^3 + \frac{1}{\phiint^6}\connf^{\star 6}\left(\orig\right) +  \frac{1}{\phiint^5}\connf^{\star 2\star 2\cdot 2}\left(\orig\right) + \frac{1}{\phiint^5}\connf^{\star 1 \star 2\cdot 3}\left(\orig\right)}.
    \end{multline}
\end{theorem}

\paragraph{Remarks on Graphical Notation.}
It will often be convenient and clearer to represent the objects like $\connf^{\star n}\left(\orig\right)$ and $\connf^{\star n_1 \star n_2 \cdot n_3}\left(\orig\right)$ pictorially. By expanding out the convolutions in these expressions it is clear that they are integrals over some finite set of points with functions associating pairs of these points (and sometimes the origin). We are therefore able to represent these integrals pictorially as rooted graphs. In these we represent the spatial origin $\orig\in\Rd$ with the root vertex $\OriginDot$, and an integral of some $x\in\Rd$ with the vertex $\IntegralDot$. If we can interpret a $\connf$ function to be ``connecting'' two $\Rd$-values, then we draw a line $\connfline$ between the vertices corresponding to the two $\Rd$-values. For example, this allows us to graphically represent objects such as
\begin{align}
    \connf^{\star 3}\left(\orig\right) &= \int\connf(x)\connf(y)\connf(x-y)\dd x\dd y = \loopthree,\\
    \connf^{\star 1\star 2\cdot 2}\left(\orig\right) &= \int\connf(x)\connf(y)\connf(z)\connf(x-z)\connf(z-y)\dd x\dd y\dd z = \fourcrosstwo. 
\end{align}
Observe that convolution is a commutative binary relation. This means for example that various diagrams the position of the root vertex $\OriginDot$ is not important. The most common example of this in our arguments will relate to $\connf^{\star 1 \star 2 \cdot 2}\left(\orig\right)$. By first recalling that $\connf^{\star n_1\star n_2 \cdot n_3}\left(\orig\right)$ is invariant under the permutation of $n_1$, $n_2$ and $n_3$, and then using the commutativity property of convolution, we find
\begin{equation}
    \fourcrosstwo = \connf\star\left(\connf^{\star 2}\cdot\connf^{\star 2}\right)\left(\orig\right) = \connf^{\star 2}\star\left(\connf\cdot\connf^{\star 2}\right)\left(\orig\right) = \connf\star\left(\connf\cdot\connf^{\star 2}\right)\star\connf\left(\orig\right) = \fourcrossone.
\end{equation}
We will tend to prefer $\fourcrossone$ over $\fourcrosstwo$, as we find the former slightly easier to read.

This graphical notation allows us to write the expansion of Theorem~\ref{thm:CriticalIntensityExpansion} in a form that is much easier to read. By a rescaling argument (see \cite[Section~5.1]{HeyHofLasMat19} for the details), we may assume without loss of generality that
\begin{equation}
    \phiint=\int \connf(x)\dd x =1,
\end{equation}
and we shall silently make this assumption in our analysis. 
Under this scaling choice, the expansion \eqref{eq:ExpansionMain} is represented pictorially by
\begin{multline}\label{eq:expansionmainpic}
    \lambda_c = 1+\loopthreeempty + \frac{3}{2}\loopfourempty + 2\loopfiveempty - \frac{5}{2}\fourcrossoneempty +2\left(\loopthreeempty\right)^2 \\+ \LandauBigO{\loopthreeempty\times\loopfourempty + \left(\loopthreeempty\right)^3 + \loopsixempty +  \phiThreeTwoOneempty + \phiTwoTwoTwoempty}.
\end{multline}

For some calculations, we will want to integrate a $\tlam$ function instead of a $\connf$ ($\tlam$ is defined below at \eqref{eqn:two-point_function}). We will also sometimes find it convenient to write the sum of two integrals as one integral by using $1-\connf$ to associate two $\Rd$-values. When we can interpret a $\tlam$ function to be ``connecting'' two $\Rd$-values, then we draw a green line $\tlamline$ between the vertices corresponding to the two $\Rd$-values, and similarly we draw a red line $\Notconnfline$ when a $1-\connf$ connects two values. As examples, we can use these to represent the following two integrals:
\begin{align}
    \int\connf(y)\tlam(x)\tlam(x-y)\dd x\dd y & = \PiOnePtSix,\\
    \int\connf(x)\connf(y)\connf(z-x)\connf(z-y)\left(1-\connf(z)\right)\dd x\dd y\dd z& = \FtwoPtOne.
\end{align}

\subsection{Applications}\label{sec:applications}
The result of Theorem~\ref{thm:CriticalIntensityExpansion} is very general in that the Assumptions~\ref{Assumption} and \ref{AssumptionBeta} apply to very many models. We now apply it to a number of examples. 

\subsubsection{The Gilbert disk model resp.\ the Hyper-Sphere RCM}
        For $R>0$, the Hyper-Sphere RCM is defined by having
        \begin{equation}
            \connf(x) = \Id_{\left\{\abs*{x}< R\right\}}.
        \end{equation}
        This is the classical model for Boolean percolation studied by Gilbert in 1961 \cite{Gil61}. 
        Writing $B\left(x;a,b\right)=\int^x_0t^{a-1}\left(1-t\right)^{b-1}\dd t$ for the \emph{incomplete Beta function} and $\Gamma\left(x\right)=\int^\infty_0t^{x-1}\e^{-t}\dd t$ is the \emph{Gamma function}, we obtain the following expansion of the critical intensity: 
        
\begin{corollary}\label{cor:sphere}
    For the Hyper-Sphere RCM with radius $R=R(d)>0$,
    \begin{equation}
        \frac{\pi^{\frac{d}{2}}}{\Gamma\left(\frac{d}{2}+1\right)}R^{d}\lambda_c =1 + \frac{3}{2\sqrt{\pi}}\frac{\Gamma\left(\frac{d}{2}+1\right)}{\Gamma\left(\frac{d}{2}+\frac{1}{2}\right)}B\left(\frac{3}{4};\frac{d}{2}+\frac{1}{2},\frac{1}{2}\right) + \LandauBigO{\frac{1}{\sqrt{d}}\left(\frac{16}{27}\right)^\frac{d}{2}}.
    \end{equation}
\end{corollary}

\begin{remark}
    Here we only expand as far as the $\connf^{\star 3}\left(\orig\right)$ term, and our error is the asymptotic size of the $\connf^{\star 4}\left(\orig\right)$ term. This is because these are the only terms for which we have rigorous closed-form expressions for their asymptotic size. Conjecture~\ref{conj:HyperDiscModel} gives the expected terms in the expansion based on numerical estimates of their asymptotic behaviour. \hfill $\diamond$
\end{remark}
        
\begin{figure}
    \centering
    \includegraphics[width=.3\textwidth]{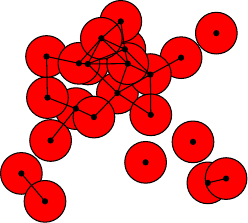}
    \hskip5em
    \includegraphics[width=.3\textwidth]{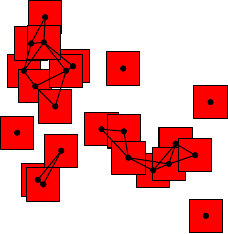}
    \caption{Left: The Hyper-Sphere RCM -- two Poisson points are connected whenever the circles of radius $R/2$ overlap.
    Right: The Hyper-Cube RCM -- two Poisson points are connected whenever the cubes of side length $L/2$ overlap.}
    \label{fig:enter-label}
\end{figure}
\subsubsection{The Hyper-Cube RCM}
While the hyper-sphere model is a good example showing that the numerical integration of the various convolutions of the adjacency function in \eqref{eq:ExpansionMain} can get fairly involved, the calculations simplify significantly for the Hyper-Cubic RCM given by 
        \begin{equation}
            \connf(x) = \prod^d_{j=1}\Id_{\left\{\abs*{x_j}\leq L/2\right\}},
        \end{equation}
where $x=\left(x_1,\ldots,x_d\right)\in\Rd$ and  $L>0$ is a parameter.

\begin{corollary}\label{cor:cube}
    For the Hyper-Cubic RCM with side length $L=L(d)>0$, as $d\to\infty$
    \begin{equation}
        L^d\lambda_c = 1 + \left(\frac{3}{4}\right)^d + \frac{3}{2}\left(\frac{2}{3}\right)^d + 2\left(\frac{115}{192}\right)^d - \frac{5}{2}\left(\frac{7}{12}\right)^d + 2\left(\frac{9}{16}\right)^d + \LandauBigO{\left(\frac{11}{20}\right)^d}.
    \end{equation}
\end{corollary}

\subsubsection{The Gaussian RCM}
    For $\sigma^2>0$ and $0<\Acal\leq \left(2\pi\sigma^2\right)^\frac{d}{2}$, the Gaussian RCM is defined by having
        \begin{equation}
            \connf\left(x\right) = \frac{\Acal}{\left(2\pi \sigma^2\right)^\frac{d}{2}}\exp\left(-\frac{1}{2\sigma^2}\abs*{x}^2\right).
        \end{equation}
    The parameter $\sigma$ is a length scale parameter while the $\Acal$ factor ensures $\Acal = \int\connf(x)\dd x$. The upper bound on $\Acal$ is only there to ensure $\connf$ is $\left[0,1\right]$-valued.
    Then we have the following expansion: 
\begin{corollary}\label{cor:Gaussian}
    For the Gaussian RCM with $\Acal=\Acal(d)>0$ and $\sigma=\sigma(d)>0$ such that $\liminf_{d\to\infty}\connf\left(\orig\right)^{\frac{1}{d}}>0$, as $d\to\infty$
    \begin{equation}
        \Acal\lambda_c = 1 + \Acal\left(6\pi\sigma^2\right)^{-\frac{d}{2}} + \frac{3}{2}\Acal\left(8\pi\sigma^2\right)^{-\frac{d}{2}} + 2\Acal\left(10\pi\sigma^2\right)^{-\frac{d}{2}} + \LandauBigO{\Acal\left(12\pi\sigma^2\right)^{-\frac{d}{2}}}.
    \end{equation}
    In particular, if $\connf\left(\orig\right)=\Acal\left(2\pi\sigma^2\right)^{-\frac{d}{2}} = 1$, then
    \begin{equation}
        \Acal\lambda_c = 1 + 3^{-\frac{d}{2}} + \frac{3}{2}\times4^{-\frac{d}{2}} +2\times5^{-\frac{d}{2}} + \LandauBigO{6^{-\frac{d}{2}}}. 
    \end{equation}
\end{corollary}

\subsubsection{The Coordinate-wise Cauchy RCM}
    In a similar flavour to the previous example, let  $\gamma>0$ and $0<\Acal\leq \left(\gamma\pi\right)^d$, and define the Coordinate-Cauchy RCM through
        \begin{equation}
            \connf(x) = \frac{\Acal}{\left(\gamma\pi\right)^d}\prod^d_{j=1}\frac{\gamma^2}{\gamma^2+x^2_j},
        \end{equation}
        where $x=\left(x_1,\ldots,x_d\right)\in\Rd$. Like for the Gaussian RCM we have a length-scale parameter $\gamma$ while the $\Acal$ factor ensures $\Acal = \int\connf(x)\dd x$ and the upper bound on $\Acal$ is only there to ensure $\connf$ is $\left[0,1\right]$-valued.
    Then the expansion of the critical intensity is as follows:     
    \begin{corollary}\label{cor:Cauchy}
        For the Coordinate-Cauchy RCM with $\Acal=\Acal(d)>0$ and $\gamma=\gamma(d)>0$ such that $\liminf_{d\to\infty}\connf\left(\orig\right)^{\frac{1}{d}}>0$, as $d\to\infty$
        \begin{equation}
            \Acal\lambda_c = 1 + \Acal\left(3\gamma\pi\right)^{-d} + \frac{3}{2}\Acal\left(4\gamma\pi\right)^{-d} + 2\Acal\left(5\gamma\pi\right)^{-d} + \LandauBigO{\Acal\left(6\gamma\pi\right)^{-d}}.
        \end{equation}
        In particular, if $\connf\left(\orig\right) = \Acal\left(\gamma\pi\right)^{-d} = 1$, then
        \begin{equation}
            \Acal \lambda_c = 1 + 3^{-d} + \frac{3}{2}\times 4^{-d} + 2\times 5^{-d} + \LandauBigO{6^{-d}}.
        \end{equation}
    \end{corollary}

    \begin{remark}
    The condition on $\connf\left(\orig\right)$ appearing in Corollaries~\ref{cor:Gaussian} and \ref{cor:Cauchy} is to 
    ensure that \ref{Assump:ExponentialDecay} is satisfied. If this were not imposed, then the terms in our expansion could be so small that extra error terms arising from the volume of small balls of fixed radius could become significant and dominate. \hfill $\diamond$
\end{remark}


\subsection{Discussion}

Our results reveal a remarkable difference between continuum percolation models and lattice percolation: while the expansion in \eqref{eq:pc_bond} and \eqref{eq:pc_site} decays algebraically in $d$, we observe that the expansions in Corollaries \ref{cor:sphere}--\ref{cor:Cauchy} decay exponentially in $d$. Interestingly, the expansion in \eqref{eq:ExpansionMain} resp.\ \eqref{eq:expansionmainpic} is indeed algebraic, and it is the calculation of the convolutions of $\connf$ that transform it to an exponentially decaying series. This is reflected in the observation that the hypercubic lattice is a ``sparse'' graph in high dimensional Euclidean space. 
Indeed, the analysis in \cite{SatKamSak20} suggests that we do have exponential decay on lattices that use the space more efficiently such as the body-centred cubic lattice. 

Torquato \cite{Tor12} has provided an expansion for $\lambda_c$ using exact calculations. Interestingly, for the hyper-cubic Boolean model, we seem to get a slightly different expansion as the $\connf^{\star 5}\left(\orig\right)$ term is absent from their expression. 

It is clear that the value of $\lambda_c$ is highly sensitive to the choice of the connectivity function $\connf$. As a result, we get fairly different expansions for the four models in Section \ref{sec:applications}. 
Jonasson \cite{Jon01} has shown that for Boolean models, $\lambda_c$ is maximised for the hypersphere model, and minimised for a certain triangular shape. 

Our analysis is based on the lace expansion for the (plain) random connection model derived in \cite{HeyHofLasMat19}. 
A key quantity in that expansion is the \emph{lace expansion coefficient} $\Pi_{\lambda_c}(x)$ (defined in Definition \ref{def:Pi} below), see \eqref{eqn:OZE}. The main insight is that $\int\Pi_{\lambda_c}(x)\dd x$ encodes $\lambda_c$, see \eqref{eqn:critical}, and we therefore need to investigate this integral as the dimension $d$ increases. While the original lace expansion only needs (fairly crude) upper bounds on the different terms that constitute $\Pi_{\lambda_c}(x)$, in the present work we need to improve and refine these bounds to get asymptotically matching upper and lower bounds. This is the content of Section \ref{sec:Pibd}. 

In our main expansion in \eqref{eq:ExpansionMain}, there are various terms appearing on the right-hand side. Apart from the constant term ${\phiint}^{-1}$, the main contribution is given by the single loop diagram ${\phiint^{-3}}\connf^{\star 3}\left(\orig\right)$. 
However, the order of the further terms may depend on the particular form of $\connf$, e.g.\ compare Corollaries \ref{cor:cube} and \ref{cor:Gaussian}.

It is an open problem to extend this analysis to the marked random connection model, for which the lace expansion has recently been derived in \cite{DicHey2022triangle}.

\section{Preliminaries}

    Recall that $\eta$ denotes the homogeneous Poisson point process on $\Rd$ that gives the vertex set of the RCM. We then let $\xi$ denote the vertex set and the edge set together - the whole random graph. We also want to consider the augmented configurations $\eta^x$ and $\xi^x$. Here $\eta^x$ is produced by introducing an extra vertex at $x\in\Rd$, and $\xi^x$ then takes this augmented vertex set, copies the old edges, and independently forms edges between the old vertices and the new vertex. This can also be extended to get $\eta^{x,y}$ and $\xi^{x,y}$ for $x,y\in\Rd$, or for any finite number of augmenting vertices. For the full details of this construction see \cite[Section~2.2]{HeyHofLasMat19}. 

    Recall that $\connf(x)$ returns the probability that a vertex at the origin and a vertex at $x$ have a common edge, or are adjacent. Given two vertices $x,y\in\Rd$, we say that $x$ and $y$ are \emph{connected in $\xi^{x,y}$}, or $\conn{x}{y}{\xi^{x,y}}$, if there exists a finite sequence of distinct vertices $x=u_0,u_1,\ldots,u_k,u_{k+1}=y\in\eta^{x,y}$ (with $k\in\N_0$) such that $u_i\sim u_{i+1}$ for all $0\leq i\leq k$. We can then define the \emph{two-point} (or \emph{pair-connectedness}) function $\tlam\colon \Rd \to \left[0,1\right]$ by
    \begin{equation}
    \label{eqn:two-point_function}
        \tlam(x) := \pla\left(\conn{\orig}{x}{\xi^{\orig,x}}\right).
    \end{equation}

Now we introduce two preliminary results that we will use on many occasions in this paper: Mecke's (multivariate) equation, and the BK inequality. 

\paragraph{Mecke's Equation} Since our vertex set $\eta$ is a Poisson point process, we will often rely on a result called Mecke's Equation to use integral expressions to describe the expected number of certain configurations in our RCM. For a discussion of this result see \cite[Chapter~4]{LasPen17}. Given $m\in\N$ and a measurable non-negative function $f=f(\xi,\vec{x})$, the Mecke equation for $\xi$ states that 
\begin{equation}
    \E_\lambda \left[ \sum_{\vec x \in \eta^{(m)}} f(\xi, \vec{x})\right] = \lambda^m \int
				\E_\lambda\left[ f\left(\xi^{x_1, \ldots, x_m}, \vec x\right)\right] \dd \vec{x},  \label{eq:prelim:mecke_n}
\end{equation}
where $\vec x=(x_1,\ldots,x_m)$ and $\eta^{(m)}=\{(x_1,\ldots,x_m)\colon x_i \in \eta, x_i \neq x_j \text{ for } i \neq j\}$.

\paragraph{BK Inequality} We give an overview here, but the full details can be found in \cite{HeyHofLasMat19}. Given two increasing events $E_1$ and $E_2$, we define $E_1\circ E_2$ to be the event that $E_1$ and $E_2$ both occur, but do so on disjoint subsets of the vertices $\eta$. Note that in the case of $E_1=\left\{\conn{x}{y}{\xi^{x,y}}\right\}$ and $E_2=\left\{\conn{u}{v}{\xi^{u,v}}\right\}$, $E_1\circ E_2$ can still occur if $x\in\left\{u,v\right\}$ or $y\in\left\{u,v\right\}$ - the intermediate vertices need to be disjoint. The BK inequality then gives us a simple upper bound on the probability of this disjoint occurence.

\begin{theorem}[BK inequality]
    Let $E_1$ and $E_2$ be two increasing events that live on some bounded measurable subset on $\Rd$. Then
    \begin{equation}
        \pla\left(E_1\circ E_2\right) \leq \pla\left(E_1\right)\pla\left(E_2\right).
    \end{equation}
\end{theorem}
\begin{proof}
    See \cite[Theorem~2.1]{HeyHofLasMat19}.
\end{proof}

\begin{definition}
    We make use of a bootstrap function also used in \cite{HeyHofLasMat19} (itself adapted from an argument in \cite{HeyHofSak08}). Recall that we are using the scaling choice that $\fconnf(0) = \phiint = 1$. For $\lambda\geq 0$ and $k,l\in\Rd$, we define
    \begin{align}
        \mulam &:= 1- \frac{1}{\ftlam(0)}\\
        \fgmu(k) &:= \frac{1}{1-\mulam\fconnf(k)}.
    \end{align}
    Note that $\fgmu$ can be interpreted as the Fourier transform of the Green's function of a random walk with transition density $\mulam\connf$. We can define $f \colon \R_{\geq 0}\to \R_{\geq 0}$ with
    \begin{equation}
        f(\lambda):= \sup_{k\in\Rd}\frac{\abs*{\ftlam(k)}}{\fgmu(k)}.
    \end{equation}
\end{definition}

\begin{prop}
\label{prop:bootstrapbound}
    Suppose Assumption~\ref{Assumption} holds. Then for $d$ sufficiently large, $f(\lambda)\leq 2$ for all $\lambda\in\left[0,\lambda_c\right)$.
\end{prop}

\begin{proof}
    This is implied by \cite[Proposition~5.10]{HeyHofLasMat19}.
\end{proof}

\begin{lemma}
\label{lem:tailbound}
    Suppose Assumption~\ref{Assumption} holds and that there exists $\rho>0$ such that $\liminf_{d\to\infty}\rho^{-d}\connf^{\star m}\left(\orig\right) >0$. Let $d$ be sufficiently large, $m\geq 1$ be even, $s\geq 1$, and $\lambda\in\left[0,\lambda_c\right]$. Then there exists $K_{s}<\infty$ independent of $d$, $m$, and $\lambda$ such that
    \begin{equation}
        \sup_{x\in\Rd}\connf^{\star m}\star\tlam^{\star s}\left(x\right) \leq K_{s}\connf^{\star m}\left(\orig\right).
    \end{equation}
\end{lemma}
This is a key lemma in our proof as it allows us to identify leading order decay for convolutions of the adjacency function and the two-point function. 
\begin{proof}
    First let us consider $\lambda<\lambda_c$. We slightly adapt \cite[Lemma~5.4]{HeyHofLasMat19} for our purposes. From the Fourier inverse formula,
    \begin{equation}
        \sup_{x\in\Rd}\connf^{\star m}\star\tlam^{\star s}\left(x\right) \leq \sup_{x\in\Rd}\int \e^{-ik\cdot x}\fconnf(k)^m\ftlam(k)^s\frac{\dd k}{\left(2\pi\right)^d} \leq \int \fconnf(k)^m\abs*{\ftlam(k)}^s\frac{\dd k}{\left(2\pi\right)^d}.
    \end{equation}
    We can omit $\abs*{\cdot}$ from around $\fconnf(k)^m$ because $\fconnf(k)$ is real and $m$ is even. From the definition of the bootstrap function $f(\lambda)$, we can bound $\abs*{\ftlam(k)}$ with $f(\lambda)\fgmu(k)$ and then use $\mulam\leq 1$ to get
    \begin{equation}
    \label{eqn:integralbound}
        \sup_{x\in\Rd}\connf^{\star m}\star\tlam^{\star s}\left(x\right) \leq f(\lambda)^{s}\int \frac{\fconnf(k)^m}{\left(1-\mulam\fconnf(k)\right)^s}\frac{\dd k}{\left(2\pi\right)^d} \leq f(\lambda)^{s}\int \frac{\fconnf(k)^m}{\left(1-\fconnf(k)\right)^s}\frac{\dd k}{\left(2\pi\right)^d}.
    \end{equation}
    Recall the parameter $b>0$ arising from Assumption~\ref{Assump:QuadraticBound}. We partition the integral on the right hand side of \eqref{eqn:integralbound} into one integral over $\abs*{k}\leq b$, and one integral over $\abs*{k} > b$. For $\abs*{k}\leq b$, \ref{Assump:QuadraticBound} tells us that there exists $c_1>0$ such that $\left(1-\fconnf(k)\right)^{-1}\leq c_1^{-1}\abs*{k}^{-2}$, and therefore
    \begin{equation}
        \int_{\abs*{k}\leq b} \frac{\fconnf(k)^m}{\left(1-\fconnf(k)\right)^s}\frac{\dd k}{\left(2\pi\right)^d} \leq \frac{1}{c_1^s}\int_{\abs*{k}\leq b} \frac{1}{\abs*{k}^{2s}} \frac{\dd k}{\left(2\pi\right)^d} = \frac{1}{c_1^s}\frac{\mathfrak{S}_{d-1}}{d-2s}\frac{b^{d-2s}}{\left(2\pi\right)^{d}},
    \end{equation}
    where $\mathfrak{S}_{d-1} = d\pi^{\frac{d}{2}}\Gamma\left(1 + \frac{d}{2}\right)^{-1}$ is the surface area of a dimension $d$ hyper-sphere with unit radius. An application of Stirling's formula tells us that for all $\rho>0$ we have $\frac{\mathfrak{S}_{d-1}}{d-2s}\frac{b^{d-2s}}{\left(2\pi\right)^{d}} \leq \rho^d$ for sufficiently large $d$. Therefore this contribution is negligible for our purposes. For $\abs*{k}> b$, \ref{Assump:QuadraticBound} tells us that $\left(1-\fconnf(k)\right)^{-1}\leq c_2^{-1}$, and therefore
    \begin{equation}
        \int_{\abs*{k}> b} \frac{\fconnf(k)^m}{\left(1-\fconnf(k)\right)^s}\frac{\dd k}{\left(2\pi\right)^d} \leq \frac{1}{c_2^s}\int \fconnf(k)^m \frac{\dd k}{\left(2\pi\right)^d} = \frac{1}{c_2^s}\connf^{\star m}\left(\orig\right).
    \end{equation}
    In conjunction with \eqref{eqn:integralbound} and Proposition~\ref{prop:bootstrapbound}, this proves the result for $\lambda<\lambda_c$.

    To extend the result to $\lambda\leq \lambda_c$, we note that $\tlam(x)$ is monotone increasing in $\lambda$ for all $x\in\Rd$. Monotone convergence and the independence of the bound on $\lambda$ then proves the full result.    
\end{proof}

\begin{definition}
For $n\in\N$ and $x,y\in \Rd$, $x$ is connected to $y$ in $\xi^{x,y}$ by a path of length exactly $n$ if there exists a sequence of vertices $x=u_0,u_1,\ldots,u_{n-1},u_n=y$ such that $u_i\sim u_{i+1}$ for $0\leq i \leq n-1$. We then define $\left\{\xconn{x}{y}{\xi^{x,y}}{=n}\right\}$ as the event that $x$ is connected to $y$ in $\xi^{x,y}$ by a path of length exactly $n$, but no path of length $<n$. For $\lambda>0$ we denote
\begin{equation}
    \connf^{[n]}(x) := \pla\left(\xconn{\orig}{x}{\xi^{\orig,x}}{= n}\right).
\end{equation}
In particular, $\connf^{[1]}\equiv \connf$.

Additionally define for finite $A\subset\Rd$, 
\begin{equation}
    \connf^{[n]}_{\thinn{A}}\left(x,y\right) := \pla\left(\xconn{x}{y}{\xi^{x,y}_{\thinn{A}}}{= n}\right).
\end{equation}
That is, $\connf^{[n]}_{\thinn{A}}\left(x,y\right)$ is the probability that there exists a path of length $n$ connecting $x$ and $y$ in $\xi^{x,y}$, but none of the interior vertices in this path are adjacent to any vertices in $A$ and there is no path connecting $x$ and $y$ in $\xi^{x,y}$ that is of length $<n$. A more formal definition of $\xi^{x,y}_{\thinn{A}}$ can be found below in Definition~\ref{defn:ThinningsPivots}.

\end{definition}

\begin{lemma}
Let $x,y\in\Rd$ be distinct, $\lambda>0$, and $A\subset \Rd$ be a finite number of singletons. Then for $n\geq 1$,
    \begin{align}
        \connf^{[1]}_{\thinn{A}}(x,y) &= \connf(x-y) \label{eqn:Exconnf_1_AvoidA}\\
        \connf^{[n+1]}_{\thinn{A}}(x,y) &= \left(1-\connf(x-y)\right)\left(1 - \exp\left(-\lambda\int \connf(v-y)\connf^{[n]}_{\thinn{A\cup \left\{y\right\}}}(x,v)\prod_{z\in A}\left(1-\connf(v-z)\right)\dd v\right)\right)\label{eqn:Exconnf_n+1_AvoidA}\\
        \connf^{[n+1]}(x) &= \left(1-\connf(x)\right)\left(1 - \exp\left(-\lambda\int \connf(v)\connf^{[n]}_{\thinn{\orig}}(x,v)\dd v\right)\right). \label{eqn:Exconnf_n+1}
    \end{align}
In particular,
\begin{align}
    \connf^{[2]}(x) &= \left(1-\connf(x)\right)\left(1-\exp\left(-\lambda\connf^{\star 2}(x)\right)\right)\label{eqn:ExConnf2}\\
    \connf^{[3]}(x) &= \left(1-\connf(x)\right)\left(1 - \exp\left(-\lambda\int \connf(v)\left(1-\connf(x-v)\right)\right.\right.\nonumber\\
        &\hspace{3cm}\left.\left.\times\left(1 - \exp\left(-\lambda\int \connf(w-v)\connf(x-w)\left(1-\connf(w)\right)\dd w\right)\right)\dd v\right)\right).\label{eqn:ExConnf3}
\end{align}
\end{lemma}

\begin{proof}
    To show \eqref{eqn:Exconnf_1_AvoidA}, observe that if $\adja{x}{y}{\xi^{x,y}}$ then there are no interior points on this path to be adjacent to $A$. Therefore $\Exconnf{1}_{\thinn{A}}(x,y) = \Exconnf{1}(x-y) = \connf(x-y)$.

    For \eqref{eqn:Exconnf_n+1_AvoidA}, we first note that the existence of a single edge connecting $x$ and $y$ is independent of everything else. Since we cannot have this edge, we have a factor of $1-\connf(x-y)$ outside everything else. Let us now consider the neighbours of $y$ in $\eta$. The event $\left\{\xconn{x}{y}{\xi^{x,y}_{\thinn{A}}}{= n+1}\right\}$ occurs exactly when $x\not\sim y$ and there exists a neighbour $v$ of $y$ that is not adjacent to any point in $A$ and has a path of length $n$ from $v$ to $x$ that does not use any vertex adjacent to $A$ or adjacent to $y$ (otherwise a ``shortcut'' would exist). The existence of such a path is exactly the event $\left\{\xconn{x}{v}{\xi^{x,v}_{\thinn{A\cup\left\{y\right\}}}}{=n}\right\}$. Since $\eta$ is a Poisson point process, the number of such vertices is a Poisson distributed random variable with mean given by (via Mecke's equation)
    \begin{multline}
        \E_\lambda\left[\#\left\{v\in\eta\colon v\sim y, \xconn{x}{y}{\xi^{x,y}_{\thinn{A}}}{= n+1}, v\not\sim z \text{ for all }z\in A\right\}\right] \\= \lambda\int \connf(v-y)\connf^{[n]}_{\thinn{A\cup \left\{y\right\}}}(x,v)\prod_{z\in A}\left(1-\connf(v-z)\right)\dd v.
    \end{multline}
    If $X$ is a Poisson random variable with mean $M$, then $\mathbb{P}\left(X\geq 1\right) = 1 - \e^{-M}$. Since the number $\#\left\{v\in\eta\colon v\sim y, \xconn{x}{y}{\xi^{x,y}_{\thinn{A}}}{= n+1}, v\not\sim z \text{ for all }z\in A\right\}$ is a Poisson random variable, this returns the required second factor in \eqref{eqn:Exconnf_n+1_AvoidA}.

    To get \eqref{eqn:Exconnf_n+1}, use \eqref{eqn:Exconnf_n+1_AvoidA} with $A=\emptyset$ and $y=\orig$.

    To calculate $\Exconnf{2}$ and $\Exconnf{3}$, we iteratively use \eqref{eqn:Exconnf_1_AvoidA}, \eqref{eqn:Exconnf_n+1_AvoidA}, and \eqref{eqn:Exconnf_n+1}. For $\Exconnf{2}$ we have
    \begin{align}
        \Exconnf{2}(x) &= \left(1-\connf(x)\right)\left(1 - \exp\left(-\lambda\int \connf(v)\connf^{[1]}_{\thinn{\orig}}(x,v)\dd v\right)\right)\nonumber\\
        &= \left(1-\connf(x)\right)\left(1 - \exp\left(-\lambda\int \connf(v)\connf(x-v)\dd v\right)\right) \nonumber\\
        &= \left(1-\connf(x)\right)\left(1-\exp\left(-\lambda\connf^{\star 2}(x)\right)\right).
    \end{align}
    Similarly, we find
    \begin{equation}
        \Exconnf{2}_{\thinn{\orig}}(x,v) = \left(1-\connf(x-v)\right)\left(1 - \exp\left(-\lambda\int \connf(w-v)\connf(x-w)\left(1-\connf(w)\right)\dd w\right)\right),
    \end{equation}
    and therefore
    \begin{align}
        \Exconnf{3}(x) &= \left(1-\connf(x)\right)\left(1 - \exp\left(-\lambda\int \connf(v)\Exconnf{2}_{\thinn{\orig}}(x,v)\dd v\right)\right)\nonumber\\
        &= \left(1-\connf(x)\right)\left(1 - \exp\left(-\lambda\int \connf(v)\left(1-\connf(x-v)\right)\right.\right.\nonumber\\
        &\hspace{4cm}\left.\left.\times\left(1 - \exp\left(-\lambda\int \connf(w-v)\connf(x-w)\left(1-\connf(w)\right)\dd w\right)\right)\dd v\right)\right).
    \end{align}
\end{proof}

\begin{lemma}
\label{lem:ExconnfBound}
For $n\geq 1$, $\lambda>0$, and $x\in\Rd$,
    \begin{equation}
        \connf^{[n]}(x) \leq \lambda^{n-1}\connf^{\star n}(x).
    \end{equation}
\end{lemma}

\begin{proof}
    The expression $\connf^{[n]}(x)$ gives the probability that there exists at least one path from $\orig$ to $x$ of length $n$, and no shorter paths. We can bound this by the probability that there exists at least one path from $\orig$ to $x$ of length $n$. Then by Markov's inequality this is bounded by the expected number of paths from $\orig$ to $x$ of length $n$. By Mecke's equation this is given by $\lambda^{n-1}\connf^{\star n}(x)$.
\end{proof}

\begin{lemma}\label{lem:tauUpperbound}
For $m,n\geq 1$, $\lambda>0$, and $x\in\Rd$,
    \begin{equation}
        \sum^{m}_{i=1}\Exconnf{i}(x) \leq \tlam(x) \leq \sum^{n}_{i=1}\Exconnf{i}(x) + \lambda^n\connf^{\star (n+1)}(x) + \lambda^{n+1}\connf^{\star (n+1)}\star\tlam(x).
    \end{equation}
\end{lemma}

\begin{proof}
    First note that the events $\left\{\left\{\xconn{\orig}{x}{\xi^{\orig,x}}{=i}\right\}\right\}_{i\in\N}$ are pairwise disjoint. They are also all contained in the event $\left\{\conn{\orig}{x}{\xi^{\orig,x}}\right\}$. Therefore $\sum^{m}_{i=1}\Exconnf{i}(x) \leq \tlam(x)$.

    For the upper bound, the above comments imply that $\tlam(x)-\sum^{n+1}_{i=1}\Exconnf{i}(x)$ is the probability that $\orig$ and $x$ are connected in $\xi^{\orig,x}$ by some path of length $n+2$ or longer. We can then use Markov's inequality to bound this probability by the expected number of paths of length $n+2$ or longer. By using Mecke's equation, we get
    \begin{align}
        \tlam(x)-\sum^{n+1}_{i=1}\Exconnf{i}(x) &\leq \E_{\lambda}\left[\sum_{y\in\eta}\Id_{\left\{\xconn{\orig}{y}{\xi^{\orig}}{=n+1}\right\}\circ\left\{\conn{y}{x}{\xi^x}\right\}}\right]\nonumber\\
        &= \lambda \int \pla\left(\left\{\xconn{\orig}{y}{\xi^{\orig,y}}{=n+1}\right\}\circ\left\{\conn{y}{x}{\xi^{y,x}}\right\}\right)\dd y\nonumber\\
        &\leq \lambda\int\Exconnf{n+1}(y)\tlam(x-y)\dd y.
    \end{align}
    In this last inequality we have used the the BK inequality to bound the probability of the vertex-disjoint occurrence. We therefore have
    \begin{equation}
        \tlam(x) \leq \sum^{n}_{i=1}\Exconnf{i}(x) + \Exconnf{n+1}(x) + \lambda\Exconnf{n+1}\star\tlam(x).
    \end{equation}
    Bounding $\Exconnf{n+1}(x) \leq \lambda^{n}\connf^{\star (n+1)}(x)$ (as shown in Lemma~\ref{lem:ExconnfBound}) in these last two terms then gives the result.
\end{proof}

\begin{lemma}
    If $n_1,n_2,n_3\geq 2$, then
    \begin{equation}
        \int \connf^{\star n_1}(x)\connf^{\star n_2}(x)\connf^{\star n_3}(x)\dd x \leq \left(\int\connf(x)\dd x\right)^{n_1+n_2+n_3 - 6} \int\left(\connf^{\star 2}(x)\right)^3\dd x.
    \end{equation}
    If $n_1,n_2\geq 2$ and $n_1+n_2\geq 6$, then
    \begin{equation}
        \int \connf^{\star n_1}(x)\connf^{\star n_2}(x)\connf(x)\dd x \leq \connf^{\star (n_1+n_2)}\left(\orig\right) \leq \left(\int\connf(x)\dd x\right)^{n_1+n_2 - 6}\connf^{\star 6}\left(\orig\right).
    \end{equation}
\end{lemma}

\begin{proof}
    Recall that the Fourier transform of the convolution of two functions equals the pointwise product of their individual Fourier transforms, and the Fourier transform of the pointwise product of two functions equals the convolution of their individual Fourier transforms. Therefore
    \begin{equation}
        \int \connf^{\star n_1}(x)\connf^{\star n_2}(x)\connf^{\star n_3}(x)\dd x = \int \fconnf(k)^{n_1}\fconnf(k-l)^{n_2}\fconnf(l)^{n_3}\frac{\dd k \dd l}{\left(2\pi\right)^{2d}}.
    \end{equation}
    We then note that having $\connf(x)\geq 0$ implies $\sup_{k}\abs*{\fconnf(k)} = \fconnf(0) = \int\connf(x)\dd x$. Therefore a supremum bound implies
    \begin{align}
        \int \fconnf(k)^{n_1}\fconnf(k-l)^{n_2}\fconnf(l)^{n_3}\frac{\dd k \dd l}{\left(2\pi\right)^{2d}} &\leq \left(\int\connf(x)\dd x\right)^{n_1+n_2+n_3 - 6} \int \abs*{\fconnf(k)}^{2}\abs*{\fconnf(k-l)}^{2}\abs*{\fconnf(l)}^{2}\frac{\dd k \dd l}{\left(2\pi\right)^{2d}}\nonumber\\
        & = \left(\int\connf(x)\dd x\right)^{n_1+n_2+n_3 - 6} \int \fconnf(k)^{2}\fconnf(k-l)^{2}\fconnf(l)^{2}\frac{\dd k \dd l}{\left(2\pi\right)^{2d}}\nonumber\\
        & = \left(\int\connf(x)\dd x\right)^{n_1+n_2+n_3 - 6} \int\left(\connf^{\star 2}(x)\right)^3\dd x.
    \end{align}

    For the second inequality, we bound $\connf(x)\leq 1$ to leave the convolution $\connf^{\star n_1}\star\connf^{\star n_2}\left(\orig\right)= \connf^{\star (n_1+n_2)}\left(\orig\right)$. Then like above we have
    \begin{multline}
        \connf^{\star (n_1+n_2)}\left(\orig\right) = \int \fconnf(k)^{n_1 + n_2}\frac{\dd k}{\left(2\pi\right)^d} \\
        \leq \left(\int\connf(x)\dd x\right)^{n_1+n_2 - 6} \int \fconnf(k)^{6}\frac{\dd k}{\left(2\pi\right)^d} = \left(\int\connf(x)\dd x\right)^{n_1+n_2 - 6}\connf^{\star 6}\left(\orig\right).
    \end{multline}
\end{proof}

\section{Lace Expansion Coefficients}\label{sec:Pibd}
The key to our proof is a decomposition of the lace expansion coefficients. In preparation for defining them, we need a few more elementary definitions. The full definitions can be found in \cite{HeyHofLasMat19}.

\begin{definition}[Thinnings and Pivotal Points]
\label{defn:ThinningsPivots}
    Let $x,y\in\Rd$ and $A\subset \Rd$ be a locally finite set. 
    \begin{enumerate}
        \item Let $\eta$ be a vertex set. We produce a vertex set $\eta_{\thinn{A}}$ by retaining each $\omega\in\eta$ with probability $\overline{\connf}(A,\omega):= \prod_{z\in A}\left(1-\connf(\omega,z)\right)$. We call $\eta_{\thinn{A}}$ an \emph{$A$-thinning} of $\eta$. A similar procedure can be followed to define $\eta^x_{\thinn{A}}$ from $\eta^x$.

        \item Define $\left\{\xconn{x}{y}{\xi}{A}\right\}$ to be the event that $x,y\in\eta$ and $x$ is connected to $y$ in $\xi$, but that this connection does not survive an $A$-thinning of $\eta\setminus\left\{x\right\}$. In particular, the connection does not survive if $y$ is thinned out.

        \item The vertex $u\in\Rd$ is \emph{pivotal} and $u\in\piv{x,y,\xi}$ if every path on $\xi^{x,y}$ that connects $x$ to $y$ uses the vertex $u$. The end points $x$ and $y$ are never said to be pivotal.

        \item Define
        \begin{equation}
            E\left(x,y;A,\xi\right) := \left\{\xconn{x}{y}{\xi}{A}\right\}\cap\left\{\not\exists w\in\piv{x,y;\xi}\colon \xconn{x}{w}{\xi}{A}\right\}.
        \end{equation}
        If one considers the pivotal points from $x$ to $y$ in $\xi$ in sequence, then this is the event that an $A$-thinning breaks the connection after the last pivotal point and not before.

        \item Define
        \begin{equation}
            \left\{\dconn{x}{y}{\xi^{x,y}}\right\}:= \left\{\conn{x}{y}{\xi^{x,y}}\right\}\circ\left\{\conn{x}{y}{\xi^{x,y}}\right\}.
        \end{equation}
        Note that this is equal to the event that $x$ and $y$ are adjacent or there exist vertices $u,v$ in $\eta$ that are adjacent to $x$ and have disjoint paths to $y$ that both do not contain $x$. Alternatively, there are no pivotal points for the connection of $x$ and $y$ in $\xi^{x,y}$.
    \end{enumerate}
\end{definition}

We are now able to define the lace expansion coefficients, which will be the main objects of study in the remainder of the paper.

\begin{definition}\label{def:Pi}
    For $n\in\N$, $x\in\Rd$, and $\lambda\in\left[0,\lambda_c\right]$ we define
	\begin{align}
	    \Pi_\lambda^{(0)}(x) &:= \pla \left(\dconn{\orig}{x}{\xi^{\orig,x}}\right) - \connf(x), \label{eq:LE:Pi0_def} \\
     	\Pi_\lambda^{(n)}(x) &:= \lambda^n \int \pla \left( \{\dconn{\orig}{u_0}{\xi^{\orig, u_0}_{0}}\} \cap \bigcap_{i=1}^{n} E\left(u_{i-1},u_i; \C_{i-1}, \xi^{u_{i-1}, u_i}_{i}\right) \right) \dd \vec u_{[0,n-1]} , \label{eq:LE:Pin_def}
	\end{align}
where $u_n=x$, $\left\{\xi_i\right\}_{i\geq 0}$ are independent copies of $\xi$, and $\C_{i} = \C\left(u_{i-1}, \xi^{u_{i-1}}_{i}\right)$ is the cluster of $u_{i-1}$ in $\xi^{u_{i-1}}_i$. Then we further define
\begin{equation}
    \Pi_\lambda(x) = \sum_{n=0}^\infty\left(-1\right)^n\Pi_\lambda^{(n)}(x).
\end{equation}
Note that \cite[Corollary~6.1]{HeyHofLasMat19} proves that $\LacelamC^{(n)}(x) = \lim_{\lambda\nearrow\lambda_c}\Lacelam^{(n)}(x)$, and (in the proof) that $\fLacelamC^{(n)}(0) = \lim_{\lambda\nearrow\lambda_c}\fLacelam^{(n)}(0)$ and $\fLacelamC(0) = \lim_{\lambda\nearrow\lambda_c}\fLacelam(0)$.
\end{definition}

\begin{prop}
    Suppose Assumption~\ref{Assumption} holds and $d$ is sufficiently large. Then for all $\lambda\leq\lambda_c$ and $x\in\Rd$
    \begin{equation}
    \label{eqn:OZE}
        \tlam(x) = \connf(x) + \Lacelam(x) + \lambda\left(\connf + \Lacelam\right)\star\tlam(x).
    \end{equation}
\end{prop}

\begin{proof}
    This is the Ornstein-Zerneke equation for the random connection model, and it is proven in \cite{HeyHofLasMat19}. The $\lambda<\lambda_c$ result is in Corollary~5.3, and the $\lambda=\lambda_c$ result is in Corollary~6.1.
\end{proof}

Our main result for this section is the following proposition. 
\begin{prop}
\label{prop:PiBounds}
Suppose Assumptions~\ref{Assumption} and \ref{Assump:ExponentialDecay} hold. Also let $n_0\geq 4$ and $N\geq 1$ be fixed. Then as $d\to\infty$,
    \begin{align}
        \lambda_c\fLacelamC^{(0)}(0) &= \frac{1}{2}\lambda_c^3\loopfour - \frac{1}{2}\lambda_c^3\fourcrossone + \lambda_c^4\loopfive \nonumber\\ 
        &\hspace{4cm}+ \LandauBigO{\loopsix + \phiThreeTwoOne+\phiTwoTwoTwo},\label{eqPi0bd}\\
        \lambda_c\fLacelamC^{(1)}(0) &= \lambda_c^2\loopthree + 2\lambda_c^3\loopfour + 3\lambda_c^4\loopfive - 2\lambda_c^3\fourcrossone \nonumber\\
        &\hspace{4cm}+ \LandauBigO{\loopsix + \phiThreeTwoOne+\phiTwoTwoTwo},\label{eqPi1bd}\\
        \lambda_c\fLacelamC^{(2)}(0) &= \lambda_c^3\fourcrossone + \LandauBigO{\loopsix + \phiThreeTwoOne+\phiTwoTwoTwo},\label{eqPi2bd}\\
        \lambda_c\fLacelamC^{(3)}(0) &= \LandauBigO{\loopsix + \phiThreeTwoOne+\phiTwoTwoTwo},\label{eqPi3bd}\\
        \lambda_c\sum^{n_0}_{n=4}\left(-1\right)^n\fLacelamC^{(n)}(0) &= \LandauBigO{\loopsix},\label{eqPinbd}\\
        \lambda_c\sum^\infty_{n=N}\left(-1\right)^n\fLacelamC^{(n)}(0) &= \LandauBigO{\beta^N}. \label{eqn:PiTailBound}
    \end{align}
\end{prop}

Note that when Assumption~\ref{Assump:NumberBound} holds we can choose a fixed finite $N^*$ such that
\begin{equation}
    N^* \geq \ceil*{\frac{1}{\log \beta(d)}\log\left(\loopsix + \phiThreeTwoOne + \phiTwoTwoTwo\right)}
\end{equation}
for all $d\in\N$. If we then let $N=N^*$ in \eqref{eqn:PiTailBound}, the bound becomes
\begin{equation}
    \lambda_c\sum^\infty_{n=N^*}\left(-1\right)^n\fLacelamC^{(n)}(0) = \LandauBigO{\loopsix + \phiThreeTwoOne+\phiTwoTwoTwo}.
\end{equation}

\begin{corollary}
\label{thm:CoefficientExpansion}
Suppose Assumptions~\ref{Assumption} and \ref{AssumptionBeta} hold. Then as $d\to\infty$,
    \begin{multline}
        \lambda_c\fLacelamC(0) = - \lambda_c^2\loopthree - \frac{3}{2}\lambda_c^3\loopfour - 2\lambda_c^4\loopfive + \frac{5}{2}\lambda_c^3\fourcrossone \\+ \LandauBigO{\loopsix + \phiThreeTwoOne + \phiTwoTwoTwo}.
    \end{multline}
\end{corollary}
\begin{proof}
The corollary follows from 
$\fLacelamC(0)=\sum^\infty_{n=0}\left(-1\right)^n\fLacelamC^{(n)}(0)$ and the bounds in Proposition \ref{prop:PiBounds}.
\end{proof}

We prove Proposition \ref{prop:PiBounds} in the remainder of the section: 
\eqref{eqPi0bd} is proved in Section \ref{sec:Pi0bd}, \eqref{eqPi1bd} is proved in Section \ref{sec:Pi1bd}, \eqref{eqPi2bd} is proved in Section \ref{sec:Pi2bd}, \eqref{eqPi3bd} and \eqref{eqPinbd} are proven in Section \ref{sec:Pi0bd}. But first we show how it implies our main result.

\begin{proof}[Proof of Theorem~\ref{thm:CriticalIntensityExpansion}]
    By applying the Fourier transform to both sides of \eqref{eqn:OZE}, we can rearrange terms to find
    \begin{equation}
        \ftlam(k) = \frac{\fconnf(k) + \fLacelam(k)}{1- \lambda\left(\fconnf(k) + \fLacelam(k)\right)}
    \end{equation}
    for all $k\in\Rd$ and $\lambda\leq\lambda_c$ (where we interpret the right hand side as $=\infty$ if the denominator vanishes). Since Mecke's equation implies $\chi\left(\lambda\right) = 1 + \lambda \ftlam(0)$, and $\lambda_c = \inf\left\{\lambda>0 \colon \chi\left(\lambda\right)=\infty\right\}$, this tells us that $\lambda_c$ satisfies
    \begin{equation}
    \label{eqn:critical}
        \lambda_c\left(1+\fLacelamC(0)\right) = 1,
    \end{equation}
    where we have used $\fconnf(0)=1$. We now aim to use our expansion for $\fLacelamC(0)$ to get an expansion for $\lambda_c$.
    
    Let us denote $a=\loopthreeempty$, $b=\frac{3}{2}\loopfourempty- \frac{5}{2}\fourcrossoneempty$, $c=2\loopfiveempty$, and $r=\loopsixempty + \phiThreeTwoOneempty + \phiTwoTwoTwoempty$. 
    Using Corollary \ref{thm:CoefficientExpansion}, \eqref{eqn:critical} becomes
    \begin{equation}
        \lambda_c - a\lambda_c^2 - b\lambda_c^3 - c\lambda_c^4 +  \LandauBigO{r}= 1.
    \end{equation}

    We can rearrange this to get
    \begin{equation}
        \lambda_c = 1+ a\lambda_c^2 + b\lambda_c^3 + c\lambda_c^4 + \LandauBigO{r},
    \end{equation}
    and by substituting this into itself produces
    \begin{align}
        \lambda_c &= 1 + a\left(1+ a\lambda_c^2 + b\lambda_c^3 + c\lambda_c^4 + \LandauBigO{r}\right)^2 + b\left(1+ a\lambda_c^2 + b\lambda_c^3 + c\lambda_c^4 + \LandauBigO{r}\right)^3\nonumber\\
        &\hspace{7cm} + c\left(1+ a\lambda_c^2 + b\lambda_c^3 + c\lambda_c^4 + \LandauBigO{r}\right)^4 + \LandauBigO{r}\nonumber\\
        &= 1 + a + 2a^2\lambda_c^2 + \LandauBigO{ab\lambda_c^3 + a^3\lambda_c^4} +b + \LandauBigO{ab\lambda^2_c} + c + \LandauBigO{ac\lambda_c^2} + \LandauBigO{r}\nonumber\\
        & = 1 + a + b+ c+ 2a^2 + \LandauBigO{ab + a^3 + r}.
    \end{align}
    Finally, note that $b=\LandauBigO{\loopfourempty}$ and so the last term is exatly as stated in our result.
\end{proof}

\subsection{Bounds on the Zeroth Lace Expansion Coefficient}
\label{sec:Pi0bd}

In this subsection we prove \eqref{eqPi0bd}. 

\paragraph{Upper Bound on $\fLacelamC^{(0)}(0)$}

\begin{lemma}
Suppose Assumption~\ref{Assumption} holds. Then as $d\to\infty$,
    \begin{multline}
        \lambda_c\fLacelamC^{(0)}(0) \leq \frac{1}{2}\lambda_c^3\loopfour - \frac{1}{2}\lambda_c^3\fourcrossone + \lambda_c^4\loopfive \\+ \LandauBigO{\loopsix + \phiThreeTwoOne + \phiTwoTwoTwo}
    \end{multline}
\end{lemma}

\begin{proof}
    We first consider $\pla\left(\dconn{\orig}{x}{\xi^{\orig,x}}\right)$. Since the existence of an edge between $\orig$ and $x$ is independent of everything else,
    \begin{multline}
        \pla\left(\dconn{\orig}{x}{\xi^{\orig,x}}\right) \\= \connf(x) + \left(1-\connf(x)\right)\pla\left(\exists u,v\in\eta\colon \orig\sim u, \orig\sim v, \left\{\conn{u}{x}{\xi^{x}}\right\}\circ\left\{\conn{v}{x}{\xi^{x}}\right\}\right).
    \end{multline}
    Then note that the disjoint occurrence is a subset of the intersection of each occurrence: $\left\{\conn{u}{x}{\xi^{x}}\right\}\circ\left\{\conn{v}{x}{\xi^{x}}\right\} \subset \left\{\conn{u}{x}{\xi^{x}}\right\}\cap\left\{\conn{v}{x}{\xi^{x}}\right\}$. Therefore
    \begin{multline}
        \pla\left(\exists u,v\in\eta\colon u\ne v, \orig\sim u, \orig\sim v, \left\{\conn{u}{x}{\xi^{x}}\right\}\circ\left\{\conn{v}{x}{\xi^{x}}\right\}\right) \\\leq \pla\left(\#\left\{u\in\eta\colon \orig\sim u, \conn{u}{x}{\xi^{x}}\right\}\geq 2\right).
    \end{multline}
    Since $\eta$ is a Poisson point process, the number of such vertices is Poisson distributed and Mecke's equation tells us that the expected number of such vertices is given by
    \begin{equation}
        \E_{\lambda}\left[\#\left\{u\in\eta\colon\orig\sim u, \conn{u}{x}{\xi^{x}}\right\}\right] = \lambda\int\connf(v)\tlam(x-v)\dd v = \lambda\connf\star\tlam(x).
    \end{equation}
    Therefore
    \begin{align}
        &\pla\left(\exists u,v\in\eta\colon u\ne v, \orig\sim u, \orig\sim v, \left\{\conn{u}{x}{\xi^{x}}\right\}\circ\left\{\conn{v}{x}{\xi^{x}}\right\}\right) \nonumber\\
        &\hspace{5cm}\leq 1 - \pla\left(\#\left\{u\in\eta\colon\orig\sim u, \conn{u}{x}{\xi^{x}}\right\}\leq 1\right)\\
        & \hspace{5cm}= 1 - \exp\left(-\lambda\connf\star\tlam(x)\right) - \lambda\connf\star\tlam(x)\exp\left(-\lambda\connf\star\tlam(x)\right).
    \end{align}
    Using this with $1-\e^{-x} - x\e^{-x}\leq \frac{1}{2}x^2 + \frac{1}{6}x^3$ for all $x\in\R$ and \eqref{eq:LE:Pi0_def}, we can get
    \begin{multline}
        \lambda\fLacelam^{(0)}(0) \leq \lambda\int \left(1-\connf(x)\right)\left(1 - \exp\left(-\lambda\connf\star\tlam(x)\right) - \lambda\connf\star\tlam(x)\exp\left(-\lambda\connf\star\tlam(x)\right)\right)\dd x \\
        \leq \frac{1}{2}\lambda\int\left(1-\connf(x)\right)\left(\lambda\connf\star\tlam(x)\right)^2\dd x + \frac{1}{6}\lambda\int\left(1-\connf(x)\right)\left(\lambda\connf\star\tlam(x)\right)^3\dd x.
    \end{multline}
    By applying $\tlam(x) \leq \connf(x) + \lambda\connf\star\tlam(x)$ iteratively, we get
    \begin{align}
        &\int\left(1-\connf(x)\right)\left(\connf\star\tlam(x)\right)^2\dd x \nonumber\\
        &\hspace{1cm}\leq \int \left(1-\connf(x)\right)\connf^{\star 2}(x)^2\dd x + 2\lambda\int \left(1-\connf(x)\right)\connf^{\star 2}(x)\connf^{\star 3}(x)\dd x + 2\lambda^2\int \left(1-\connf(x)\right)\connf^{\star 2}(x)\connf^{\star 4}(x)\dd x\nonumber\\
        &\hspace{2cm} + 2\lambda^3\int \left(1-\connf(x)\right)\connf^{\star 2}(x)\connf^{\star 4}\star\tlam(x)\dd x + \lambda^2\int \left(1-\connf(x)\right)\connf^{\star 3}(x)^2\dd x \nonumber\\
        &\hspace{2cm} + 2\lambda^3\int \left(1-\connf(x)\right)\connf^{\star 3}(x)\connf^{\star 3}\star\tlam(x)\dd x + \lambda^4\int \left(1-\connf(x)\right)\connf^{\star 3}\star\tlam(x)^2\dd x\nonumber\\
        &\hspace{1cm}\leq \int \left(1-\connf(x)\right)\connf^{\star 2}(x)^2\dd x + 2\lambda\int \left(1-\connf(x)\right)\connf^{\star 2}(x)\connf^{\star 3}(x)\dd x\nonumber\\
        &\hspace{2cm}+ 3\lambda^2\connf^{\star 6}\left(\orig\right) + 4\lambda^3\connf^{\star 6}\star\tlam\left(\orig\right) + \lambda^4\connf^{\star 6}\star\tlam^{\star 2}\left(\orig\right).
    \end{align}
    From Lemma~\ref{lem:tailbound}, we know that for $\lambda\leq\lambda_c$ these last three terms are all $\LandauBigO{\connf^{\star 6}\left(\orig\right)}$. By further expanding the first two terms via the $\left(1-\connf(x)\right)$ factors, we find
    \begin{multline}
        \int\left(1-\connf(x)\right)\left(\connf\star\tlamC(x)\right)^2\dd x = \connf^{\star 4}\left(\orig\right)  - \int \connf(x)\connf^{\star 2}(x)^2\dd x + 2\lambda_c\connf^{\star 5}\left(\orig\right)\\ + \LandauBigO{\int \connf(x)\connf^{\star 2}(x)\connf^{\star 3}(x)\dd x + \connf^{\star 6}\left(\orig\right)}.
    \end{multline}
    By the same approach, we find
    \begin{align}
        &\int\left(1-\connf(x)\right)\left(\connf\star\tlam(x)\right)^3\dd x \nonumber\\
        &\hspace{1cm}\leq \int \left(1-\connf(x)\right)\connf^{\star 2}(x)^3\dd x 
        + 3\lambda\int \left(1-\connf(x)\right)\connf^{\star 2}(x)^2\connf^{\star 3}(x)\dd x \nonumber\\
        &\hspace{2cm} + 3\lambda^2\int \left(1-\connf(x)\right)\connf^{\star 2}(x)^2\connf^{\star 4}(x)\dd x 
        + 3\lambda^3\int \left(1-\connf(x)\right)\connf^{\star 2}(x)^2\connf^{\star 4}\star\tlam(x)\dd x  \nonumber\\
        &\hspace{2cm}+ 3\lambda^2\int \left(1-\connf(x)\right)\connf^{\star 2}(x)\connf^{\star 3}(x)^2\dd x
        + 6\lambda^3\int \left(1-\connf(x)\right)\connf^{\star 2}(x)\connf^{\star 3}(x)\connf^{\star 4}(x)\dd x  \nonumber\\
        &\hspace{2cm} + 6\lambda^4\int \left(1-\connf(x)\right)\connf^{\star 2}(x)\connf^{\star 3}(x)\connf^{\star 4}\star\tlam(x)\dd x 
        + 3\lambda^4\int \left(1-\connf(x)\right)\connf^{\star 2}(x)\connf^{\star 3}\star\tlam(x)^2\dd x \nonumber\\
        &\hspace{2cm} + \lambda^3\int \left(1-\connf(x)\right)\connf^{\star 3}(x)^3\dd x 
        + 3\lambda^4\int \left(1-\connf(x)\right)\connf^{\star 3}(x)^2\connf^{\star 3}\star\tlam(x)\dd x \nonumber\\
        &\hspace{2cm}+ 3\lambda^5\int \left(1-\connf(x)\right)\connf^{\star 3}(x)\connf^{\star 3}\star\tlam(x)^2\dd x 
        + \lambda^6\int \left(1-\connf(x)\right)\connf^{\star 3}\star\tlam(x)^3\dd x\nonumber\\
        &\hspace{1cm}\leq \int \left(1-\connf(x)\right)\connf^{\star 2}(x)^3\dd x 
        + 3\lambda\int \left(1-\connf(x)\right)\connf^{\star 2}(x)^2\connf^{\star 3}(x)\dd x \nonumber\\
        &\hspace{2cm} + 6\lambda^2\connf^{\star 6}\left(\orig\right)\int\connf(v)\dd v + 3\lambda^3\connf^{\star 6}\star\tlam\left(\orig\right)\int\connf(v)\dd v +6\lambda^3\connf^{\star 7}\left(\orig\right)\int\connf(v)\dd v \nonumber\\
        &\hspace{2cm} + 6\lambda^4\connf^{\star 7}\star\tlam\left(\orig\right)\int\connf(v)\dd v + 3\lambda^4\connf^{\star 6}\star\tlam^{\star 2}\left(\orig\right)\int\connf(v)\dd v \nonumber\\
        &\hspace{2cm}+ \lambda^3\connf^{\star 6}\left(\orig\right)\left(\int\connf(v)\dd v\right)^2  + 3\lambda^4\connf^{\star 6}\star\tlam\left(\orig\right)\left(\int\connf(v)\dd v\right)^2 + 3\lambda^5\connf^{\star 6}\star\tlam^{\star 2}\left(\orig\right)\left(\int\connf(v)\dd v\right)^2 \nonumber\\
        &\hspace{2cm} + \lambda^6\connf^{\star 6}\star \tlam^{\star 2}\left(\orig\right)\left(\int\connf(v)\dd v\right)^3.
    \end{align}
    Note that in this last inequality we identify two paths that form a loop - this contributes the terms $\connf^{\star 6}\left(\orig\right)$, $\connf^{\star 6}\star\tlam\left(\orig\right)$, etc. The $\left(1-\connf(x)\right)$ we again simply bound by $1$. This leaves a third path from $\orig$ to $x$. We deal with this by bounding one of the steps in the convolution by $1$ and the remaining steps form a `loose' integration. For example,
    \begin{multline}
        \int \left(1-\connf(x)\right)\connf^{\star 3}\star\tlam(x)^3\dd x \leq \int \connf^{\star 3}\star\tlam(x)^2\left(\int\connf^{\star 3}(u)\tlam(x-u)\dd u\right)\dd x \\\leq \int \connf^{\star 3}\star\tlam(x)^2\left(\int\connf^{\star 3}(u)\dd u\right)\dd x = \connf^{\star 6}\star \tlam^{\star 2}\left(\orig\right)\left(\int\connf(v)\dd v\right)^3.
    \end{multline}
    Recall that we have chosen the scaling $\int\connf(v)\dd v=1$ for our proof. Furthermore, by bounding $1-\connf(x)$ we find that the first two terms are $\LandauBigO{\int \connf^{\star 2}(x)^3\dd x }$. Therefore
    \begin{equation}
        \int\left(1-\connf(x)\right)\left(\connf\star\tlamC(x)\right)^3\dd x = \LandauBigO{\int \connf^{\star 2}(x)^3\dd x + \connf^{\star 6}\left(\orig\right)}.
    \end{equation}

    In summary, these bounds give us
    \begin{multline}
        \lambda_c\fLacelamC^{(0)}(0) \leq \frac{1}{2}\lambda_c^3\int \connf^{\star 2}(x)^2\dd x - \frac{1}{2}\lambda_c^3\int\connf(x) \connf^{\star 2}(x)^2\dd x + \lambda_c^4\int \connf^{\star 2}(x)\connf^{\star 3}(x)\dd x\\ + \LandauBigO{\int \connf^{\star 2}(x)^3\dd x + \int \connf(x)\connf^{\star 2}(x)\connf^{\star 3}(x)\dd x + \connf^{\star 6}\left(\orig\right)}
    \end{multline}
    as required.
\end{proof}

\paragraph{Lower Bound on $\fLacelamC^{(0)}(0)$}

\begin{lemma}
    \begin{equation}
        \lambda_c\fLacelamC^{(0)}(0) \geq \frac{1}{2}\lambda_c^3\loopfour - \frac{1}{2}\lambda_c^3\fourcrossone + \lambda_c^4\loopfive +\LandauBigO{\phiThreeTwoOne+\phiTwoTwoTwo}
    \end{equation}
\end{lemma}

\begin{proof}
    We lower bound $\Lacelam^{(0)}(x)$ by identifying an appropriate subset of $\left\{\dconn{\orig}{x}{\xi^{\orig,x}}\right\}$. Consider $\Fcal := \Fcal_1 \cup \Fcal_2 \cup \Fcal_3$, where
    \begin{align}
        \Fcal_1 :=& \left\{\orig\sim x\right\} \\
        \Fcal_2 :=& \left\{\orig\not\sim x\right\}
        \cap\left\{\#\left\{u\in\eta\colon \orig\sim u\sim x\right\}\geq 2\right\}\\
        \Fcal_3 :=& \left\{\orig\not\sim x\right\}
        \cap\left\{\#\left\{u\in\eta\colon \orig\sim u\sim x\right\}= 1\right\}
        \cap \left\{\#\left\{v\in\eta\colon \adja{\orig}{v}{\xi^{\orig}}, \xconn{v}{x}{\xi^{v,x}_{\thinn{\orig}}}{=2}\right\}\geq 1\right\}.
    \end{align}
    In each, either $\orig$ is adjacent to $x$ or there exist two vertex disjoint paths from $\orig$ to $x$. Therefore $\Fcal\subset \left\{\dconn{\orig}{x}{\xi^{\orig,x}}\right\}$. The components $\Fcal_1$, $\Fcal_2$, and $\Fcal_3$ are also all disjoint by construction, so
    \begin{equation}
        \pla\left(\dconn{\orig}{x}{\xi^{\orig,x}}\right) \geq \pla\left(\Fcal_1\right) + \pla\left(\Fcal_2\right) + \pla\left(\Fcal_3\right).
    \end{equation}
    Since $\eta$ is distributed as a Poisson point process on $\Rd$ with intensity $\lambda$,
    \begin{align}
        \pla\left(\Fcal_1\right) =& \connf(x)\\
        \pla\left(\Fcal_2\right) =& \left(1-\connf(x)\right)\left(1-\exp\left(-\lambda\connf^{\star 2}(x)\right) - \lambda\connf^{\star 2}(x)\exp\left(-\lambda\connf^{\star 2}(x)\right)\right)\\
        \pla\left(\Fcal_3\right) =& \lambda\connf^{\star 2}(x)\exp\left(-\lambda\connf^{\star 2}(x)\right)\connf^{[3]}(x)\nonumber\\
        =& \lambda\connf^{\star 2}(x)\exp\left(-\lambda\connf^{\star 2}(x)\right)\left(1-\connf(x)\right)\nonumber\\
        &\times\left(1 - \exp\left(-\lambda\int \connf(v)\left(1-\connf(x-v)\right)\left(1 - \exp\left(-\lambda\int \connf(x-w)\left(1-\connf(w)\right)\connf(w-v)\dd w\right)\right)\dd v\right)\right).
    \end{align}
    Therefore
    \begin{equation}
        \fLacelamC^{(0)}(0) \geq \int \left(\mathbb{P}_{\lambda_c}\left(\Fcal_2\right) + \mathbb{P}_{\lambda_c}\left(\Fcal_3\right)\right)\dd x,
    \end{equation}
    and we now want to lower bound the integrals of $\mathbb{P}_{\lambda_c}\left(\Fcal_2\right)$ and $\mathbb{P}_{\lambda_c}\left(\Fcal_3\right)$.
    
    By using $1 - \e^{-x} - x\e^{-x} \geq \frac{1}{2}x^2 - \frac{1}{2}x^3$ for all $x\in\R$,
    \begin{equation}
        \lambda\int\pla\left(\Fcal_2\right)\dd x \geq \frac{1}{2}\lambda^3\FtwoPtOne - \frac{1}{2}\lambda^4\FtwoPtTwo = \frac{1}{2}\lambda^3\loopfour - \frac{1}{2}\lambda^3\fourcrosstwo + \LandauBigO{\phiTwoTwoTwo}
    \end{equation}
    We find our lower bound on $\lambda\int\pla\left(\Fcal_3\right)\dd x$ in a few more steps. Since we have $x\e^{-x} \geq x - x^2$ for all $x\in\R$,
    \begin{equation}
    \label{eqn:F3Pt1}
        \lambda\connf^{\star 2}(x)\exp\left(-\lambda\connf^{\star 2}(x)\right)\left(1-\connf(x)\right) \geq \lambda\FthreePtOne{\orig}{x} - \lambda^2\FthreePtTwo{\orig}{x}.
    \end{equation}
    Since we have $1-\e^{-x} \geq x - \frac{1}{2}x^2$ for all $x\in\R$,
    \begin{equation}
        1 - \exp\left(-\lambda\int \connf(x-w)\left(1-\connf(w)\right)\connf(w-v)\dd w\right) \geq \lambda\FthreePtThree{\orig}{x}{v} -\frac{1}{2}\lambda^2\FthreePtFour{\orig}{x}{v},
    \end{equation}
    and
    \begin{multline}
        \lambda\int \connf(v)\left(1-\connf(x-v)\right)\left(1 - \exp\left(-\lambda\int \connf(x-w)\left(1-\connf(w)\right)\connf(w-v)\dd w\right)\right)\dd v \\
        \geq \lambda^2\FthreePtFive{\orig}{x} - \frac{1}{2}\lambda^3\FthreePtSix{\orig}{x}.
    \end{multline}
    Since $x\mapsto 1-\e^{-x}$ is monotone increasing and $1-\e^{-x} \geq x - \frac{1}{2}x^2$ for all $x\in\R$,
    \begin{multline}
    \label{eqn:F3Pt2}
        1 - \exp\left(-\lambda\int \connf(v)\left(1-\connf(x-v)\right)\left(1 - \exp\left(-\lambda\int \connf(x-w)\left(1-\connf(w)\right)\connf(w-v)\dd w\right)\right)\dd v\right)\\
        \geq \lambda^2\FthreePtFive{\orig}{x} - \frac{1}{2}\lambda^3\FthreePtSix{\orig}{x} - \frac{1}{2}\lambda^4\FthreePtSeven{\orig}{x} + \frac{1}{2}\lambda^5\FthreePtEight{\orig}{x} - \frac{1}{8}\lambda^6\FthreePtNine{\orig}{x}.
    \end{multline}
    When we combine \eqref{eqn:F3Pt1} and \eqref{eqn:F3Pt2} and integrate over $x$, we see that many integrals can be bounded by integrals of two loops. These terms will be $\LandauBigO{\phiThreeTwoOne+\phiTwoTwoTwo}$. Therefore
    \begin{multline}
        \lambda\int\pla\left(\Fcal_3\right)\dd x \geq \lambda^4\FthreePtTen + \LandauBigO{\phiThreeTwoOne+\phiTwoTwoTwo} \\= \lambda^4\loopfive + \LandauBigO{\phiThreeTwoOne+\phiTwoTwoTwo}.
    \end{multline}
    We then have a lower bound on $\fLacelam^{(0)}(0)$ for any $\lambda>0$, and this gives the required result.
\end{proof}

\subsection{Bounds on the First Lace Expansion Coefficient}
\label{sec:Pi1bd}
In this subsection we prove \eqref{eqPi1bd}. 

\paragraph{Upper Bound on $\fLacelamC^{(1)}(0)$}

\begin{lemma}
\label{lem:Pi1_UpperBound}
Suppose Assumption~\ref{Assumption} holds. Then as $d\to\infty$,
    \begin{multline}
    \lambda_c\fLacelamC^{(1)}(0) \leq\lambda_c^2\loopthree + 2\lambda_c^3\loopfour + 3\lambda_c^4\loopfive - 2\lambda_c^3\fourcrossone \\ + \LandauBigO{\loopsix + \phiTwoTwoTwo + \phiThreeTwoOne}
\end{multline}
\end{lemma}
    
We borrow from \cite{HeyHofLasMat19} in bounding $\pla \left( \{\dconn{\orig}{u}{\xi^{\orig, u}_{0}}\} \cap E\left(u,x; \C_{0}, \xi^{u, x}_{1}\right) \right)$, but we need to make refinements so that our lower bound will match the upper bound at the precision we are interested in. We begin by bounding $\{\dconn{\orig}{u}{\xi^{\orig, u}_{0}}\} \cap E\left(u,x; \C_{0}, \xi^{u, x}_{1}\right)$ by a slightly different event.

\begin{definition}
    Let $\xi_0,\xi_1$ be independent instances of the random graph with locally finite vertex sets $\eta_0$ and $\eta_1$.
    \begin{itemize}
    \item Let $\left\{\sqconn{u}{x}{(\xi_0, \xi_1)}\right\}$ denote the event that $u\in\eta_0$ and $x\in\eta_1$, but that $x$ does not survive a $\C\left(u, \xi^{u}_0\right)$-thinning of $\eta_1$. 
    \item Let $m\in\N$ and $\vec x, \vec y \in \left(\Rd\right)^m$. We define $\bigcirc_m^\leftrightarrow((x_j, y_j)_{1 \leq j \leq m}; \xi)$ as the event that $\{\conn{x_j}{y_j}{\xi}\}$ occurs for every $1 \leq j \leq m$ with the additional requirement that every point in $\eta$ is the interior vertex of at most one of the $m$ paths, and none of the $m$ paths contains an interior vertex in the set $\left\{x_j\colon j\in[m]\right\} \cup \left\{y_j: j\in [m]\right\}$.
    \item Let $\bigcirc_m^\sqarrow\left( (x_j,y_j)_{1 \leq j \leq m}; (\xi_0,\xi_1)\right)$ be the intersection of the following two events. Firstly, that $\bigcirc_{m-1}^\leftrightarrow\left((x_j,y_j)_{1 \leq j <m};\xi_0\right)$ occurs but no path uses $x_{m}$ or $y_{m}$ as an interior vertex. Secondly, that $\{\sqconn{x_{m}}{y_{m}}{(\xi_0[\eta_0\setminus  \{x_i, y_i\}_{1 \leq i <m}],\xi_1)}\}$ occurs in such a way that at least one point $z$ in $\xi_0$ that is responsible for thinning out $y_m$ is connected to $x_m$ by a path $\gamma$ so that $z$ as well as all interior vertices of $\gamma$ are not contained in any path of the $\bigcirc_{m-1}^\leftrightarrow((x_j,y_j)_{1 \leq j <m};\xi_0)$ event.
\end{itemize}
    Now let $t,u,w,x,z\in\Rd$. Then define
    \begin{align}
        F^{(1)}_0\left(w,u,z;\xi_0,\xi_1\right) &:= \left\{\notadja{\orig}{u}{\xi_0}\right\}\cap\bigcirc_4^\sqarrow\left(\left(\orig,u\right),\left(\orig,w\right),\left(u,w\right),\left(w,z\right);\left(\xi_0,\xi_1\right)\right)\\
        F^{(2)}_0\left(w,u,z;\xi_0,\xi_1\right) &:= \left\{w=\orig\right\}\cap \left\{\adja{\orig}{u}{\xi_1}\right\}\cap\left\{\sqconn{w}{z}{\left(\xi_0\setminus\{u\},\xi_1\right)}\right\}\\
        F_1^{(1)}\left(u,t,z,x;\xi_1\right) &:= \left\{\#\left\{t,z,x\right\}=3\right\}\cap \bigcirc_4^\leftrightarrow\left(\left(u,t\right),\left(t,z\right),\left(t,x\right),\left(z,x\right);\xi_1\right)\cap\left\{\notadja{t}{x}{\xi_1}\right\}\\
        F_1^{(2)}\left(u,t,z,x;\xi_1\right) &:=\left\{t=z=x\right\}\cap\left\{\conn{u}{x}{\xi_1}\right\}.
    \end{align}
    Also let $F_0 := F^{(1)}_0 \cup F^{(2)}_0$ and $F_1 := F^{(1)}_1 \cup F^{(2)}_1$.
\end{definition}

\begin{lemma}
    Let $x,u\in\Rd$ be distinct points. Then
    \begin{equation}
        \Id_{\left\{\dconn{\orig}{u}{\xi^{\orig, u}_{0}}\right\}}\Id_{E\left(u,x; \C_{0}, \xi^{u, x}_{1}\right)} \leq \sum_{z\in\eta_1^x}\left(\sum_{w\in \eta_0^\orig}\Id_{F_0\left(w,u,z;\xi_0^{\orig,u},\xi_1^{u,x}\right)}\right)\left(\sum_{t\in\eta_1^{u,x}}\Id_{F_1\left(u,t,z,x;\xi_1^{u,x}\right)}\right).
    \end{equation}
\end{lemma}

\begin{proof}
    We first prove that
    \begin{equation}
    \label{eqn:F-bounds_FirstStep}
        \Id_{E\left(u,x; \C_{0}, \xi^{u, x}_{1}\right)} \leq \sum_{z\in\eta_1^x}\sum_{t\in\eta_1^{u,x}}\Id_{F_1\left(u,t,z,x;\xi_1^{u,x}\right)}\Id_{\left\{\sqconn{\orig}{z}{\left(\xi_0^{\orig},\xi_1^{u,x}\right)}\right\}}.
    \end{equation}
    Note that the event $E\left(u,x; \C_{0}, \xi^{u, x}_{1}\right)$ is contained in the event that $u$ is connected to $x$ and that this connection fails after a $\C_0$-thinning of $\eta_1^x$. There are two cases under which this can happen.
    
    Case (a): The point $x$ itself is thinned out. In this case
    \begin{multline}
        E\left(u,x; \C_{0}, \xi^{u, x}_{1}\right) \subset \left\{\conn{u}{x}{\xi_1^{u,x}}\right\}\cap\left\{\sqconn{\orig}{x}{\left(\xi_0^{\orig},\xi_1^{u,x}\right)}\right\} \\= F^{(2)}_1\left(u,x,x,x;\xi_1^{u,x}\right) \cap\left\{\sqconn{\orig}{x}{\left(\xi_0^{\orig},\xi_1^{u,x}\right)}\right\}.
    \end{multline}

    Case (b): The point $x$ is not thinned out. This implies that there is at least one interior point on the path between $u$ and $x$, and that at least one of these interior points is thinned out by $\C_0$. Let $t$ be the last pivotal point in $\piv{u,x;\xi_1^{u,x}}$, and set $t=u$ if $\piv{u,x;\xi_1^{u,x}}=\emptyset$. Since $t$ is the last pivotal point (or there are no pivotal points), we have $\left\{\dconn{t}{x}{\xi_1^x}\right\}$. The event $E\left(u,x; \C_{0}, \xi^{u, x}_{1}\right)$ implies that all the paths from $t$ to $x$ fail after a $\C_0$-thinning, but that $t$ is not thinned out. We can pick any thinned out point on a path from $t$ to $x$ to be our $z$, while noting that $t$ and $x$ cannot be adjacent. Therefore this case corresponds to the possible occurrences of $F^{(1)}_1$, and we have proven \eqref{eqn:F-bounds_FirstStep}.

    Now it only remains to prove that
    \begin{equation}
        \Id_{\left\{\dconn{\orig}{u}{\xi^{\orig, u}_{0}}\right\}}\Id_{\left\{\sqconn{\orig}{z}{\left(\xi_0^{\orig},\xi_1^{u,x}\right)}\right\}} \leq \sum_{w\in \eta_0^\orig}\Id_{F_0\left(w,u,z;\xi_0^{\orig,u},\xi_1^{u,x}\right)}.
    \end{equation}
    The event $\left\{\sqconn{\orig}{z}{\left(\xi_0^{\orig},\xi_1^{u,x}\right)}\right\}$ implies that there exists at least one point in $\C_0$ that is responsible for thinning out $z$. Let $\gamma$ denote a the path from $\orig$ to this point in $\C_0$. We once again now have two cases to consider.
    
    Case (a): $\notadja{\orig}{u}{\xi^{\orig,u}_0}$. Then $\left\{\dconn{\orig}{u}{\xi^{\orig, u}_{0}}\right\}$ implies that there exist two disjoint paths (denoted $\gamma'$ and $\gamma''$) from $\orig$ to $u$. Both of these paths are necessarily of length greater than or equal to $2$. Let $w$ denote the last vertex $\gamma$ shares with either $\gamma'$ or $\gamma''$ (allowing for the possibility that $w=\orig$). Requiring that $\gamma$, $\gamma'$, and $\gamma''$ exist results precisely in the event $F_0^{(1)}$.

    Case (b): $\adja{\orig}{u}{\xi^{\orig,u}_0}$. Now we fix $w=\orig$ immediately. The existence of the path from $\orig$ to the thinning point implies the event $F_0^{(2)}$.
\end{proof}

\begin{definition}[The $\psi$ functions] \label{def:DB:psi_functions}
Let $r,s,u,w,x,y\in\Rd$ and $n\geq 1$. We first set $\tlamo(x) := \lambda^{-1}\delta_{x,\orig} + \tlam(x)$ and $\tlam^{(\geq 2)}(x) := \pla\left(\xconn{\orig}{x}{\xi^{\orig,x}}{\geq 2}\right) = \tlam(x) - \connf(x)$. Also define
\begin{align*}
    \psi_0^{(1)}(w,u) &:= \lambda^2\tlam^{(\geq 2)}(u)\tlam(u-w)\tlam(w),\\
    \psi_0^{(2)}(w,u) &:=\lambda^2 \delta_{w,\orig} \tlam^{(\geq 2)}(u)\int \tlam(u-t)\tlam(t) \dd t,\\
    \psi_0^{(3)}(w,u) &:= \lambda\connf(u) \delta_{w,\orig},\\
    \psi^{(1)}(w,u,r,s) &:= \lambda^4\tlam(w-u)\int \tlamo(t-s) \tlam(t-w)\tlam(u-z)\tlam(z-t)\tlam(z-r)\dd z \dd t, \\
    \psi^{(2)}(w,u,r,s) &:= \lambda^4\tlamo(w-s)\int \tlam(t-z)\tlam(z-u)\tlam(u-t)\tlamo(t-w)\tlam(z-r)\dd z \dd t, \\
	\psi^{(3)}(w,u,r,s) &:= \lambda^2\tlam(u-w)\tlam(w-s)\tlam(u-r),\\
	\psi^{(4)}(w,u,r,s) &:= \lambda\delta_{w,s}\tlam(u-w)\tlam(u-r),\\
	\psi_n^{(1)} (x,r,s) &:= \lambda^3\int \tlamo(t-s)\tlam(z-r)\tlam(t-z)\tlam(z-x)\tlam^{(\geq 2)}(x-t)\dd z\dd t,\\
    \psi_n^{(2)}(x,r,s) &:=\lambda \tlam(x-s)\tlam(x-r),
\end{align*}
and set $\psi_0 := \psi_0^{(1)}+\psi_0^{(2)}+\psi_0^{(3)}$, $\psi_n := \psi_n^{(1)} + \psi_n^{(2)}$, and $\psi:= \psi^{(1)}+\psi^{(2)}+\psi^{(3)} + \psi^{(4)}$.
\end{definition}

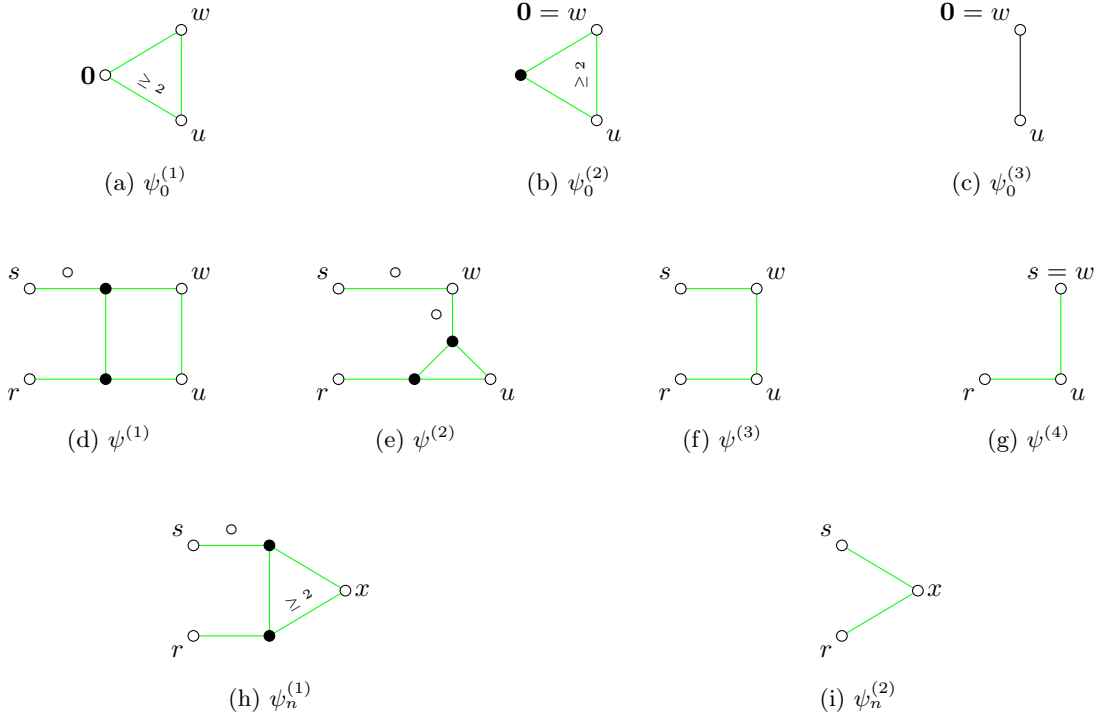
\begin{figure}
    \centering
    \begin{subfigure}[b]{0.3\textwidth}
    \centering
        \begin{tikzpicture}
        \draw[green] (0,0) -- (1,0.6) -- (1,-0.6) -- cycle;
        \draw (0.5,-0.3) circle (0pt) node[rotate=-31,above]{\tiny $\geq 2$};
        \filldraw[fill=white] (0,0) circle (2pt) node[left]{$\orig$};
        \filldraw[fill=white] (1,0.6) circle (2pt) node[above right]{$w$};
        \filldraw[fill=white] (1,-0.6) circle (2pt) node[below right]{$u$};
        \end{tikzpicture}
    \caption{$\psi^{(1)}_0$}
    \end{subfigure}
    \hfill
    \begin{subfigure}[b]{0.3\textwidth}
    \centering
        \begin{tikzpicture}
        \draw[green] (0,0) -- (1,0.6) -- (1,-0.6) -- cycle;
        \draw (1,0) circle (0pt) node[rotate=90,above]{\tiny $\geq 2$};
        \filldraw (0,0) circle (2pt);
        \filldraw[fill=white] (1,0.6) circle (2pt) node[above left]{$\orig=w$};
        \filldraw[fill=white] (1,-0.6) circle (2pt) node[below right]{$u$};
       \end{tikzpicture}
    \caption{$\psi^{(2)}_0$}
    \end{subfigure}
    \hfill
    \begin{subfigure}[b]{0.3\textwidth}
    \centering
        \begin{tikzpicture}
        \draw (0,0.6) -- (0,-0.6);
        \filldraw[fill=white] (0,0.6) circle (2pt) node[above left]{$\orig=w$};
        \filldraw[fill=white] (0,-0.6) circle (2pt) node[below right]{$u$};
        \end{tikzpicture}
    \caption{$\psi^{(3)}_0$}
    \end{subfigure}
    \begin{subfigure}[b]{0.24\textwidth}
    \centering
        \begin{tikzpicture}
        \draw[green] (0,0.6) -- (1,0.6) -- (1,-0.6) -- (0,-0.6);
        \draw (0.5,0.6) circle (0pt) node[above]{$\circ$};
        \draw[green] (1,0.6) -- (2,0.6) -- (2,-0.6) -- (1,-0.6);
        \draw (0,1.75) circle (0pt);
        \filldraw[fill=white] (0,0.6) circle (2pt) node[above left]{$s$};
        \filldraw[fill=white] (0,-0.6) circle (2pt) node[below left]{$r$};
        \filldraw (1,0.6) circle (2pt);
        \filldraw (1,-0.6) circle (2pt);
        \filldraw[fill=white] (2,0.6) circle (2pt) node[above right]{$w$};
        \filldraw[fill=white] (2,-0.6) circle (2pt) node[below right]{$u$};
        \end{tikzpicture}
    \caption{$\psi^{(1)}$}
    \end{subfigure}
    \hfill
    \begin{subfigure}[b]{0.24\textwidth}
    \centering
        \begin{tikzpicture}
        \draw[green] (0,0.6) -- (1.5,0.6) -- (1.5,-0.1) -- (1,-0.6);
        \draw (0.75,0.6) circle (0pt) node[above]{$\circ$};
        \draw (1.5,0.25) circle (0pt) node[left]{$\circ$};
        \draw[green] (0,-0.6) -- (1,-0.6) -- (2,-0.6) -- (1.5,-0.1);
        \filldraw[fill=white] (0,0.6) circle (2pt) node[above left]{$s$};
        \filldraw[fill=white] (0,-0.6) circle (2pt) node[below left]{$r$};
        \filldraw[fill=white] (1.5,0.6) circle (2pt) node[above right]{$w$};
        \filldraw (1,-0.6) circle (2pt);
        \filldraw (1.5,-0.1) circle (2pt);
        \filldraw[fill=white] (2,-0.6) circle (2pt) node[below right]{$u$};
        \end{tikzpicture}
    \caption{$\psi^{(2)}$}
    \end{subfigure}
    \hfill
    \begin{subfigure}[b]{0.24\textwidth}
    \centering
        \begin{tikzpicture}
        \draw[green] (0,0.6) -- (1,0.6) -- (1,-0.6) -- (0,-0.6);
        \filldraw[fill=white] (0,0.6) circle (2pt) node[above left]{$s$};
        \filldraw[fill=white] (0,-0.6) circle (2pt) node[below left]{$r$};
        \filldraw[fill=white] (1,0.6) circle (2pt) node[above right]{$w$};
        \filldraw[fill=white] (1,-0.6) circle (2pt) node[below right]{$u$};
        \end{tikzpicture}
    \caption{$\psi^{(3)}$}
    \end{subfigure}
    \hfill
    \begin{subfigure}[b]{0.24\textwidth}
    \centering
        \begin{tikzpicture}
        \draw[green] (1,0.6) -- (1,-0.6) -- (0,-0.6);
        \filldraw[fill=white] (0,-0.6) circle (2pt) node[below left]{$r$};
        \filldraw[fill=white] (1,0.6) circle (2pt) node[above]{$s=w$};
        \filldraw[fill=white] (1,-0.6) circle (2pt) node[below right]{$u$};
        \end{tikzpicture}
    \caption{$\psi^{(4)}$}
    \end{subfigure}
    \hspace*{\fill}%
    \begin{subfigure}[b]{0.45\textwidth}
    \centering
        \begin{tikzpicture}
        \draw[green] (0,0.6) -- (1,0.6) -- (1,-0.6);
        \draw (0.5,0.6) circle (0pt) node[above]{$\circ$};
        \draw[green] (1,0.6) -- (2,0) -- (1,-0.6) -- (0,-0.6);
        \draw (1.5,-0.3) circle (0pt) node[rotate=31,above]{\tiny $\geq 2$};
        \filldraw[fill=white] (0,0.6) circle (2pt) node[above left]{$s$};
        \filldraw[fill=white] (0,-0.6) circle (2pt) node[below left]{$r$};
        \filldraw (1,0.6) circle (2pt);
        \filldraw (1,-0.6) circle (2pt);
        \filldraw[fill=white] (2,0) circle (2pt) node[right]{$x$};
        \end{tikzpicture}
    \caption{$\psi^{(1)}_n$}
    \end{subfigure}
    \hfill
    \begin{subfigure}[b]{0.45\textwidth}
    \centering
        \begin{tikzpicture}
        \draw[green] (1,0.6) -- (2,0) -- (1,-0.6);
        \draw (0,1.75) circle (0pt);
        \filldraw[fill=white] (1,0.6) circle (2pt) node[above left]{$s$};
        \filldraw[fill=white] (1,-0.6) circle (2pt) node[below left]{$r$};
        \filldraw[fill=white] (2,0) circle (2pt) node[right]{$x$};
        \end{tikzpicture}
    \caption{$\psi^{(2)}_n$}
    \end{subfigure}
    \hspace*{\fill}%
    \caption{Diagrams of the $\psi_0$, $\psi$, and $\psi_n$ functions.}
    \label{fig:psiFunctions}
\end{figure}

For our bounds on $\fLacelamC^{(1)}(0)$ we will only require $\psi_0$ and $\psi_1$. Later we will also use $\psi$ to bound $\fLacelamC^{(n)}(0)$ for $n\geq 2$.

\begin{lemma}
\label{lem:Pi1UpperStep1}
\begin{multline}
    \lambda_c\fLacelamC^{(1)}(0) \leq \int \psi_0\left(w,u\right)\psi_1\left(x,w,u\right)\dd u\dd w\dd x \\= \lambda_c^2\int \connf(u)\tlamC(x)\tlamC(u-x)\dd u\dd x + \LandauBigO{\loopsix  + \phiThreeTwoOne + \phiTwoTwoTwo}.
\end{multline}
\end{lemma}

\begin{proof}
    The first inequality is proven in very nearly exactly the same way as \cite[Proposition~7.2]{HeyHofLasMat19}. The first difference is that our event $F^{(1)}_1$ has the intersection with $\left\{\notadja{t}{x}{\xi_1}\right\}$. Since the event $\bigcirc_4^\leftrightarrow\left(\left(u,t\right),\left(t,z\right),\left(t,x\right),\left(z,x\right);\xi_1\right)$ ensures that $t$ and $x$ are connected in $\xi_1$, this means that the event $\xconn{t}{x}{\xi_1}{\geq 2}$ occurs. This then manifests in the end result as the occurrence of a $\tlam^{(\geq 2)}$ function rather than a $\tlam$ function in the integral in $\psi^{(1)}_1$. Similarly, the event $F^{(1)}_0$ implies that $\xconn{\orig}{u}{\xi_0}{\geq 2}$, and this results in the $\tlam^{(\geq 2)}(u)$ appearing rather than $\tlam(u)$ in $\psi^{(1)}_0$ and $\psi^{(2)}_0$.

    For the equality, we first note that 
    \begin{equation}
        \int \psi^{(3)}_0\left(w,u\right)\psi^{(2)}_1\left(x,w,u\right)\dd u\dd w\dd x = \lambda^2\int \connf(u)\tlam(x)\tlam(u-x)\dd u\dd x.
    \end{equation}
    Our task in then to show that all the other terms $\int \psi^{(j_0)}_0\left(w,u\right)\psi^{(j_1)}_1\left(x,w,u\right)\dd u\dd w\dd x$ are error terms. To make it clearer what we are trying to bound, we present the integral $\int \psi_0\left(w,u\right)\psi_1\left(x,w,u\right)\dd u\dd w\dd x$ diagrammatically:
    \begin{multline}
        \lambda\fLacelam^{(1)}(0) \leq \lambda^5\PiOnePtOne + \lambda^3\PiOnePtTwo + \lambda^5\PiOnePtThree \\
        +\lambda^3\PiOnePtFour + \lambda^4\PiOnePtFive + \lambda^2\PiOnePtSix.
    \end{multline}
    The last of these six diagrams will be the only relevant one for our level of precision. To demonstrate how we bound these other five, we examine the second:
    \begin{equation}
        \lambda^3\PiOnePtTwo = \lambda^3\int\tlam^{(\geq 2)}(u)\tlam(w)\tlam(w-u)\tlam(x-w)\tlam(x-u)\dd u\dd w\dd x.
    \end{equation}
    First we expand $\tlam^{(\geq 2)}(u)\leq \lambda\connf^{\star 2}(u) + \lambda^2\connf^{\star 3}(u) + \lambda^3\connf^{\star 4}(u) + \lambda^4\connf^{\star 5}(u) + \lambda^5\connf^{\star 6}(u) + \lambda^6\connf^{\star 6}\star \tlam(u)$. For the two diagrams that result from the last two terms in this expansion, we can bound $\tlam(w-u)\leq 1$ to get terms of the form 
    $\lambda^{j+5}\connf^{\star 6}\star\tlam^{\star j}\left(\orig\right)$ for $j\in\left\{3,4\right\}$. From Lemma~\ref{lem:tailbound}, both of these are $\LandauBigO{\connf^{\star 6}\left(\orig\right)}$ when $\lambda\leq \lambda_c$. For the remaining diagrams we then bound $\tlam(w) \leq \connf(w) + \lambda\connf^{\star 2}(w) + \lambda^2\connf^{\star 3}(w) + \lambda^3\connf^{\star 4}(w) + \lambda^4\connf^{\star 4}\star \tlam(w)$ and if the diagrams contain a loop of at least six $\connf$ functions and maybe some $\tlam$ functions, we once again bound $\tlam(w-u)\leq 1$ and use Lemma~\ref{lem:tailbound} to show that they are $\LandauBigO{\connf^{\star 6}\left(\orig\right)}$. For the remaining diagrams we bound $\tlam(x-w) \leq \connf(x-w) + \lambda\connf^{\star 2}(x-w) + \lambda^2\connf^{\star 3}(x-w) + \lambda^3\connf^{\star 3}\star \tlam(x-w)$. Again bounding $\tlam(w-u)\leq 1$ allows us to use Lemma~\ref{lem:tailbound} to show that some of these diagrams are $\LandauBigO{\connf^{\star 6}\left(\orig\right)}$. After bounding $\tlam(x-u) \leq \connf(x-u) + \lambda\connf^{\star 2}(x-u) + \lambda^2\connf^{\star 2}\star\tlam(x-u)$ and showing that some terms are $\LandauBigO{\connf^{\star 6}\left(\orig\right)}$, we arrive at
    \begin{equation}
        \lambda_c^3\PiOnePtTwo \leq \lambda_c^4\int \connf^{\star 2}(u)\connf(w)\tlamC(w-u)\connf(x-w)\connf(x-u)\dd u\dd w\dd x + \LandauBigO{\connf^{\star 6}\left(\orig\right)}.
    \end{equation}
    Then we bound $\tlam(w-u)\leq \connf(w-u) + \lambda\connf^{\star 2}(w-u) + \lambda^2\connf^{\star 3}(w-u) + \lambda^3\connf^{\star 3}\star\tlam(w-u)$ to get
    \begin{align}
        &\lambda^4\int \connf^{\star 2}(u)\connf(w)\tlam(w-u)\connf(x-w)\connf(x-u)\dd u\dd w\dd x \nonumber\\
        &\hspace{4cm}\leq \lambda^4\int \connf^{\star 2}(u)\connf(w)\connf(w-u)\connf(x-w)\connf(x-u)\dd u\dd w\dd x \nonumber\\
        &\hspace{5cm}+ \lambda^5\int \connf^{\star 2}(u)\connf(w)\connf^{\star 2}(w-u)\connf(x-w)\connf(x-u)\dd u\dd w\dd x \nonumber\\
        &\hspace{5cm}+ \lambda^6\int \connf^{\star 2}(u)\connf(w)\connf^{\star 3}(w-u)\connf(x-w)\connf(x-u)\dd u\dd w\dd x\nonumber\\
        &\hspace{5cm}+ \lambda^7\int \connf^{\star 2}(u)\connf(w)\connf^{\star 3}\star\tlam(w-u)\connf(x-w)\connf(x-u)\dd u\dd w\dd x.
    \end{align}
    From the commutativity of convolution, observe that the first two terms are $\LandauBigO{\int\connf(x)\connf^{\star 2}(x)\connf^{\star 3}(x)\dd x}$. For the last two terms we bound $\int\connf(x-w)\connf(x-u)\dd x \leq \int\connf(x-w)\dd x = 1$ for all $u,w\in\Rd$. Therefore we can apply Lemma~\ref{lem:tailbound} to show that these diagrams are $\LandauBigO{\connf^{\star 6}\left(\orig\right)}$ when $\lambda\leq \lambda_c$. In summary, we have
    \begin{equation}
        \lambda_c^3\PiOnePtTwo = \LandauBigO{\phiThreeTwoOne + \loopsix}.
    \end{equation}
    Repeating these ideas for the other diagrams produces
    \begin{align}
        \lambda_c^5\PiOnePtOne &= \LandauBigO{\loopsix},\\
        \lambda_c^5\PiOnePtThree &= \LandauBigO{\loopsix},\\
        \lambda_c^3\PiOnePtFour &= \LandauBigO{\phiTwoTwoTwo + \loopsix},\\
        \lambda_c^4\PiOnePtFive &= \LandauBigO{\phiThreeTwoOne + \loopsix}.
    \end{align}
\end{proof}

\begin{lemma}
    Let $x,u\in\Rd$. Then
    \begin{align}
    \label{eqn:TauTauExpansion}
        \tlam(x)\tlam(u-x)&\leq \connf(x)\connf(u-x) + \connf(x)\Exconnf{2}(u-x) + \Exconnf{2}(x)\connf(u-x) + \connf(x)\Exconnf{3}(u-x) + \Exconnf{3}(x)\connf(u-x)\nonumber\\
        &\qquad+ \Exconnf{2}(x)\Exconnf{2}(u-x) + \lambda^3\connf(x)\connf^{\star 4}(u-x) + \lambda^5\connf(x)\connf^{\star 4}\star\tlam(u-x)\nonumber\\
        &\qquad + \lambda^3\connf^{\star 2}(x)\connf^{\star 3}(u-x) + \lambda^4\connf^{\star 2}(x)\connf^{\star 3}\star\tlam(u-x) + \lambda^4\connf^{\star 3}(x)\connf^{\star 2}(u-x)  \nonumber\\
        &\qquad+ \lambda^4\connf^{\star 3}(x)\connf^{\star 2}\star\tlam(u-x) + \lambda^3\connf^{\star 4}(x)\connf(u-x) + \lambda^4\connf^{\star 4}(x)\connf\star\tlam(u-x)\nonumber\\
        &\qquad + \lambda^4\connf^{\star 4}\star\tlam(x)\connf(u-x)+ \lambda^5\connf^{\star 4}\star\tlam(x)\connf\star\tlam(u-x).
    \end{align}
    Therefore
    \begin{align}
        \lambda_c^2\connf\star\tlamC^{\star 2}\left(\orig\right) &\leq \lambda_c^2\connf^{\star 3}\left(\orig\right) + 2\lambda_c^2\connf^{\star 2}\star\Exconnf{2}\left(\orig\right) + 2\lambda_c^2\connf^{\star 2}\star\Exconnf{3}\left(\orig\right) + \lambda_c^2\connf\star\Exconnf{2}\star\Exconnf{2}\left(\orig\right) + \LandauBigO{\connf^{\star 6}\left(\orig\right)}\nonumber\\
        &\leq \lambda_c^2\loopthree + 2\lambda_c^3\loopfour + 3\lambda_c^4\loopfive - 2\lambda_c^3\fourcrossone + \LandauBigO{\loopsix + \phiTwoTwoTwo}.
    \end{align}
\end{lemma}
\begin{proof}
    Equation \eqref{eqn:TauTauExpansion} follows from applying Lemma~\ref{lem:tauUpperbound} to both $\tau(x)$ and $\tau(u-x)$ with $n= 4$, and then bounding $\Exconnf{m} \leq \lambda^{m-1}\connf^{\star m}$ in some places.

    For the second part, we use \eqref{eqn:TauTauExpansion} in conjunction with Lemma~\ref{lem:tailbound} to show that many of the terms are $\LandauBigO{\connf^{\star 6}\left(\orig\right)}$ and produce the first inequality. We then immediately have $\lambda^2\connf^{\star 3}\left(\orig\right) = \lambda^2 \loopthree$, and simply bounding $\Exconnf{3}\leq \lambda^2\connf^{\star 3}$ and $\Exconnf{2}\leq \lambda\connf^{\star 2}$ gives
    \begin{equation}
        2\lambda^2\connf^{\star 2}\star\Exconnf{3}\left(\orig\right) + \lambda^2\connf\star\Exconnf{2}\star\Exconnf{2} \left(\orig\right) \leq 3\lambda^4\loopfive.
    \end{equation}
    To bound $2\lambda^2\connf^{\star 2}\star\Exconnf{2}\left(\orig\right)$ appropriately requires a little more care though. Recall from \eqref{eqn:ExConnf2} that $\connf^{[2]}(x) = \left(1-\connf(x)\right)\left(1-\exp\left(-\lambda\connf^{\star 2}(x)\right)\right)$. Using $1-\e^{-x} \leq x - \frac{1}{2}x^2 + \frac{1}{6}x^3$ for all $x\in\R$ then gives
    \begin{multline}
        \lambda^2\connf^{\star 2}\star\Exconnf{2}\left(\orig\right) \leq \lambda^3\int\left(1-\connf(x)\right)\connf^{\star 2}(x)^2\dd x \\-\frac{1}{2}\lambda^4\int\left(1-\connf(x)\right)\connf^{\star 2}(x)^3\dd x + \frac{1}{6}\lambda^5\int\left(1-\connf(x)\right)\connf^{\star 2}(x)^4\dd x.
    \end{multline}
    The second term we can safely neglect, and for the third term we use $1-\connf(x)\leq 1$ and $\connf^{\star 2}(x)\leq \int\connf(x)\dd x=1$ for all $x\in\Rd$. This produces
    \begin{multline}
        \lambda^2\connf^{\star 2}\star\Exconnf{2}\left(\orig\right) \leq \lambda^3\int\left(1-\connf(x)\right)\connf^{\star 2}(x)^2\dd x + \frac{1}{6}\lambda^5\int\connf^{\star 2}(x)^3\dd x \\= \lambda^3\loopfour - \lambda^3\fourcrossone + \LandauBigO{\phiTwoTwoTwo},
    \end{multline}
    as required.
\end{proof}

This concludes the proof of Lemma~\ref{lem:Pi1_UpperBound}.

\paragraph{Lower Bound on $\fLacelamC^{(1)}(0)$}

\begin{lemma}
    \begin{equation}
        \lambda_c\fLacelamC^{(1)}(0) \geq \lambda_c^2\loopthree + 2\lambda_c^3\loopfour + 3\lambda_c^4\loopfive - 2\lambda_c^3\fourcrossone + \LandauBigO{\phiTwoTwoTwo + \phiThreeTwoOne}
    \end{equation}
\end{lemma}

\begin{proof}
Our strategy is to identify disjoint events contained in $\{\dconn{\orig}{u}{\xi^{\orig, u}_{0}}\} \cap E\left(u,x; \C_{0}, \xi^{u, x}_{1}\right)$ for each $u,x\in\Rd$, and show that the integral of their probabilities is equal to our upper bound to the required precision. Our disjoint events are the following:
\begin{align}
    \Gcal_1 &:= \left\{\adja{\orig}{u}{\xi^{\orig,u}_0}\right\}\cap\left\{\adja{u}{x}{\xi^{u,x}_1} \right\}\cap\left\{x\notin \eta^{u,x}_{1,\thinn{\orig}}\right\} \\
    \Gcal_2 &:=  \left\{\adja{\orig}{u}{\xi^{\orig,u}_0}\right\}\cap\left\{\xconn{u}{x}{\xi^{u,x}_1}{2} \right\}\cap\left\{x\notin \eta^{u,x}_{1,\thinn{\orig}}\right\}\\
    \Gcal_3 &:=  \left\{\adja{\orig}{u}{\xi^{\orig,u}_0}\right\}\cap\left\{\adja{u}{x}{\xi^{u,x}_1} \right\}\cap\left\{x\in \eta^{u,x}_{1,\thinn{\orig}}\right\}\cap\left\{\exists v\in\eta_0\colon \adja{\orig}{v}{\xi^\orig_0}, x\notin \eta^{u_0,x}_{1,\thinn{v}}\right\} \\
    \Gcal_4 &:=  \left\{\adja{\orig}{u}{\xi^{\orig,u}_0}\right\}\cap\left\{\xconn{u}{x}{\xi^{u,x}_1}{3} \right\}\cap\left\{x\notin \eta^{u,x}_{1,\thinn{\orig}}\right\}\\
    \Gcal_5 &:=  \left\{\adja{\orig}{u}{\xi^{\orig,u}_0}\right\}\cap\left\{\xconn{u}{x}{\xi^{u,x}_1}{2} \right\}\cap\left\{x\in \eta^{u,x}_{1,\thinn{\orig}}\right\}\cap \left\{\exists v\in\eta_0\colon \adja{\orig}{v}{\xi^\orig_0}, x\notin \eta^{u_0,x}_{1,\thinn{v}}\right\}\\
    \Gcal_6 &:=  \left\{\adja{\orig}{u}{\xi^{\orig,u}_0}\right\}\cap\left\{\adja{u}{x}{\xi^{u,x}_1} \right\}\cap\left\{x\in \eta^{u,x}_{1,\thinn{\orig}}\right\}
    \cap\left\{\not\exists v\in\eta_0\colon \adja{\orig}{v}{\xi^\orig_0}, x\in\eta^{u_0,x}_{1,\thinn{v}}\right\}\nonumber\\
    &\hspace{5cm}\cap \left\{\exists w\in\eta_0\colon \xconn{\orig}{w}{\xi^\orig_0}{2}, x\notin \eta^{u_0,x}_{1,\thinn{w}}\right\}
\end{align}
Observe that these events are indeed disjoint, and all are subsets of $\{\dconn{\orig}{u}{\xi^{\orig, u}_{0}}\} \cap E\left(u,x; \C_{0}, \xi^{u, x}_{1}\right)$. We now want to bound their probabilities from below. For $\Gcal_1$, the events $\big\{\adja{\orig}{u}{\xi^{\orig,u}_0}\big\}$ and $\left\{\adja{u}{x}{\xi^{u,x}_1} \right\}$ are clearly independent. The event $\big\{x\notin \eta^{u,x}_{1,\thinn{\orig}}\big\}$ is also independent of these previous two, because it uses a thinning random variable from $\eta_1^{u,x}$. The probability that $x$ is thinned out by the single vertex $\orig$ is also equal to the probability that an edge forms between these vertices. Therefore $\pla\left(\Gcal_1\right) = \connf(u)\connf(x-u)\connf(x)$. The other events proceed similarly with a few points to note. All the events that are intersected to compose the $\Gcal_i$ are independent because they use different (independent) edge random variables and thinning random variables. Also, the events like $\big\{\xconn{u}{x}{\xi^{u,x}_1}{n} \big\}$ have probability given exactly by $\Exconnf{n}(x-u)$ by definition of that function. The event $\big\{x\in \eta^{u,x}_{1,\thinn{\orig}}\big\}\cap\big\{\exists v\in\eta_0\colon \adja{\orig}{v}{\xi^\orig_0}, x\notin \eta^{u_0,x}_{1,\thinn{v}}\big\}$ says that $x$ \emph{is not} thinned out by $\orig$, and that there exists a $v$ that forms an edge with $\orig$ and thins out $x$. 
This has probability equal to that of the event that no edge forms between $\orig$ and $x$, and that they have at least one mutual neighbour. This is precisely the probability given by $\Exconnf{2}(x)$. Similar considerations allow us to find factors of $\Exconnf{2}$ and $\Exconnf{3}$ in the probability of the remaining events. The lower bounds we use are summarised here:
\begin{align}
    \pla\left(\Gcal_1\right) &= \connf(u)\connf(x-u)\connf(x),\\
    \pla\left(\Gcal_2\right) &= \connf(u)\Exconnf{2}(x-u)\connf(x)\nonumber\\
    &\geq \connf(u)\connf(x)\left(1-\connf(x-u)\right)\left(\lambda\connf^{\star 2}(x-u) - \frac{1}{2}\lambda^2\connf^{\star 2}(x-u)^2\right),\\
    \pla\left(\Gcal_3\right) &= \connf(u)\connf(x-u)\Exconnf{2}(x)\nonumber\\
    &\geq \connf(u)\connf(x-u)\left(1-\connf(x)\right)\left(\lambda\connf^{\star 2}(x) - \frac{1}{2}\lambda^2\connf^{\star 2}(x)^2\right),\\
    \pla\left(\Gcal_4\right) &= \connf(u)\Exconnf{3}(x-u)\connf(x),\\
    \pla\left(\Gcal_5\right) &= \connf(u)\Exconnf{2}(x-u)\Exconnf{2}(x)\nonumber\\
    &\geq \connf(u)\left(1-\connf(x-u)\right)\left(1-\connf(x)\right)\left(\lambda^2\connf^{\star 2}(x-u)\connf^{\star 2}(x) - \frac{1}{2}\lambda^3\connf^{\star 2}(x-u)^2\connf^{\star 2}(x)\right.\nonumber\\
    &\hspace{5cm}\left.- \frac{1}{2}\lambda^3\connf^{\star 2}(x-u)\connf^{\star 2}(x)^2 + \frac{1}{4}\lambda^4\connf^{\star 2}(x-u)^2\connf^{\star 2}(x)^2\right),\\
    \pla\left(\Gcal_6\right) &= \connf(u)\connf(x-u)\Exconnf{3}(x).
\end{align}
For $\pla\left(\Gcal_2\right)$, $\pla\left(\Gcal_3\right)$, and $\pla\left(\Gcal_5\right)$ we have used a lower bound on $\Exconnf{2}$ by observing $1-\e^{-x}\geq x-\frac{1}{2}x^2$ for all $x\in\R$ and using this with the expression for $\Exconnf{2}$ in \eqref{eqn:ExConnf2}.

From these we can bound
\begin{align}
    \lambda^2\int\pla\left(\Gcal_1\right)\dd u\dd x &= \lambda^2\loopthree,\\
    \lambda^2\int\pla\left(\Gcal_2\right)\dd u\dd x &\geq \lambda^3\GtwoPtOne - \frac{1}{2}\lambda^4\GtwoPtTwo\nonumber\\
    &= \lambda^3\loopfour - \lambda^3\fourcrossone + \LandauBigO{\phiTwoTwoTwo},\\
    \lambda^2\int\pla\left(\Gcal_3\right)\dd u\dd x &\geq \lambda^3\GthreePtOne - \frac{1}{2}\lambda^4\GthreePtTwo\nonumber\\
    &= \lambda^3\loopfour - \lambda^3\fourcrosstwo + \LandauBigO{\phiTwoTwoTwo},\\
    \lambda^2\int\pla\left(\Gcal_5\right)\dd u\dd x &\geq \lambda^4\GfivePtOne -\frac{1}{2}\lambda^5\GfivePtTwo -\frac{1}{2}\lambda^5\GfivePtThree + \frac{1}{4}\lambda^6\GfivePtFour\nonumber\\
    &= \lambda^4\loopfive + \LandauBigO{\phiThreeTwoOne}.
\end{align}

To bound the integrals of $\pla\left(\Gcal_4\right)$ and $\pla\left(\Gcal_6\right)$, recall 
\begin{equation}
    \Exconnf{3}(x) \geq \left(1-\connf(x)\right)\times\left(\lambda^2\FthreePtFive{\orig}{x} - \frac{1}{2}\lambda^3\FthreePtSix{\orig}{x} - \frac{1}{2}\lambda^4\FthreePtSeven{\orig}{x} + \frac{1}{2}\lambda^5\FthreePtEight{\orig}{x} - \frac{1}{8}\lambda^6\FthreePtNine{\orig}{x}\right).
\end{equation}
Then
\begin{align}
    \lambda^2\int\pla\left(\Gcal_4\right) & \geq \lambda^4\GfourPtOne -\frac{1}{2} \lambda^5\GfourPtTwo -\frac{1}{2} \lambda^6\GfourPtThree + \frac{1}{2} \lambda^7\GfourPtFour - \frac{1}{8}\lambda^8\GfourPtFive\nonumber\\
    & = \lambda^4\loopfive + \LandauBigO{\phiThreeTwoOne},\\
    \lambda^2\int\pla\left(\Gcal_6\right) & \geq \lambda^4\GsixPtOne -\frac{1}{2} \lambda^5\GsixPtTwo -\frac{1}{2} \lambda^6\GsixPtThree + \frac{1}{2}\lambda^7\GsixPtFour -\frac{1}{8}\lambda^8 \GsixPtFive\nonumber\\
    & = \lambda^4\loopfive + \LandauBigO{\phiThreeTwoOne}.
\end{align}
Summing the integrals of the probabilities of the $\Gcal_i$ events gives us our desired lower bound for any $\lambda>0$, and in particular $\lambda=\lambda_c$.
\end{proof}

\subsection{Bounds on the Second Lace Expansion Coefficient}
\label{sec:Pi2bd}

We start with an upper bound, which we need both for $n=1$ and also for $n\ge3$ in the next subsection. 

\begin{prop}
\label{prop:LaceBoundwithDiagram}
    Let $n \geq 1$, $x\in\Rd$, and $\lambda\in \left[0,\lambda_c\right]$. Then
    \begin{equation}
        \label{eqn:Lacefunction_bound}
        \lambda\Lacelam^{(n)}(x) \leq \int  \psi_n(x,w_{n-1},u_{n-1}) \left( \prod_{i=1}^{n-1} \psi(w_i,u_i,w_{i-1},u_{i-1}) \right) \psi_0(w_0,u_0) \dd\left( \left(\vec w, \vec u\right)_{[0,n-1]} \right).
    \end{equation}
    Furthermore, for $\lambda\in \left[0,\lambda_c\right]$ there exists $c>0$ (independent of $\lambda$ and $d$) such that
    \begin{equation}
        \sum_{n=N}^\infty \fLacelam^{(n)}(0) \leq c\beta^N.
    \end{equation}
\end{prop}

\begin{proof}
    As in Lemma~\ref{lem:Pi1UpperStep1}, the first inequality is proven in very nearly exactly the same way as \cite[Proposition~7.2]{HeyHofLasMat19}. The proof of \cite[Corollary~5.3]{HeyHofLasMat19} only needs to be slightly adjusted to get the second part of our result for $\lambda<\lambda_c$, and a dominated convergence argument like that appearing in the proof of \cite[Corollary~6.1]{HeyHofLasMat19} allows us to extend the result to $\lambda=\lambda_c$.
\end{proof}

\paragraph{Upper Bound on $\fLacelamC^{(2)}(0)$}

\begin{lemma}
\label{lem:Pi2UpperBound}
Suppose Assumption~\ref{Assumption} holds. Then as $d\to\infty$,
    \begin{multline}
        \lambda_c\fLacelamC^{(2)}(0) \leq \int\psi_0(w_0,u_0)\psi(w_1,u_1,w_0,u_0)\psi_2(x,w_1,u_1)\dd w_0\dd u_0\dd w_1 \dd u_1 \dd x \\
        = \lambda_c^3\fourcrossone + \LandauBigO{\loopsix + \phiThreeTwoOne + \phiTwoTwoTwo}.
    \end{multline}
\end{lemma}

\begin{proof}
    The first inequality is an application of \eqref{eqn:Lacefunction_bound}. After expanding $\psi_0$, $\psi$, and $\psi_2$, we get 
    \begin{multline}
        \int\psi_0(w_0,u_0)\psi(w_1,u_1,w_0,u_0)\psi_2(x,w_1,u_1)\dd w_0\dd u_0\dd w_1 \dd u_1 \dd x \\= \sum_{j_0=1}^3\sum_{j_1=1}^4\sum_{j_2=1}^2\int\psi^{(j_0)}_0(w_0,u_0)\psi^{(j_1)}(w_1,u_1,w_0,u_0)\psi^{(j_2)}_2(x,w_1,u_1)\dd w_0\dd u_0\dd w_1 \dd u_1 \dd x.
    \end{multline}
    We can index the $3\times4\times2=24$ resulting diagrams by $\left(j_0,j_1,j_2\right)$. For $\left(j_0,j_1,j_2\right)\notin\left\{\left(3,2,2\right),\left(3,3,2\right),\left(3,4,2\right)\right\}$, we can identify a cycle of length $6$ or longer that visits each vertex. For each factor of $\tlam$ that is not part of this cycle we can then bound by $1$. For each factor of $\tlam$ that is part of the cycle, we bound $\tlam\leq \connf + \lambda\connf\star\tlam$. For $\lambda\leq\lambda_c$, Lemma~\ref{lem:tailbound} then lets us bound each of these diagrams by $\LandauBigO{\connf^{\star 6}\left(\orig\right)}$.

    For the $\left(3,2,2\right)$ diagram, we first expand out the $\tlamo$ edges. In many of the resulting diagrams we can apply the above strategy of finding a cycle and bounding the excess edges to bound the diagram by $\LandauBigO{\connf^{\star 6}\left(\orig\right)}$. The result is that for $\lambda\leq \lambda_c$ we have
    \begin{multline}
        \int\psi^{(3)}_0(w_0,u_0)\psi^{(2)}(w_1,u_1,w_0,u_0)\psi^{(2)}_2(x,w_1,u_1)\dd w_0\dd u_0\dd w_1 \dd u_1 \dd x \\= \lambda^4\int\connf(u_0)\tlam(z)\tlam(z-u_0)\tlam(u_1-z)\tlam(u_1-u_0)\tlam(x-u_1)\tlam(x-u_0)\dd u_0 \dd u_1 \dd z \dd x +\LandauBigO{\connf^{\star 6}\left(\orig\right)}.
    \end{multline}
    In this first integral we can bound $\tlam(u_1-u_0)\leq 1$, $\tlam(z-u_0)\leq \connf(z-u_0) + \lambda\int\connf(x)\dd x$, and the other $\tlam\leq \connf + \lambda\connf\star\tlam$ to find
    \begin{multline}
        \lambda_c^4\int\connf(u_0)\tlamC(z)\tlamC(z-u_0)\tlamC(u_1-z)\tlamC(u_1-u_0)\tlamC(x-u_1)\tlamC(x-u_0)\dd u_0 \dd u_1 \dd z \dd x \\= \LandauBigO{\loopsix + \phiThreeTwoOne}.
    \end{multline}

    For the $\left(3,3,2\right)$ diagram we bound $\tlam(u_1-z)\leq \connf(u_1-z) + \lambda\int\connf(x)\dd x$, and the other $\tlam\leq \connf + \lambda\connf\star\tlam$ to find
    \begin{multline}
        \lambda_c^4\int \connf(u_0)\tlamC(z-u_0)\tlamC(u_1)\tlamC(z-u_1)\tlamC(x-u_1)\tlamC(x-z)\dd u_0 \dd u_1 \dd z \dd x \\= \LandauBigO{\loopsix + \phiThreeTwoOne}.
    \end{multline}

    For the $\left(3,4,2\right)$ diagram we bound $\tlam(u_1-u_0)\leq \connf(u_1-u_0) + \lambda\connf^{\star 2}(u_1-u_0) + \lambda^2\left(\int\connf(x)\dd x\right)^2$, and the other $\tlam\leq \connf + \lambda\connf\star\tlam$ to find
    \begin{multline}
        \lambda_c^3\int\connf(u_0)\tlamC(u_1)\tlamC(u_1-u_0)\tlamC(x-u_0)\tlamC(x-u_1)\dd u_0 \dd u_1 \dd x \\\leq \lambda_c^3\fourcrossone + \LandauBigO{\loopsix + \phiThreeTwoOne + \phiTwoTwoTwo}.
    \end{multline}
\end{proof}

\paragraph{Lower Bound on $\fLacelamC^{(2)}(0)$}

\begin{lemma}
    \begin{equation}
        \lambda_c\fLacelamC^{(2)}(0) \geq \lambda_c^3\fourcrossone.
    \end{equation}
\end{lemma}

\begin{proof}
    We begin by identifying a suitable event for each $u_0,u_1,x\in\Rd$ that is contained in $\left\{\dconn{\orig}{u}{\xi_0^{\orig,u_0}}\right\}\cap E\left(u_0,u_1,\C_0,\xi^{u_0,u_1}_1\right)\cap E\left(u_1,x,\C_1,\xi_2^{u_1,x}\right)$. We choose the following event $\Hcal_1$:
\begin{equation}
    \Hcal_1 = \left\{\adja{\orig}{u_0}{\xi^{\orig,u_0}_0}\right\}\cap\left\{\adja{u_0}{u_1}{\xi^{u_0,u_1}_1}\right\}\cap\left\{u_1\notin \eta^{u_0,u_1}_{1,\thinn{\orig}}\right\}\cap \left\{\adja{u_1}{x}{\xi^{u_1,x}_2}\right\}\cap\left\{x\notin \eta^{u_1,x}_{2,\thinn{u_0}}\right\}
\end{equation}
and note that $\lambda\fLacelam^{(2)}(0)\geq \lambda^3\int \mathbb{P}_{\lambda}\left(\Hcal_1\right)\dd u_0 \dd u_1 \dd x$.

This event is constructed so that all the intersecting events are independent, and the probability of each is easily calculated  - once we recall that the probability that a singleton thins out a vertex is equal to the probability that an edge forms between the singleton and the vertex. Therefore
\begin{equation}
    \pla\left(\Hcal_1\right) := \connf(u_0)\connf(u_1-u_0)\connf(u_1)\connf(x-u_1)\connf(x-u_0).
\end{equation}
Integrating $\pla\left(\Hcal_1\right)$ then gives a lower bound for $\lambda_c\fLacelamC^{(2)}(0)$. This lower bound is then
\begin{equation}
    \lambda^3\int \pla\left(\Hcal_1\right)\dd u_0 \dd u_1 \dd x = \lambda^3\fourcrossone.
\end{equation}
This gives us our desired lower bound for any $\lambda>0$, and in particular $\lambda=\lambda_c$. 
\end{proof}

\subsection{Bounds on Later Lace Expansion Coefficients}
We first prove \ref{eqPi3bd} and then \ref{eqPinbd}.

\paragraph{Upper Bound on $\fLacelamC^{(3)}(0)$}
\label{sec:Pinbd}
We are first dealing with the case $n=3$, which required a special treatment, and subsequently with the general case $n\ge4$. 

\begin{lemma}
Suppose Assumption~\ref{Assumption} holds. Then as $d\to\infty$,
    \begin{equation}
    \lambda_c\fLacelamC^{(3)}(0) \leq \LandauBigO{\loopsix + \phiThreeTwoOne}.
\end{equation}
\end{lemma}

\begin{proof}
    From \eqref{eqn:Lacefunction_bound} we have
    \begin{equation}
        \lambda_c\fLacelamC^{(3)}(0) \leq \int\psi_0(w_0,u_0)\psi(w_1,u_1,w_0,u_0)\psi(w_2,u_2,w_1,u_1)\psi_3(x,w_2,u_2)\dd w_0\dd u_0\dd w_1 \dd u_1 \dd w_2 \dd u_2 \dd x.
    \end{equation}
    Then as in the proof of Lemma~\ref{lem:Pi2UpperBound} we expand out the $\psi_0$, $\psi$, and $\psi_3$ functions. Then for each integral we aim to identify a cycle of length $6$ or longer that visits each vertex. For each factor of $\tlamC$ that is not part of this cycle we can then bound by $1$. For each factor of $\tlamC$ that is part of the cycle, we bound $\tlamC\leq \connf + \lambda_c\connf\star\tlamC$. Lemma~\ref{lem:tailbound} then lets us bound each of these diagrams by $\LandauBigO{\connf^{\star 6}\left(\orig\right)}$.

    The only integral that we cannot perform this strategy for corresponds to the integral
    \begin{multline}
        \int\psi^{(3)}_0(w_0,u_0)\psi^{(4)}(w_1,u_1,w_0,u_0)\psi^{(4)}(w_2,u_2,w_1,u_1)\psi^{(2)}_3(x,w_2,u_2)\dd w_0\dd u_0\dd w_1 \dd u_1 \dd w_2 \dd u_2 \dd x \\=\lambda_c^4\int \connf(u_0)\tlamC(u_1)\tlamC(u_1-u_0)\tlamC(u_2-u_0)\tlamC(u_2-u_1)\tlamC(x-u_1)\tlamC(x-u_2)\dd u_0 \dd u_1 \dd u_2 \dd x. 
    \end{multline}
    If we bound $\tlamC(u_2-u_1)\leq \connf(u_2-u_1) + \lambda_c\connf^{\star 2}(u_2-u_1) + \lambda_c^2\connf^{\star 3}(u_2-u_1) + \lambda_c^3\left(\int\connf(x)\dd x\right)^3$, and the other $\tlamC\leq \connf + \lambda_c\connf\star\tlamC$ we find that 
    \begin{multline}
        \lambda_c^4\int \connf(u_0)\tlamC(u_1)\tlamC(u_1-u_0)\tlamC(u_2-u_0)\tlamC(u_2-u_1)\tlamC(x-u_1)\tlamC(x-u_2)\dd u_0 \dd u_1 \dd u_2 \dd x \\
        = \LandauBigO{\loopsix + \phiThreeTwoOne}.
    \end{multline}
\end{proof}

\paragraph{Upper Bound on $\fLacelamC^{(n)}(0)$ for $n\geq 4$}

\begin{lemma}
Suppose Assumption~\ref{Assumption} holds and $n\geq 1$. Then as $d\to\infty$,
    \begin{equation}
    \lambda_c\fLacelamC^{(n)}(0) =
    \begin{cases}
        \LandauBigO{\connf^{\star (n+2)}\left(\orig\right)}&\colon n \text{ is even,}\\
        \LandauBigO{\connf^{\star (n+1)}\left(\orig\right)}&\colon n \text{ is odd.}\\
    \end{cases}
\end{equation}
\end{lemma}

\begin{proof}
    We begin this proof by using Proposition~\ref{prop:LaceBoundwithDiagram} to get an upper bound for $\lambda_c\fLacelamC^{(n)}(0)$ in terms of a sum of integrals of $\psi^{(j_0)}_0$, $\psi^{(j)}$, and $\psi^{(j_n)}_n$. Our strategy to bound each of these diagrams is to identify a loop of length at least $n+2$ around each of these diagrams.

    For each $\psi$-function we provide an upper bound in a $\psibar$-function, so that when they are applied to our integral bounds we get terms of the form  $\lambda^{m-1}\tlam^{\star m}\left(\orig\right)$ for some $m\geq n+1$. We have
\begin{align*}
    \psi_0^{(1)}(w,u) &\leq \psibar_0^{(1)}(w,u) := \lambda^2\tlam(u)\tlam(w),\\
    \psi_0^{(2)}(w,u) &\leq \psibar_0^{(2)}(w,u) := \lambda^2 \delta_{w,\orig}\int \tlam(u-t)\tlam(t) \dd t,\\
    \psi_0^{(3)}(w,u) & \leq \psibar_0^{(3)}(w,u) := \lambda\tlam(u) \delta_{w,\orig},\\
    \psi^{(1)}(w,u,r,s) &\leq \psibar^{(1)}(w,u,r,s) := \lambda^4\int \tlamo(t-s)\tlam(w-t)\dd t \int \tlam(u-z)\tlam(z-r)\dd z, \\
    \psi^{(2)}(w,u,r,s) &\leq \psibar^{(2)}(w,u,r,s) := \lambda^4\tlamo(w-s)\int \tlam(z-r)\tlam(t-z)\tlam(u-t)\dd z \dd t \\ &\hspace{7cm}+ \lambda^3\tlamo(w-s)\int \tlam(z-r)\tlam(u-z)\dd z, \\
	\psi^{(3)}(w,u,r,s) &\leq \psibar^{(3)}(w,u,r,s) := \lambda^2\tlam(w-s)\tlam(u-r),\\
	\psi^{(4)}(w,u,r,s) &\leq \psibar^{(4)}(w,u,r,s) := \lambda\delta_{w,s}\tlam(u-r),\\
	\psi_n^{(1)} (x,r,s) &\leq \psibar_n^{(1)}(w,r,s) := \lambda^3\int \tlamo(t-s)\tlam(z-r)\tlam(z-x)\tlam(x-t)\dd z\dd t,\\
    \psi_n^{(2)}(x,r,s) &= \psibar_n^{(2)}(x,r,s) := \lambda \tlam(x-s)\tlam(x-r).
\end{align*}
We can also define $\psibar_0$, $\psibar$, and $\psibar_n$ analogously to how we defined $\psi_0$, $\psi$, and $\psi_n$.

\begin{figure}
    \centering
    \begin{subfigure}[b]{0.3\textwidth}
    \centering
        \begin{tikzpicture}
        \draw[green] (0,0) -- (1,0.6);
        \draw[green] (1,-0.6) -- (0,0);
        \filldraw[fill=white] (0,0) circle (2pt) node[left]{$\orig$};
        \filldraw[fill=white] (1,0.6) circle (2pt) node[above right]{$w$};
        \filldraw[fill=white] (1,-0.6) circle (2pt) node[below right]{$u$};
        \end{tikzpicture}
    \caption{$\psibar^{(1)}_0$}
    \end{subfigure}
    \hfill
    \begin{subfigure}[b]{0.3\textwidth}
    \centering
        \begin{tikzpicture}
        \draw[green] (0,0) -- (1,0.6);
        \draw[green] (1,-0.6) -- (0,0);
        \filldraw (0,0) circle (2pt);
        \filldraw[fill=white] (1,0.6) circle (2pt) node[above left]{$\orig=w$};
        \filldraw[fill=white] (1,-0.6) circle (2pt) node[below right]{$u$};
       \end{tikzpicture}
    \caption{$\psibar^{(2)}_0$}
    \end{subfigure}
    \hfill
    \begin{subfigure}[b]{0.3\textwidth}
    \centering
        \begin{tikzpicture}
        \draw[green] (0,0.6) -- (0,-0.6);
        \filldraw[fill=white] (0,0.6) circle (2pt) node[above left]{$\orig=w$};
        \filldraw[fill=white] (0,-0.6) circle (2pt) node[below right]{$u$};
        \end{tikzpicture}
    \caption{$\psibar^{(3)}_0$}
    \end{subfigure}
    \begin{subfigure}[b]{0.2\textwidth}
    \centering
        \begin{tikzpicture}
        \draw[green] (0,0.6) -- (1,0.6);
        \draw[green] (1,-0.6) -- (0,-0.6);
        \draw (0.5,0.6) circle (0pt) node[above]{$\circ$};
        \draw[green] (1,0.6) -- (2,0.6);
        \draw[green] (2,-0.6) -- (1,-0.6);
        \draw (0,1.75) circle (0pt);
        \filldraw[fill=white] (0,0.6) circle (2pt) node[above left]{$s$};
        \filldraw[fill=white] (0,-0.6) circle (2pt) node[below left]{$r$};
        \filldraw (1,0.6) circle (2pt);
        \filldraw (1,-0.6) circle (2pt);
        \filldraw[fill=white] (2,0.6) circle (2pt) node[above right]{$w$};
        \filldraw[fill=white] (2,-0.6) circle (2pt) node[below right]{$u$};
        \end{tikzpicture}
    \caption{$\psibar^{(1)}$}
    \end{subfigure}
    \hfill
    \begin{subfigure}[b]{0.4\textwidth}
    \centering
        \begin{tikzpicture}
        \draw[green] (0,0.6) -- (1.5,0.6);
        \draw[green] (1.5,-0.1) -- (1,-0.6);
        \draw (0.75,0.6) circle (0pt) node[above]{$\circ$};
        \draw[green] (0,-0.6) -- (1,-0.6);
        \draw[green] (2,-0.6) -- (1.5,-0.1);
        \filldraw[fill=white] (0,0.6) circle (2pt) node[above left]{$s$};
        \filldraw[fill=white] (0,-0.6) circle (2pt) node[below left]{$r$};
        \filldraw[fill=white] (1.5,0.6) circle (2pt) node[above right]{$w$};
        \filldraw (1,-0.6) circle (2pt);
        \filldraw (1.5,-0.1) circle (2pt);
        \filldraw[fill=white] (2,-0.6) circle (2pt) node[below right]{$u$};
        \draw (2.4,0) circle (0pt) node {$+$};
        \draw[green] (3,0.6) -- (4.5,0.6);
        \draw (3.75,0.6) circle (0pt) node[above]{$\circ$};
        \draw[green] (3,-0.6) -- (5,-0.6);
        \filldraw[fill=white] (3,0.6) circle (2pt) node[above left]{$s$};
        \filldraw[fill=white] (3,-0.6) circle (2pt) node[below left]{$r$};
        \filldraw[fill=white] (4.5,0.6) circle (2pt) node[above right]{$w$};
        \filldraw (4,-0.6) circle (2pt);
        \filldraw[fill=white] (5,-0.6) circle (2pt) node[below right]{$u$};
        \end{tikzpicture}
    \caption{$\psibar^{(2)}$}
    \end{subfigure}
    \hfill
    \begin{subfigure}[b]{0.19\textwidth}
    \centering
        \begin{tikzpicture}
        \draw[green] (0,0.6) -- (1,0.6);
        \draw[green] (1,-0.6) -- (0,-0.6);
        \filldraw[fill=white] (0,0.6) circle (2pt) node[above left]{$s$};
        \filldraw[fill=white] (0,-0.6) circle (2pt) node[below left]{$r$};
        \filldraw[fill=white] (1,0.6) circle (2pt) node[above right]{$w$};
        \filldraw[fill=white] (1,-0.6) circle (2pt) node[below right]{$u$};
        \end{tikzpicture}
    \caption{$\psibar^{(3)}$}
    \end{subfigure}
    \hfill
    \begin{subfigure}[b]{0.19\textwidth}
    \centering
        \begin{tikzpicture}
        \draw[green] (1,-0.6) -- (0,-0.6);
        \filldraw[fill=white] (0,-0.6) circle (2pt) node[below left]{$r$};
        \filldraw[fill=white] (1,0.6) circle (2pt) node[above]{$s=w$};
        \filldraw[fill=white] (1,-0.6) circle (2pt) node[below right]{$u$};
        \end{tikzpicture}
    \caption{$\psibar^{(4)}$}
    \end{subfigure}
    \hspace*{\fill}%
    \begin{subfigure}[b]{0.45\textwidth}
    \centering
        \begin{tikzpicture}
        \draw[green] (0,0.6) -- (1,0.6);
        \draw (0.5,0.6) circle (0pt) node[above]{$\circ$};
        \draw[green] (1,0.6) -- (2,0) -- (1,-0.6) -- (0,-0.6);
        \filldraw[fill=white] (0,0.6) circle (2pt) node[above left]{$s$};
        \filldraw[fill=white] (0,-0.6) circle (2pt) node[below left]{$r$};
        \filldraw (1,0.6) circle (2pt);
        \filldraw (1,-0.6) circle (2pt);
        \filldraw[fill=white] (2,0) circle (2pt) node[right]{$x$};
        \end{tikzpicture}
    \caption{$\psibar^{(1)}_n$}
    \end{subfigure}
    \hfill
    \begin{subfigure}[b]{0.45\textwidth}
    \centering
        \begin{tikzpicture}
        \draw[green] (1,0.6) -- (2,0) -- (1,-0.6);
        \draw (0,1.75) circle (0pt);
        \filldraw[fill=white] (1,0.6) circle (2pt) node[above left]{$s$};
        \filldraw[fill=white] (1,-0.6) circle (2pt) node[below left]{$r$};
        \filldraw[fill=white] (2,0) circle (2pt) node[right]{$x$};
        \end{tikzpicture}
    \caption{$\psibar^{(2)}_n$}
    \end{subfigure}
    \hspace*{\fill}%
    \caption{Diagrams of the $\psibar_0$, $\psibar$, and $\psibar_n$ functions, which we use to bound the $\psi_0$, $\psi$, and $\psi_n$ functions.}
    \label{fig:psiOverFunctions}
\end{figure}
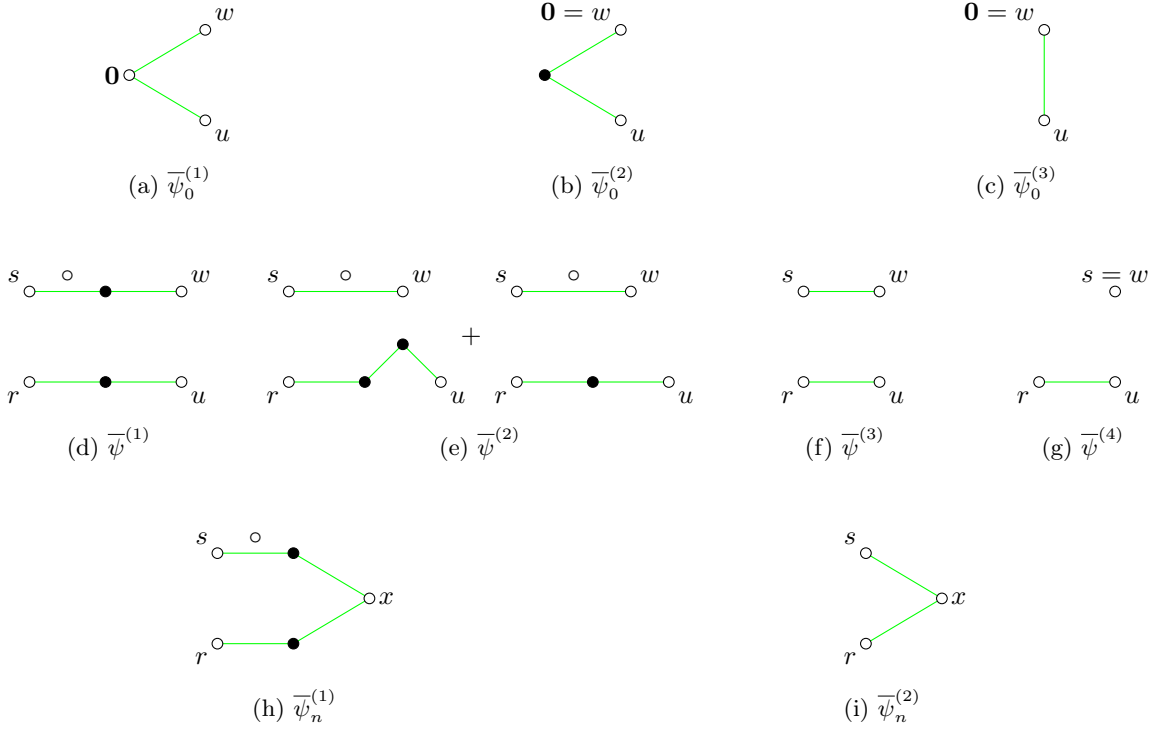

    First off, we leave $\psi_n^{(2)}$ alone. For $\psi_0^{(3)}$ we bound $\connf\leq \tlam$. For most of the others, the bound is achieved only by bounding $\tlam^{(\geq2)}\leq \tlam$ and $\tlam\leq 1$ in the appropriate places. The bound for $\psi^{(2)}$ deserves a little more explanation. Here we first expand $\tlamo(t-w) = \tlam(t-w) + \lambda^{-1}\delta_{t,w}$ to get two expressions. We then bound $\tlam\leq 1$ as for the others, but in different ways for each of the two expressions.

    We therefore find that $\psibar_0$ can contribute one or two factors of $\tlam$, $\psibar$ can contribute one, two, three, or four factors of $\tlam$, and $\psibar_n$ can contribute two, three, or four factors of $\tlam$. Our bound is therefore a sum of terms of $\tlam^{\star m}\left(\orig\right)$ where $m$ is at least $1+1\times(n-1)+2 = n+2$ and at most $2 + 4\times(n-1)+4 = 4n+2$. Therefore
    \begin{equation}
        \lambda_c\fLacelamC^{(n)}\left(0\right) \leq \LandauBigO{\sum^{4n+2}_{m=n+2}\tlam^{\star m}\left(\orig\right)}.
    \end{equation}
    
    For each factor of $\tlam$ here we now bound $\tlam\leq \connf + \lambda\connf\star\tlam$ to get
    \begin{equation}
        \lambda_c\fLacelamC^{(n)}(0) \leq \LandauBigO{\sum_{m=n+2}^{4n+2}\sum_{j=0}^{4n+2}\connf^{\star m}\star\tlam^{\star j}\left(\orig\right)}.
    \end{equation}
    If $m$ is odd and $j\geq 1$ then we bound $\connf^{\star m}\star\tlam^{\star j}\left(\orig\right) \leq \left(\int\connf(x)\dd x\right)\connf^{\star (m-1)}\star\tlam^{\star j}\left(\orig\right)$. Then Lemma~\ref{lem:tailbound} gives us
    \begin{equation}
        \lambda_c\fLacelamC^{(n)}(0) \leq 
        \begin{cases}
        \LandauBigO{\sum_{m=n+2}^{4n+2}\connf^{\star m}\left(\orig\right)} &\colon n\text{ is even,}\\
        \LandauBigO{\sum_{m=n+1}^{4n+2}\connf^{\star m}\left(\orig\right)} &\colon n\text{ is odd.}
        \end{cases}
    \end{equation}
    If $n$ is even we bound $\connf^{\star m}\left(\orig\right)\leq \left(\int\connf(x)\dd x\right)^{m-n-2}\connf^{\star (n+2)}\left(\orig\right)$ to get our result, and if $n$ is odd we bound $\connf^{\star m}\left(\orig\right)\leq \left(\int\connf(x)\dd x\right)^{m-n-1}\connf^{\star (n+1)}\left(\orig\right)$ to get our result.
\end{proof}

\begin{appendix}
\section{Calculations for Specific Models}
We now provide details for the specific percolation models in Section \ref{sec:applications}. 
To this end, we need to show that each of the four models satisfies Assumptions~\ref{Assumption} and \ref{AssumptionBeta} and find the specific values of the integrals of $\connf$ appearing in \eqref{eq:ExpansionMain}. 

\subsection{Hyper-Sphere Calculations}

Recall that for radius $R>0$, the Hyper-Sphere RCM is defined by having
        \begin{equation}
            \connf(x) = \Id_{\left\{\abs*{x}< R\right\}}.
        \end{equation}
Throughout this section we choose a scaling of $\Rd$ such that $R=R(d)$ is the radius of the unit $d$-volume ball. Therefore $R(d)=\pi^{-\frac{1}{2}}\Gamma\left(\frac{d}{2}+1\right)^{\frac{1}{d}} = \sqrt{\frac{d}{2\pi \e}}\left(1+o\left(1\right)\right)$ (by an application of Stirling's formula).

\begin{lemma}
    The Hyper-Sphere RCM satisfies Assumption~\ref{Assumption}.
\end{lemma}

\begin{proof}
    It is proven in \cite[Proposition~1.1]{HeyHofLasMat19} that the Hyper-Sphere RCM satisfies Assumption~\ref{Assumption} with $g(d) = \varrho^d$ for some $\varrho\in\left(0,1\right)$.
\end{proof}

\begin{lemma}
    The Hyper-Sphere RCM satisfies Assumption~\ref{AssumptionBeta}.
\end{lemma}

\begin{proof}
    In order to prove that \ref{Assump:ExponentialDecay} holds, we need to get a lower bound on $\connf^{\star 6}\left(\orig\right)$. We begin using the Fourier inverse formula to get
    \begin{equation}
        \connf^{\star 6}\left(\orig\right) = \int_{\Rd}\fconnf(k)^6\frac{\dd k}{\left(2\pi\right)^d}.
    \end{equation}
    Since $\connf$ is symmetric, $\fconnf(k)$ is real, and therefore we can get a lower bound on $\connf
    ^{\star 6}\left(\orig\right)$ by getting a lower bound on $\fconnf(k)^6$.

    From \cite[Appendix~B.5]{grafakos2008classical}, we can find that
    \begin{equation}
    \label{eqn:FourierAdjacencyBooleanDisc}
        \fconnf(k) = \left(\frac{2\pi R(d)}{\abs*{k}}\right)^\frac{d}{2} J_{\frac{d}{2}}\left(\abs*{k}R(d)\right),
    \end{equation}
    where $J_{\frac{d}{2}}$ is the Bessel function of the first kind of order $\frac{d}{2}$, and $R(d)$ is the radius of the unit volume ball in $d$ dimensions. In Figure~\ref{fig:Sketch_fconnf_BooleanDisc} we highlight three important values of $\abs*{k}$ in the shape of $\fconnf(k)$. The Bessel function $J_{\frac{d}{2}}$ achieves its global maximum (in absolute value) at its first non-zero stationary point, $j'_{\frac{d}{2},1}$. From \cite[p.371]{abramowitz1970handbook}, we have $j'_{\frac{d}{2},1} = \frac{d}{2} + \gamma_1\left(\frac{d}{2}\right)^\frac{1}{3} + \LandauBigO{d^{-\frac{1}{3}}}$ for a given $\gamma_1 \approx 0.81$, and $J_{\frac{d}{2}}\left(j'_{\frac{d}{2},1}\right) = \Gamma_1 d^{-\frac{1}{3}} + \LandauBigO{d^{-1}}$, where $\Gamma_1\approx 0.54$. Then $J_{\frac{d}{2}}$ has its first zero at $j_{\frac{d}{2},1}>j'_{\frac{d}{2},1}$, where $j_{\frac{d}{2},1} = \frac{d}{2} + \gamma_2\left(\frac{d}{2}\right)^{\frac{1}{3}} + \LandauBigO{d^{-\frac{1}{3}}}$ and $\gamma_2\approx 1.86$ (again, see \cite{abramowitz1970handbook}). From differential inequalities relating Bessel functions (see \cite{grafakos2008classical}), we have
    \begin{equation}
    \frac{\dd }{\dd \abs*{k}} \fconnf(k) = -R(d)\left(\frac{2\pi R(d)}{\abs*{k}}\right)^\frac{d}{2}J_{\frac{d}{2}+1}\left(\abs*{k} R(d)\right).
\end{equation}
    Therefore $\fconnf(k)$ is decreasing in $\abs*{k}$ until $\abs*{k}R(d) = j_{\frac{d}{2}+1,1} = \frac{d}{2} + \gamma_2\left(\frac{d}{2}\right)^{\frac{1}{3}} + 1 +\LandauBigO{d^{-\frac{1}{3}}}$. In particular, $j_{\frac{d}{2}+1,1}> j_{\frac{d}{2},1}$.
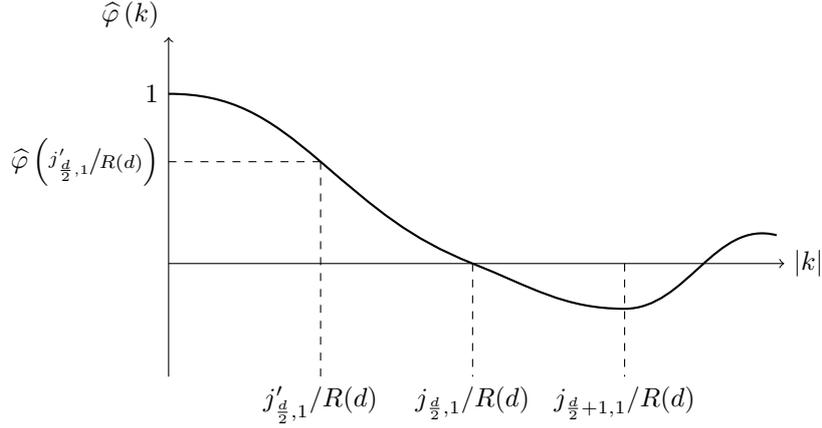
\begin{figure}
    \centering
    \begin{tikzpicture}[yscale=1.5]
        \draw[->] (0,0) -- (8.1,0)node[right]{$\abs*{k}$};
        \draw[->] (0,-1) -- (0,2)node[above left]{$\fconnf\left(k\right)$};
        \draw[thick] (0,1.5)node[left]{$1$} to [out=0,in=150] (2,0.9) to [out=330,in=165] (4,0) to [out=-15,in=180] (6,-0.4) to [out=0,in=170] (8,0.25);
        \draw[dashed] (2,0.9) -- (0,0.9)node[left]{$\fconnf\left(\sfrac{j'_{\frac{d}{2},1}}{R(d)}\right)$};
        \draw[dashed] (4,0) -- (4,-1)node[below]{$j_{\frac{d}{2},1}/R(d)$};
        \draw[dashed] (6,0) -- (6,-1)node[below]{$j_{\frac{d}{2}+1,1}/R(d)$};
        \draw[dashed] (2,0.9) -- (2,-1)node[below]{$j'_{\frac{d}{2},1}/R(d)$};
    \end{tikzpicture}
    \caption{Sketch of $\fconnf\left(k\right)$ against $\abs*{k}$. It approaches its maximum quadratically as $\abs*{k}\to0$. The first local maximum of $J_{\frac{d}{2}}$ occurs at $j'_{\frac{d}{2},1}\sim \frac{d}{2}+\gamma_1\left(\frac{d}{2}\right)^\frac{1}{3}$. The first zero of $\fconnf\left(k\right)$ occurs at $\abs*{k}R(d) = j_{\frac{d}{2},1}\sim \frac{d}{2}+\gamma_2\left(\frac{d}{2}\right)^\frac{1}{3}$ where $\gamma_2>\gamma_1$. Furthermore, $\fconnf\left(k\right)$ is strictly decreasing until $\abs*{k}R(d) = j_{\frac{d}{2}+1,1}\sim \frac{d}{2}+\gamma_2\left(\frac{d}{2}\right)^\frac{1}{3} + 1$. }
    \label{fig:Sketch_fconnf_BooleanDisc}
\end{figure}
    The significance of these points is that they allow us to bound
    \begin{equation}
        \abs*{\fconnf(k)} \geq \fconnf\left(\sfrac{j'_{\frac{d}{2},1}}{R(d)}\right)\Id_{\left\{\abs*{k}\leq \sfrac{j'_{\frac{d}{2},1}}{R(d)}\right\}}.
    \end{equation}
    Since $R(d)$ is the radius of the unit volume ball in $d$ dimensions,
    \begin{equation}
        \int_{\Rd}\Id_{\left\{\abs*{k}\leq \sfrac{j'_{\frac{d}{2},1}}{R(d)}\right\}}\frac{\dd k}{\left(2\pi\right)^d} = \left(\frac{j'_{\frac{d}{2},1}}{2\pi R(d)^2}\right)^d.
    \end{equation}
    Therefore we can arrive at
    \begin{equation}
        \connf^{\star 6}\left(\orig\right) \geq \left(\frac{2\pi R(d)^2}{j'_{\frac{d}{2},1}}\right)^{2d} J_{\frac{d}{2}}\left(j'_{\frac{d}{2},1}\right)^6 = \Gamma_1^6 \frac{1}{d^2}\left(\frac{2}{\e}+o\left(1\right)\right)^{2d}\left(1+ o(1)\right).
    \end{equation}
    Here we have used the leading order asymptotics of $R(d)$, $j'_{\frac{d}{2},1}$, and $J_{\frac{d}{2}}\left(j'_{\frac{d}{2},1}\right)$ we described above. From this lower bound, we know that $\rho$ will satisfy the bound in \ref{Assump:ExponentialDecay} if $\rho < 4 \e^{-2}$.

    From the above argument we have an exponential lower bound on $\connf^{\star 6}\left(\orig\right)$ and therefore a linear lower bound on $h(d)$. It is proven in \cite[Proposition~1.1a]{HeyHofLasMat19} that $g(d)=\varrho^{d}$ for some $\varrho\in\left(0,1\right)$, and therefore $\beta(d) = \varrho^{\frac{d}{4}}$. We can then bound $N(d)$ to show \ref{Assump:NumberBound} holds.
\end{proof}

\begin{lemma}
\label{lem:BooleanCalcExpressions}
    For $n\geq 3$, 
    \begin{equation}
    \label{eqn:BooleanLoopGeneral}
        \connf^{\star n}\left(\orig\right) = d2^{d\left(\frac{n}{2}-1\right)}\Gamma\left(\frac{d}{2}+1\right)^{n-2}\int^\infty_0 x^{-1-d\left(\frac{n}{2}-1\right)}\left(J_{\frac{d}{2}}(x)\right)^n \dd x.
    \end{equation}
    In particular,
    \begin{align}
        \connf^{\star 3}\left(\orig\right) &= \frac{d\Gamma\left(\frac{d}{2}+1\right)}{\Gamma\left(\frac{1}{2}\right)\Gamma\left(\frac{d}{2}+\frac{1}{2}\right)}\int^1_0x^{d-1}B\left(1-\frac{x^2}{4};\frac{d}{2}+\frac{1}{2},\frac{1}{2}\right)\dd x = \frac{3}{2}\frac{\Gamma\left(\frac{d}{2}+1\right)}{\Gamma\left(\frac{1}{2}\right)\Gamma\left(\frac{d}{2}+\frac{1}{2}\right)}B\left(\frac{3}{4};\frac{d}{2}+\frac{1}{2},\frac{1}{2}\right),\label{eqn:BooleanLoop3}\\
        \connf^{\star 4}\left(\orig\right) &= \frac{d\Gamma\left(\frac{d}{2}+1\right)^2}{\Gamma\left(\frac{1}{2}\right)^2\Gamma\left(\frac{d}{2}+\frac{1}{2}\right)^2}\int^2_0x^{d-1}B\left(1-\frac{x^2}{4};\frac{d}{2}+\frac{1}{2},\frac{1}{2}\right)^2\dd x,\\
        \connf^{\star 5}\left(\orig\right) &= \frac{d2^d\Gamma\left(\frac{d}{2}+1\right)^3}{\Gamma\left(\frac{1}{2}\right)\Gamma\left(\frac{d}{2}+\frac{1}{2}\right)}\int^2_0x^{\frac{d}{2}}B\left(1-\frac{x^2}{4};\frac{d}{2}+\frac{1}{2},\frac{1}{2}\right)\left(\int^\infty_0 k^{-d}\left(J_{\frac{d}{2}}(k)\right)^3J_{\frac{d}{2}-1}(kx) \dd k\right)\dd x, \\
        \connf^{\star 6}\left(\orig\right) &= \frac{d2^\frac{3d}{2}\Gamma\left(\frac{d}{2}+1\right)^4}{\Gamma\left(\frac{1}{2}\right)\Gamma\left(\frac{d}{2}+\frac{1}{2}\right)}\int^2_0x^{\frac{d}{2}}B\left(1-\frac{x^2}{4};\frac{d}{2}+\frac{1}{2},\frac{1}{2}\right)\left(\int^\infty_0 k^{-\frac{3d}{2}}\left(J_{\frac{d}{2}}(k)\right)^4J_{\frac{d}{2}-1}(kx) \dd k\right)\dd x.
    \end{align}
    Furthermore,
    \begin{align}
        \connf^{\star 1\star 2 \cdot 2}\left(\orig\right) &= \frac{d\Gamma\left(\frac{d}{2}+1\right)^2}{\Gamma\left(\frac{1}{2}\right)^2\Gamma\left(\frac{d}{2}+\frac{1}{2}\right)^2}\int_0^1 x^{d-1}B\left(1-\frac{x^2}{4};\frac{d}{2}+\frac{1}{2},\frac{1}{2}\right)^2 \dd x,\\
        \connf^{\star 2\star 2 \cdot 2}\left(\orig\right) &= \frac{d\Gamma\left(\frac{d}{2}+1\right)^3}{\Gamma\left(\frac{1}{2}\right)^3\Gamma\left(\frac{d}{2}+\frac{1}{2}\right)^3}\int_0^2 x^{d-1}B\left(1-\frac{x^2}{4};\frac{d}{2}+\frac{1}{2},\frac{1}{2}\right)^3 \dd x,\\
        \connf^{\star 1\star 2 \cdot 3}\left(\orig\right) &= \frac{d2^d\Gamma\left(\frac{d}{2}+1\right)^3}{\Gamma\left(\frac{1}{2}\right)\Gamma\left(\frac{d}{2}+\frac{1}{2}\right)}\int^1_0x^{\frac{d}{2}}B\left(1-\frac{x^2}{4};\frac{d}{2}+\frac{1}{2},\frac{1}{2}\right)\left(\int^\infty_0 k^{-d}\left(J_{\frac{d}{2}}(k)\right)^3J_{\frac{d}{2}-1}(kx) \dd k\right)\dd x.
    \end{align}
\end{lemma}

\begin{proof}
    Let $R=R(d)$ denote the radius of the unit volume $d$-dimensional Euclidean ball, i.e. $R(d)=\frac{1}{\sqrt{\pi}}\Gamma\left(\frac{d}{2}+1\right)^{\frac{1}{d}}$. In particular, note the relation
    \begin{equation}
        1 = \mathfrak{S}_{d-1}\int^R_0r^{d-1}\dd r = \frac{\mathfrak{S}_{d-1}}{d}R^d,
    \end{equation}
    where $\mathfrak{S}_{d-1}=\frac{d\pi^{\frac{d}{2}}}{\Gamma\left(\frac{d}{2}+1\right)}$ is the surface area of the unit \emph{radius} $d$-dimensional Euclidean ball.

    The general formula \eqref{eqn:BooleanLoopGeneral} follows from a Fourier decomposition. By the Fourier inversion formula,
    \begin{equation}
        \connf^{\star n}(x) = \frac{1}{\left(2\pi\right)^d}\int \fconnf(k)^n\dd k.
    \end{equation}
    Recall the expression \eqref{eqn:FourierAdjacencyBooleanDisc} for the Fourier transform $\fconnf(k)$. Then
    \begin{multline}
        \connf^{\star n}\left(\orig\right) = \frac{1}{\left(2\pi\right)^d}\left(2\pi\right)^{\frac{d}{2}n}R^{\frac{d}{2}n}\mathfrak{S}_{d-1}\int^\infty_0 k^{d-1-\frac{d}{2}n}\left(J_{\frac{d}{2}}\left(Rk\right)\right)^n\dd k\\
        =\left(2\pi\right)^{d\left(\frac{n}{2}-1\right)}R^d\mathfrak{S}_{d-1}\int^\infty_0x^{-1-d\left(\frac{n}{2}-1\right)}\left(J_{\frac{d}{2}}\left(x\right)\right)^n\dd x.
    \end{multline}
    Then observing that $R^d\mathfrak{S}_{d-1}=d$ produces the result.

    In the cases $n=3,4$, a more geometric approach may be taken. First note that $\connf^{\star 2}(x)$ can be interpreted as the $d$-volume of the intersection of a hyper-sphere of radius $R$ at the origin with a hyper-sphere of radius $R$ at the position $x$. An expression for this volume is given by \cite{li2011concise} using incomplete Beta functions:
    \begin{equation}
        \connf^{\star 2}(x) = \frac{\Gamma\left(\frac{d}{2}+1\right)}{\Gamma\left(\frac{1}{2}\right)\Gamma\left(\frac{d}{2}+\frac{1}{2}\right)}B\left(1-\frac{\abs*{x}^2}{4R^2};\frac{d}{2}+\frac{1}{2},\frac{1}{2}\right), \qquad \text{for }\abs*{x}\leq 2R.
    \end{equation}
    Clearly $\connf^{\star 2}(x)=0$ for $\abs*{x}>2R$. It then follows that
    \begin{multline}
        \connf^{\star 3}\left(\orig\right) = \int\connf(x)\connf^{\star 2}(x)\dd x = \frac{\Gamma\left(\frac{d}{2}+1\right)}{\Gamma\left(\frac{1}{2}\right)\Gamma\left(\frac{d}{2}+\frac{1}{2}\right)}\mathfrak{S}_{d-1}\int^R_0 r^{d-1}B\left(1-\frac{r^2}{4R^2};\frac{d}{2}+\frac{1}{2},\frac{1}{2}\right)\dd r\\= \frac{\Gamma\left(\frac{d}{2}+1\right)}{\Gamma\left(\frac{1}{2}\right)\Gamma\left(\frac{d}{2}+\frac{1}{2}\right)}\mathfrak{S}_{d-1}R^d\int^1_0 x^{d-1}B\left(1-\frac{x^2}{4};\frac{d}{2}+\frac{1}{2},\frac{1}{2}\right)\dd x.
    \end{multline}
    Again, noting that $R^d\mathfrak{S}_{d-1} = d$ produces the required first equality in \eqref{eqn:BooleanLoop3}. It was noted in \cite{Tor12} that for the Hyper-Sphere model we have $\connf^{\star 3}\left(\orig\right)=\frac{3}{2}\connf^{\star 2}(\tilde{x})$, where $\abs*{\tilde{x}}=R$. This can be proven by writing out the incomplete Beta function as an integral to get a double integral, partitioning the domain appropriately, and using a suitable trigonometric substitution on each part of the domain. We omit the details here. This relation allows us to get the second equality in \eqref{eqn:BooleanLoop3}.

    For the specific form of $\connf^{\star 4}\left(\orig\right)$, we do a similar calculation to that above:
    \begin{multline}
        \connf^{\star 4}\left(\orig\right) = \int\connf^{\star 2}(x)^2\dd x = \frac{\Gamma\left(\frac{d}{2}+1\right)^2}{\Gamma\left(\frac{1}{2}\right)^2\Gamma\left(\frac{d}{2}+\frac{1}{2}\right)^2}\mathfrak{S}_{d-1}\int^{2R}_0 r^{d-1}B\left(1-\frac{r^2}{4R^2};\frac{d}{2}+\frac{1}{2},\frac{1}{2}\right)^2\dd r\\= \frac{\Gamma\left(\frac{d}{2}+1\right)^2}{\Gamma\left(\frac{1}{2}\right)^2\Gamma\left(\frac{d}{2}+\frac{1}{2}\right)^2}\mathfrak{S}_{d-1}R^d\int^2_0 x^{d-1}B\left(1-\frac{x^2}{4};\frac{d}{2}+\frac{1}{2},\frac{1}{2}\right)^2\dd x.
    \end{multline}
    Using $R^d\mathfrak{S}_{d-1} = d$ gives the result.

    For $\connf^{\star1\star2\cdot2}\left(\orig\right)$ and $\connf^{\star 2\star 2\cdot 1}\left(\orig\right)$ this approach also works. We find
    \begin{align}
        \connf^{\star1\star2\cdot2}\left(\orig\right) = \int\connf(x)\connf^{\star2 }(x)^2\dd x &= \frac{\Gamma\left(\frac{d}{2}+1\right)^2}{\Gamma\left(\frac{1}{2}\right)^2\Gamma\left(\frac{d}{2}+\frac{1}{2}\right)^2}\mathfrak{S}_{d-1}\int^{R}_0 r^{d-1}B\left(1-\frac{r^2}{4R^2};\frac{d}{2}+\frac{1}{2},\frac{1}{2}\right)^2\dd r\nonumber\\
        &= \frac{\Gamma\left(\frac{d}{2}+1\right)^2}{\Gamma\left(\frac{1}{2}\right)^2\Gamma\left(\frac{d}{2}+\frac{1}{2}\right)^2}\mathfrak{S}_{d-1}R^d\int^1_0 x^{d-1}B\left(1-\frac{x^2}{4};\frac{d}{2}+\frac{1}{2},\frac{1}{2}\right)^2\dd x,
    \end{align}
    \begin{align}
        \connf^{\star2\star2\cdot2}\left(\orig\right) = \int\connf^{\star2 }(x)^3\dd x &= \frac{\Gamma\left(\frac{d}{2}+1\right)^3}{\Gamma\left(\frac{1}{2}\right)^3\Gamma\left(\frac{d}{2}+\frac{1}{2}\right)^3}\mathfrak{S}_{d-1}\int^{2R}_0 r^{d-1}B\left(1-\frac{r^2}{4R^2};\frac{d}{2}+\frac{1}{2},\frac{1}{2}\right)^3\dd r\nonumber\\
        &= \frac{\Gamma\left(\frac{d}{2}+1\right)^3}{\Gamma\left(\frac{1}{2}\right)^3\Gamma\left(\frac{d}{2}+\frac{1}{2}\right)^3}\mathfrak{S}_{d-1}R^d\int^2_0 x^{d-1}B\left(1-\frac{x^2}{4};\frac{d}{2}+\frac{1}{2},\frac{1}{2}\right)^3\dd x.
    \end{align}
    As before, using $R^d\mathfrak{S}_{d-1} = d$ gives the result.

    Evaluating $\connf^{\star 5}\left(\orig\right)$, $\connf^{\star 6}\left(\orig\right)$, and  $\connf^{\star1\star2\cdot3}\left(\orig\right)$ is more challenging than the above expressions because we don't have such a nice expression for $\connf^{\star 3}(x)$ as we did for $\connf^{\star 2}(x)$. We can nevertheless use Fourier transforms to get an expression. Using the well-known expression 
    \begin{equation}
        J_{\nu}(x) = \frac{x^\nu}{2^\nu \Gamma\left(\frac{1}{2}\right)\Gamma\left(\nu+\frac{1}{2}\right)}\int^\pi_0\e^{ix\cos \theta}\left(\sin \theta\right)^{2\nu}\dd \theta,\qquad \mathrm{Re}~\nu\geq -\frac{1}{2}
    \end{equation}
    from \cite[p.360,~Eqn.(9.1.20)]{abramowitz1970handbook}, we can write
    \begin{multline}
        \connf^{\star 3}(x) = \frac{\mathfrak{S}_{d-2}}{\left(2\pi\right)^d}\int^\infty_{0} k^{d-1}\fconnf(k)^3\left(\int^\pi_0\e^{ik\abs*{x}\cos \theta}\left(\sin \theta\right)^{d-2}\dd \theta\right)\dd k \\= \left(2\pi\right)^d R^{\frac{3}{2}d}\abs*{x}^{1-\frac{d}{2}}\int^\infty_0k^{-d}\left(J_{\frac{d}{2}}\left(kR\right)\right)^3 J_{\frac{d}{2}-1}\left(k\abs*{x}\right)\dd k.
    \end{multline}
    Using this expression with the expression for $\connf^{\star 2}(x)$ used previously then gives the result:
    \begin{align}
        \connf^{\star 1 \star 2\cdot 3}\left(\orig\right) &= \int\connf(x)\connf^{\star2}(x)\connf^{\star 3}(x)\dd x \nonumber\\&= \frac{\Gamma\left(\frac{d}{2}+1\right)}{\Gamma\left(\frac{1}{2}\right)\Gamma\left(\frac{d}{2}+\frac{1}{2}\right)}\left(2\pi\right)^d R^{\frac{3}{2}d}\mathfrak{S}_{d-1}\int^R_0 r^{d-1}r^{1-\frac{d}{2}}B\left(1-\frac{r^2}{4R^2};\frac{d}{2}+\frac{1}{2},\frac{1}{2}\right)\nonumber\\
        &\hspace{7cm}\times\left(\int^\infty_0k^{-d}\left(J_{\frac{d}{2}}\left(kR\right)\right)^3 J_{\frac{d}{2}-1}\left(kr\right)\dd k\right)\dd r\nonumber\\
        &= d2^d\frac{\Gamma\left(\frac{d}{2}+1\right)^3}{\Gamma\left(\frac{1}{2}\right)\Gamma\left(\frac{d}{2}+\frac{1}{2}\right)}\int^1_0x^{\frac{d}{2}}B\left(1-\frac{x^2}{4};\frac{d}{2}+\frac{1}{2},\frac{1}{2}\right)\left(\int^\infty_0 k^{-d}\left(J_{\frac{d}{2}}(k)\right)^3J_{\frac{d}{2}-1}(kx) \dd k\right)\dd x,
    \end{align}
    where we explicitly use $R^d=\pi^{-\frac{d}{2}}\Gamma\left(\frac{d}{2}+1\right)$. Writing $\connf^{\star5}\left(\orig\right) = \int\connf^{\star2}(x)\connf^{\star 3}(x)\dd x$ and using the same strategy gives its result.

    Getting the expression for $\connf^{\star 6}\left(\orig\right)$ requires an expression for $\connf^{\star 4}\left(x\right)$. Using the same strategy as for $\connf^{\star 3}\left(x\right)$ above, we get
     \begin{multline}
        \connf^{\star 4}(x) = \frac{\mathfrak{S}_{d-2}}{\left(2\pi\right)^d}\int^\infty_{0} k^{d-1}\fconnf(k)^4\left(\int^\pi_0\e^{ik\abs*{x}\cos \theta}\left(\sin \theta\right)^{d-2}\dd \theta\right)\dd k \\= \left(2\pi\right)^{\frac{3}{2}d} R^{2d}\abs*{x}^{1-\frac{d}{2}}\int^\infty_0k^{-\frac{3}{2}d}\left(J_{\frac{d}{2}}\left(kR\right)\right)^4 J_{\frac{d}{2}-1}\left(k\abs*{x}\right)\dd k.
    \end{multline}
    Using this with $\connf^{\star6}\left(\orig\right) = \int\connf^{\star2}(x)\connf^{\star 4}(x)\dd x$ then gives the required expression.
\end{proof}

We now turn towards asymptotic values of the terms appearing in Lemma \ref{lem:BooleanCalcExpressions}. 
For the terms $\connf^{\star 3}\left(\orig\right)$ and $\connf^{\star 4}\left(\orig\right)$, the asymptotics have already been worked out.
\begin{lemma}
    For the Hyper-Sphere RCM,
    \begin{align}
        \connf^{\star 3}\left(\orig\right) &\sim \left(\frac{27}{2\pi d}\right)^{\frac{1}{2}}\left(\frac{3}{4}\right)^\frac{d}{2},\\ 
        \connf^{\star 4}\left(\orig\right) &\sim \left(\frac{32}{3\pi d}\right)^{\frac{1}{2}}\left(\frac{16}{27}\right)^\frac{d}{2}.
    \end{align}
\end{lemma}

\begin{proof}
    These follow from the calculations in \cite{luban1982third,joslin1982third}. 
\end{proof}

\begin{figure}
    \centering
    \begin{subfigure}[b]{\textwidth}
    \includegraphics{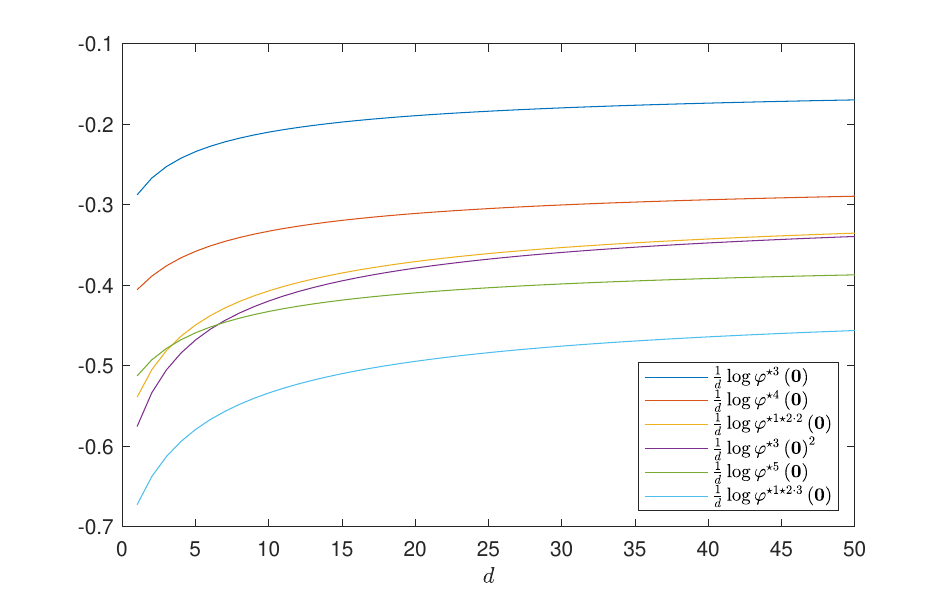}
    \caption{Plot of the larger diagrams}
    \label{fig:TopHalfPlot}
    \end{subfigure}
    \hfill
    \begin{subfigure}[b]{\textwidth}
    \includegraphics{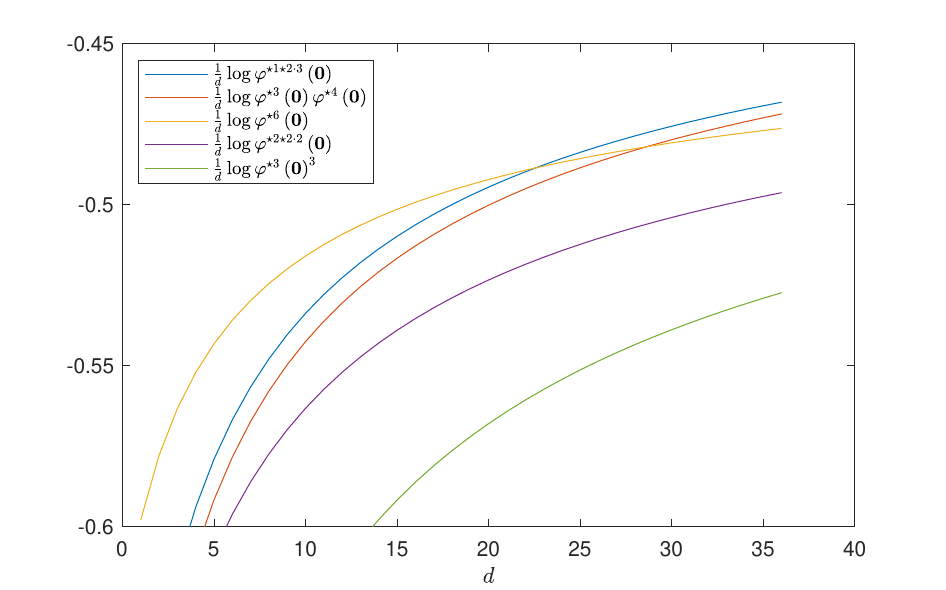}
    \caption{Plot of the smaller diagrams}
    \label{fig:BottomHalfPlot}
    \end{subfigure}
    \caption{Plots of $\frac{1}{d}\log\left(\cdot\right)$ for each of the diagrams for the Hyper-Sphere RCM. For comparison, $\frac{1}{d}\log\connf^{\star1\star2\cdot3}\left(\orig\right)$ is represented in both plots - it is the smallest of the larger diagrams and the largest of the smaller diagrams in the higher dimensions.}
    \label{fig:DiagramSizes}
\end{figure}

\begin{figure}
    \centering
    \begin{subfigure}[b]{\textwidth}
    \centering
    \includegraphics{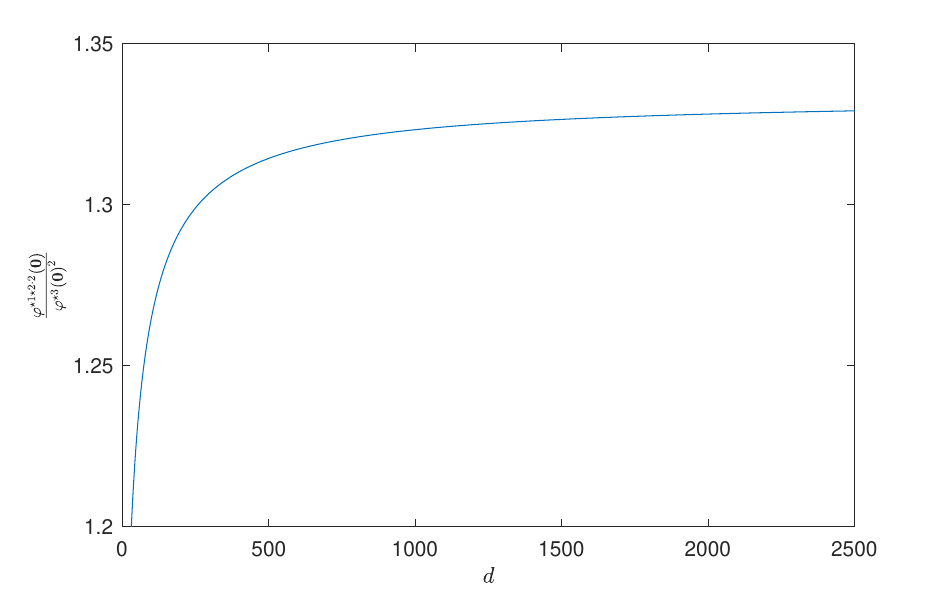}
    \caption{Plot of the ratio $\sfrac{\connf^{\star 1\star2\cdot2}\left(\orig\right)}{\connf^{\star3}\left(\orig\right)^2}$}
    \label{fig:Comparison122vs3x3}
    \end{subfigure}
    \hfill
    \begin{subfigure}[b]{\textwidth}
    \centering
    \includegraphics{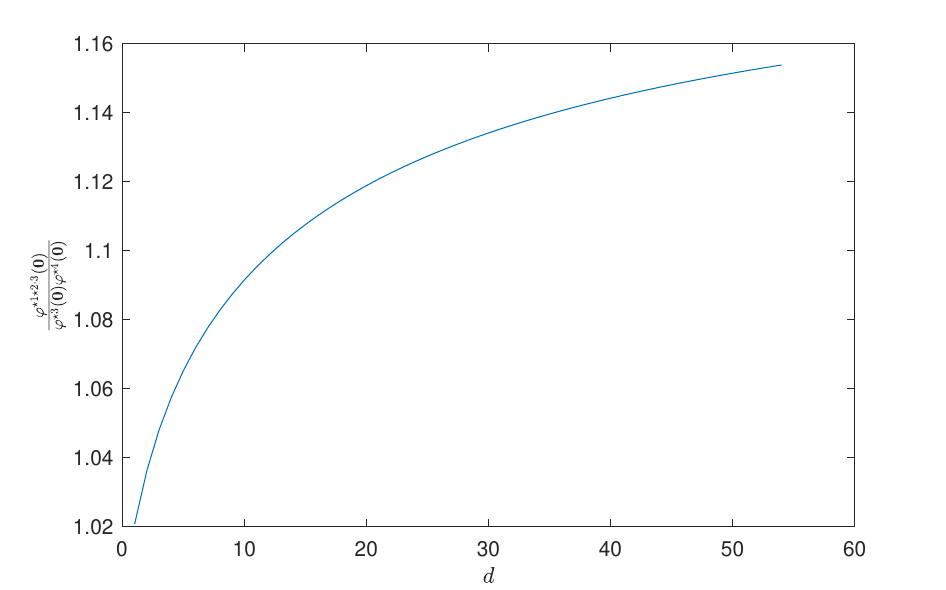}
    \caption{Plot of the ratio $\sfrac{\connf^{\star 1\star2\cdot3}\left(\orig\right)}{\connf^{\star3}\left(\orig\right)\connf^{\star4}\left(\orig\right)}$}
    \label{fig:Comparison123vs3x4}
    \end{subfigure}
    \hfill
    \caption{Plots of the ratio of diagrams of similar sizes for the Hyper-Sphere RCM.}
    \label{fig:Comparison}
\end{figure}

\begin{figure}
    \centering
    \begin{subfigure}[b]{\textwidth}
    \centering
    \includegraphics{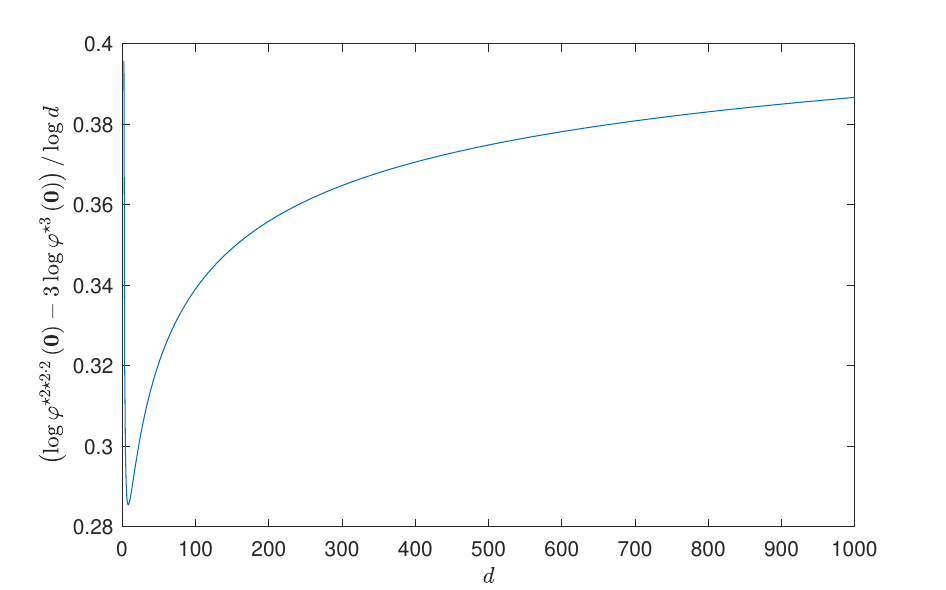}
    \caption{Plot of the ratio $\sfrac{\left(\log\varphi^{\star 2\star 2\cdot 2}\left(\mathbf{0}\right)-3\log\varphi^{\star 3}\left(\mathbf{0}\right)\right)}{\log d}$}
    \label{fig:Comparison222vs3x3x3}
    \end{subfigure}
    \hfill
    \begin{subfigure}[b]{\textwidth}
    \centering
    \includegraphics{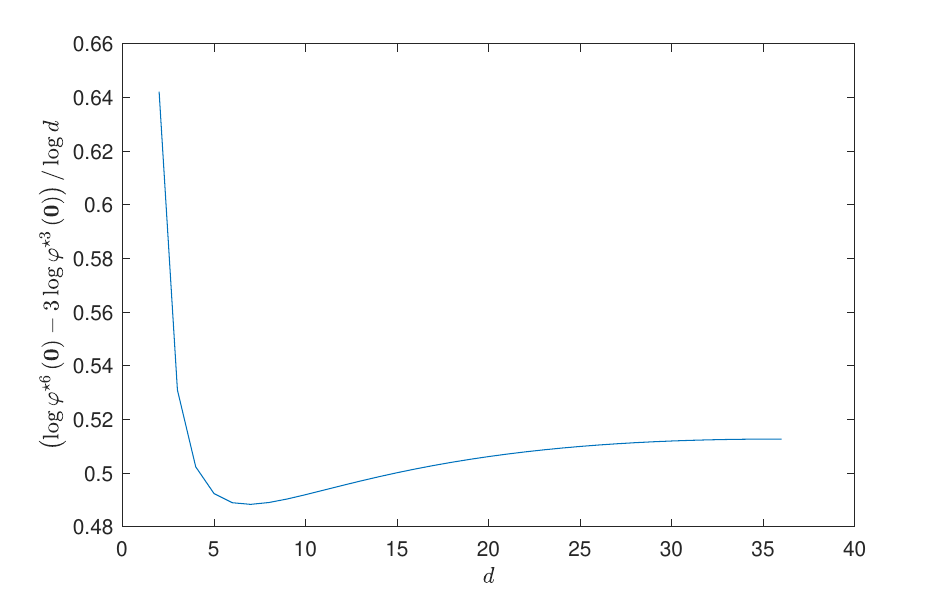}
    \caption{Plot of the ratio $\sfrac{\left(\log\varphi^{\star 6}\left(\mathbf{0}\right)-3\log\varphi^{\star 3}\left(\mathbf{0}\right)\right)}{\log d}$}
    \label{fig:Comparison6vs3x3x3}
    \end{subfigure}
    \hfill
    \caption{Plots relating $\connf^{\star 6}\left(\orig\right)$, $\connf^{\star 2\star 2\cdot 2}\left(\orig\right)$, and $\connf^{\star 3}\left(\orig\right)^3$ for the Hyper-Sphere RCM. These suggest the possibility that these three terms only differ by a polynomial factor in $d$.}
    \label{fig:ComparisonSmallTerms}
\end{figure}

\begin{remark}
\label{rem:BooleanOrderofTerms}
    These asymptotics naturally also give the asymptotics of $\connf^{\star 3}\left(\orig\right)^2$, $\connf^{\star 3}\left(\orig\right)^3$, and $\connf^{\star 3}\left(\orig\right)\connf^{\star 4}\left(\orig\right)$. For $\connf^{\star 5}\left(\orig\right)$, $\connf^{\star 6}\left(\orig\right)$, $\connf^{\star 1\star2\cdot2}\left(\orig\right)$, $\connf^{\star 2\star2\cdot2}\left(\orig\right)$, and $\connf^{\star 1\star2\cdot3}\left(\orig\right)$ we don't have any rigorous description of their asymptotic behaviour. Nevertheless we can use our expressions from Lemma~\ref{lem:BooleanCalcExpressions} and numerical integration to calculate their values for a range of dimensions. Figure~\ref{fig:DiagramSizes} presents the results of these calculations. Here we used MATLAB to plot $\frac{1}{d}\log\left(\cdot\right)$ (where $\log$ is the \emph{natural} logarithm) for each of our diagrams against the dimension $d$. We chose this function of the diagrams because if a diagram was of the form $A(d)\varrho^d$ for some constant $\varrho>0$ and some slowly varying $A(d)$, then our plot should approach $\log\varrho$ as $d\to\infty$. The data in Figure~\ref{fig:TopHalfPlot} are consistent with this behaviour (indeed we know it to be true for $\connf^{\star 3}\left(\orig\right)$, $\connf^{\star 4}\left(\orig\right)$ and $\connf^{\star 3}\left(\orig\right)^2$). We only plot the data up to $d=50$ because the calculations of $\connf^{\star 5}\left(\orig\right)$ and $\connf^{\star 1\star2\star3}\left(\orig\right)$ fail for $d>54$ - we comment on this more later. The data in Figure~\ref{fig:BottomHalfPlot} appear a little less definitive, but the authors argue these are still consistent with the hypothesised behaviour (we know it to be true for $\connf^{\star 3}\left(\orig\right)\connf^{\star 4}\left(\orig\right)$ and $\connf^{\star 3}\left(\orig\right)^3$). Note that the vertical scale is over a much narrower range than in Figure~\ref{fig:TopHalfPlot}, which gives the false impression that the plots are increasing with $d$ faster than they indeed are. We are also further restricting the domain of $d$ to $d\leq 36$. This is because the calculation of $\connf^{\star 6}\left(\orig\right)$ fails for $d>36$. \hfill$\diamond$
\end{remark}

\begin{remark}
\label{rem:BooleanCalculationDifficulties}
    We comment here on our choices of the range of dimensions $d$ presented in the data in Figure~\ref{fig:DiagramSizes}. We found that the limiting factor in our ability to calculate the expressions in Lemma~\ref{lem:BooleanCalcExpressions} were the prefactors of powers of $2$ and gamma functions. If we wanted to use \eqref{eqn:BooleanLoopGeneral} to calculate $\connf^{\star 6}\left(\orig\right)$ for $d=25$ we would have to deal with $d2^{2d}\Gamma\left(\tfrac{d}{2}+1\right)^4 \approx 2.41\times 10^{53}$, while $\connf^{\star 6}\left(\orig\right)\approx 5.34\times10^{-6}$. Fortunately MATLAB has the function \texttt{betainc($x$,$a$,$b$)}, which calculates the (normalised) incomplete beta function
    \begin{equation}
        \frac{\Gamma\left(a+b\right)}{\Gamma(a)\Gamma(b)}\int^x_0 t^{a-1}\left(1-t\right)^{b-1}\dd t = \frac{\Gamma\left(a+b\right)}{\Gamma(a)\Gamma(b)}B\left(x;a,b\right).
    \end{equation}
    This \texttt{betainc} function is more efficient at dealing with the different sizes of the prefactor and integral than our na\"ive attempts, and this is why we put the extra effort in Lemma~\ref{lem:BooleanCalcExpressions} to include factors of $B\left(x;a,b\right)$. In particular, this makes $\connf^{\star 3}\left(\orig\right)$ very easy to calculate: MATLAB got to over $d=5000$ before it produced an error (for $d=5000$, $\connf^{\star 3}\left(\orig\right)\approx 1.32\times 10^{-314}$). We can also calculate our expressions for $\connf^{\star 4}\left(\orig\right)$, $\connf^{\star 1\star2\cdot2}\left(\orig\right)$, and $\connf^{\star 2\star2\cdot 2}\left(\orig\right)$ over dimension $d=1000$. Unfortunately the use of \texttt{betainc} does not deal with the whole prefactor for $\connf^{\star 5}\left(\orig\right)$, $\connf^{\star 6}\left(\orig\right)$, and $\connf^{\star 1\star2\cdot3}\left(\orig\right)$, and this affects the dimension we can run up to. For $\connf^{\star 5}\left(\orig\right)$ and $\connf^{\star 1\star2\cdot3}\left(\orig\right)$ we can run up to $d=54$ (where they are $\approx9.06\times 10^{-10}$ and $\approx2.29\times10^{-11}$ respectively). For $\connf^{\star 6}\left(\orig\right)$ we can only run to $d=36$, where find $\connf^{\star 6}\left(\orig\right)\approx3.58\times10^{-8}$. \hfill $\diamond$
\end{remark}

\begin{remark}
\label{rem:BooleanRatios}
    Upon inspecting Figure~\ref{fig:TopHalfPlot}, the plots of $\connf^{\star1\star2\cdot2}\left(\orig\right)$ and $\connf^{\star3}\left(\orig\right)^2$ appear very close together. The plots of $\connf^{\star1\star2\cdot3}\left(\orig\right)$ and $\connf^{\star3}\left(\orig\right)\connf^{\star4}\left(\orig\right)$ in Figure~\ref{fig:BottomHalfPlot} also appear to be tracking closely together. In Figure~\ref{fig:Comparison} we plot how the ratio of these similar terms vary with dimension.

    Since we are able to evaluate $\connf^{\star1\star2\cdot2}\left(\orig\right)$ and $\connf^{\star3}\left(\orig\right)^2$ to relatively high dimensions, we are able to plot their ratio all the way up to $d=2500$ in Figure~\ref{fig:Comparison122vs3x3}. From this plot it is very tempting to suggest that their ratio is approaching a finite and positive limit. In fact, since $\sfrac{\connf^{\star 1\star2\cdot2}\left(\orig\right)}{\connf^{\star3}\left(\orig\right)^2}\approx1.329$ at $d=2500$, it is tempting to suggest that the ratio approaches $\tfrac{4}{3}$ as $d\to\infty$. Since we rigorously have the asymptotics of $\connf^{\star3}\left(\orig\right)$, this would imply the asymptotics of $\connf^{\star 1\star2\cdot2}\left(\orig\right)$.

    We are not able to evaluate $\connf^{\star1\star2\cdot3}\left(\orig\right)$ to similarly high dimensions - we can only reach $d=54$. Nevertheless, the slope of the plot in Figure~\ref{fig:Comparison123vs3x4} is shallowing and it is tempting to suggest that the ratio $\sfrac{\connf^{\star 1\star2\cdot2}\left(\orig\right)}{\connf^{\star3}\left(\orig\right)\connf^{\star4}\left(\orig\right)}$ approaches a finite and positive limit. While we don't conjecture a value for the limit here, the existence of such a limit would allow us to find the asymptotic scale of $\connf^{\star 1\star2\cdot3}\left(\orig\right)$.

    If we look at the ratio of the other pairs of diagrams it is usually very clear that one is far larger than the other, with the ratio apparently growing at an exponential rate. The only exceptions are the trio of $\connf^{\star 6}\left(\orig\right)$, $\connf^{\star 2\star 2\cdot 2}\left(\orig\right)$, and $\connf^{\star 3}\left(\orig\right)^3$. While the ratios appear to be growing for each pair in this trio, the rate seems to be slowing. If $\connf^{\star 2\star 2\cdot 2}\left(\orig\right)$ and $\connf^{\star 3}\left(\orig\right)^3$ were both decaying at the same exponential rate but had different polynomial corrections, then we would have $\sfrac{\left(\log\varphi^{\star 2\star 2\cdot 2}\left(\mathbf{0}\right)-3\log\varphi^{\star 3}\left(\mathbf{0}\right)\right)}{\log d}$ approaching a non-zero limit as $d\to\infty$. In Figure~\ref{fig:ComparisonSmallTerms} we plot this comparison for the two independent pairs in the trio, and it indeed seems plausible that the plots are approaching a non-zero limit. Nevertheless, these three terms look to be far smaller than the $\connf^{\star1\star2\cdot3}\left(\orig\right)$ and $\connf^{\star3}\left(\orig\right)\connf^{\star 4}\left(\orig\right)$ terms, and so will both be negligible for our discussion. \hfill $\diamond$
\end{remark}

The observations made in Remarks~\ref{rem:BooleanOrderofTerms}-\ref{rem:BooleanRatios} and the plots in Figures~\ref{fig:DiagramSizes} and \ref{fig:Comparison} allow us to make the following conjecture. We use the notation $f\gg g$ to indicate $\frac{f(d)}{g(d)}\to\infty$, and $f\asymp g$ to indicate $\frac{f(d)}{g(d)}$ and $\frac{g(d)}{f(d)}$ are both bounded as $d\to\infty$.

\begin{conjecture}\label{conj:HyperDiscModel}
    For the Hyper-Sphere RCM, as $d\to\infty$,
    \begin{equation}
        \connf^{\star 3}\left(\orig\right) \gg \connf^{\star 4}\left(\orig\right) \gg \connf^{\star 1\star 2\cdot 2}\left(\orig\right) \asymp \left(\connf^{\star 3}\left(\orig\right)\right)^2 \gg \connf^{\star 5}\left(\orig\right) \gg \connf^{\star 1 \star 2\cdot 3}\left(\orig\right) \asymp \connf^{\star 3}\left(\orig\right)\connf^{\star 4}\left(\orig\right),
    \end{equation}
    and
    \begin{equation}
        \connf^{\star 6}\left(\orig\right) + \connf^{\star 2\star 2\cdot 2}\left(\orig\right) + \left(\connf^{\star 3}\left(\orig\right)\right)^3 = \LandauBigO{\connf^{\star 3}\left(\orig\right)\connf^{\star 4}\left(\orig\right)}.
    \end{equation}
    Therefore
    \begin{equation}
        \phiint\lambda_c = 1 + \frac{1}{\phiint^2}\connf^{\star 3}\left(\orig\right) + \frac{3}{2}\frac{1}{\phiint^3}\connf^{\star 4}\left(\orig\right)  - \frac{5}{2}\frac{1}{\phiint^3}\connf^{\star 1\star 2\cdot2}\left(\orig\right) + 2\frac{1}{\phiint^4}\left(\connf^{\star 3}\left(\orig\right)\right)^2 + 2\frac{1}{\phiint^4}\connf^{\star 5}\left(\orig\right) + \LandauBigO{\frac{1}{d}\left(\frac{2}{3}\right)^d}.
    \end{equation}
\end{conjecture}

Note that this would be a different order of terms than that we found for the Hyper-Cube RCM in Corollary~\ref{cor:cube}.

\subsection{Hyper-Cube Calculations}

Recall that for side length $L>0$, the Hyper-Cubic RCM is defined by having
    \begin{equation}
        \connf(x) = \prod^d_{j=1}\Id_{\left\{\abs*{x_j}\leq L/2\right\}},
    \end{equation}
where $x=\left(x_1,\ldots,x_d\right)\in\Rd$. Throughout this section we choose a scaling of $\Rd$ such that $L=1$.

\begin{lemma}
\label{lem:HyperCubeCalculations}
    For the Hyper-Cube RCM with side length $L=1$,
    \begin{align}
        \connf^{\star 3}\left(\orig\right) &= \left(\frac{3}{4}\right)^d = \left(0.75\right)^d,\\
        \connf^{\star 4}\left(\orig\right)  &= \left(\frac{2}{3}\right)^d \approx \left(0.66667\right)^d,\\
        \connf^{\star 5}\left(\orig\right) &= \left(\frac{115}{192}\right)^d \approx\left(0.59896\right)^d,\\
        \connf^{\star 1\star 2\cdot 2}\left(\orig\right) &= \left(\frac{7}{12}\right)^d \approx \left(0.58333\right)^d,\\
        \connf^{\star 3}\left(\orig\right)^2 &= \left(\frac{9}{16}\right)^d =\left(0.5625\right)^d,\\ 
        \connf^{\star 6}\left(\orig\right) &= \left(\frac{11}{20}\right)^d = \left(0.55\right)^d,\\
        \connf^{\star 7}\left(\orig\right) & = \left(\frac{5887}{11520}\right)^d \approx \left(0.51102\right)^d,\\
        \connf^{\star 1 \star 2\cdot 3}\left(\orig\right)  &= \left(\frac{49}{96}\right)^d \approx \left(0.51042\right)^d,\\
        \connf^{\star 2\star 2\cdot 2}\left(\orig\right) &= \left(\frac{1}{2}\right)^d = \left(0.5\right)^d,\\
        \connf^{\star 3}\left(\orig\right)\connf^{\star 4}\left(\orig\right) &= \left(\frac{1}{2}\right)^d = \left(0.5\right)^d,\\
        \connf^{\star 8}\left(\orig\right) & = \left(\frac{151}{315}\right)^d \approx \left(0.47937\right)^d,\\
        \connf^{\star 3}\left(\orig\right)^3 &= \left(\frac{27}{64}\right)^d \approx\left(0.42188\right)^d.
    \end{align}
\end{lemma}

\begin{proof}
    First note that the hyper-cubic adjacency function factorises into the $d$ dimensions:
    \begin{equation}
        \connf(x) = \prod^d_{i=1}\Id_{\left\{\abs*{x_i}<\frac{1}{2}\right\}},
    \end{equation}
    where $x=\left(x_1,x_2,\ldots,x_d\right)$. Therefore to find the desired expressions, we only need to evaluate them for dimension $1$, and then take the result to the power $d$ to get the result for dimension $d$. Let us denote the $1$-dimensional adjacency function $\connf_1\colon \R \to \left[0,1\right]$,
    \begin{equation}
        \connf_1(x) = \begin{cases}
        1&\colon \abs*{x}<\frac{1}{2}\\
        0&\colon \abs*{x}\geq\frac{1}{2}.
        \end{cases}
    \end{equation}
    By direct calculation (these can be easily verified by Mathematica, for example), one finds
    \begin{align}
        \connf_1^{\star 2}(x) &= \begin{cases}
            1-\abs*{x} &\colon \abs*{x}<1\\
            0 &\colon \abs*{x}\geq 1,
        \end{cases}\\
        \connf_1^{\star 3}(x) &=\begin{cases}
            \frac{1}{4}\left(3-4x^2\right) &\colon \abs*{x}<\frac{1}{2}\\
            \frac{1}{8}\left(3-2\abs*{x}\right)^2 &\colon \frac{1}{2}\leq \abs*{x}<\frac{3}{2}\\
            0 &\colon \abs*{x}\geq \frac{3}{2},
        \end{cases}\\
        \connf_1^{\star 4}(x) &=\begin{cases}
            \frac{1}{6}\left(4-6x^2+3\abs*{x}^3\right) &\colon \abs*{x}<1\\
            \frac{1}{6}\left(2-\abs*{x}\right)^3 &\colon 1\leq \abs*{x}<2\\
            0 &\colon \abs*{x}\geq 2.
        \end{cases}
    \end{align}
    In particular, this means $\connf_1^{\star 3}\left(\orig\right) =\frac{3}{4}$ and $\connf_1^{\star 4}\left(\orig\right) = \frac{2}{3}$. Taking these to the power $d$ returns the required results for $\connf^{\star 3}\left(\orig\right)$ and $\connf^{\star 4}\left(\orig\right)$. These also give the results for $\connf^{\star 3}\left(\orig\right)^2$, $\connf^{\star 3}\left(\orig\right)^3$ and $\connf^{\star 3}\left(\orig\right)\connf^{\star 4}\left(\orig\right)$.

    Then let us observe and calculate
    \begin{align}
        \connf_1^{\star 5}\left(\orig\right) &= \int^1_{-1}\connf_1^{\star 2}(x)\connf_1^{\star 3}(x)\dd x = \frac{115}{192} \approx 0.59896,\\
        \connf_1^{\star 6}\left(\orig\right)&= \int^\frac{3}{2}_{-\frac{3}{2}}\connf_1^{\star 3}(x)\connf_1^{\star 3}(x)\dd x = \frac{11}{20} = 0.55,\\
        \connf_1^{\star 7}\left(\orig\right) &= \int^\frac{3}{2}_{-\frac{3}{2}}\connf_1^{\star 3}(x)\connf_1^{\star 4}(x)\dd x = \frac{5887}{11520}\approx 0.51102,\\
        \connf_1^{\star 8}\left(\orig\right) &= \int^2_{-2}\connf_1^{\star 4}(x)\connf_1^{\star 4}(x)\dd x = \frac{151}{315}\approx 0.47937.
    \end{align}
    Similarly, we find
    \begin{align}
        \connf_1^{\star 1\star 2\cdot 2} \left(\orig\right) &= \int^\frac{1}{2}_{-\frac{1}{2}}\connf_1^{\star 2}(x)^2\dd x = \frac{7}{12}\approx 0.58333\\
        \connf_1^{\star 2\star 2\cdot 2}\left(\orig\right) &= \int^1_{-1}\connf_1^{\star 2}(x)^3\dd x = \frac{1}{2} = 0.5\\
        \connf_1^{\star 1\star 2\cdot 3}\left(\orig\right) &= \int^\frac{1}{2}_{-\frac{1}{2}}\connf_1^{\star 2}(x)\connf_1^{\star 3}(x)\dd x = \frac{49}{96} \approx 0.51042.
    \end{align}
    Finally taking these values to the $d^{\mathrm {th}}$ power gives the required results.
\end{proof}

\begin{lemma}
    The Hyper-Cube RCM satisfies Assumptions~\ref{Assumption} and \ref{AssumptionBeta}.
\end{lemma}

\begin{proof}
    For Assumption~\ref{Assump:DecayBound}, recall that
    \begin{equation}
        \connf^{\star 2}(x) = \prod_{i=1}^d \left(1-\abs*{x_i}\right)\Id_{\left\{\abs*{x_i}\leq 1\right\}},
    \end{equation}
    where $x=\left(x_1,x_2,\ldots,x_d\right)$. In conjunction with Lemma~\ref{lem:HyperCubeCalculations}, we see that Assumption~\ref{Assump:DecayBound} is satisfied with $g(d)=\left(\frac{3}{4}\right)^d$. 
    
    For \ref{Assump:QuadraticBound} we note that 
    \begin{equation}
    \label{eqn:H-CubeFourier}
        \fconnf(k) = \prod^d_{i=1}\left(\frac{2}{k_i}\sin\frac{k_i}{2}\right),
    \end{equation}
    where $k=\left(k_1,k_2,\ldots,k_d\right)$. Since $\sin x \leq x - \frac{1}{6}x^3 + \frac{1}{120}x^5$ for all $x\in\R$,
    \begin{equation}
        \fconnf(k) \leq \prod^d_{i=1}\left(1-\frac{1}{24}k_i^2 + \frac{1}{1920}k_i^4\right).
    \end{equation}
    Therefore for $\max_i\abs*{k_i}\leq 3$ we have 
    \begin{align}
        \fconnf(k) &\leq \prod^d_{i=1}\left(1-\frac{71}{1920}k_i^2\right)\nonumber\\
        &= 1 - \frac{71}{1920}\abs*{k}^2 + \left(\frac{71}{1920}\right)^2\sum_{\substack{i,j=1\\ i<j}}^dk_i^2k_j^2 - \left(\frac{71}{1920}\right)^3\sum_{\substack{i,j,l=1\\ i<j<l}}^dk_i^2k_j^2k_l^2 + \ldots \pm \left(\frac{71}{1920}\right)^dk_1^2\ldots k_d^2\nonumber\\
        &\leq 1 - \frac{71}{1920}\abs*{k}^2 + \left(\frac{71}{1920}\right)^2\abs*{k}^4 + 0 + \left(\frac{71}{1920}\right)^4\abs*{k}^8 + 0 + \ldots +\begin{cases}
        \left(\frac{71}{1920}\right)^{d}\abs*{k}^{2d} &\colon d \text{ is even}\\
        \left(\frac{71}{1920}\right)^{d-1}\abs*{k}^{2d-2} &\colon d \text{ is odd.}
        \end{cases}
    \end{align}
    Here we have bounded the later negative terms above by $0$, and bounded the positive terms above by powers of $\abs*{k}^4$. Therefore if $\abs*{k}^2 < \frac{1920}{71}$ then we have
    \begin{equation}
        \fconnf(k) \leq  1 - \frac{71}{1920}\abs*{k}^2 + \sum^\infty_{n=1}\left(\frac{71}{1920}\right)^{2n}\abs*{k}^{4n} = 1 - \frac{71}{1920}\abs*{k}^2 + \left(\frac{71}{1920}\right)^{2}\abs*{k}^{4}\left(1 - \left(\frac{71}{1920}\right)^{2}\abs*{k}^{4}\right)^{-1}.
    \end{equation}
    Note that $\abs*{k}\leq 3 \implies \max_i\abs*{k_i}\leq 3$, and therefore $\abs*{k}\leq 3$ also implies
    \begin{equation}
        \fconnf(k) \leq  1 - \frac{71}{1920}\abs*{k}^2\left(1- \frac{213}{640}\left(1 - \left(\frac{213}{640}\right)^{2}\right)^{-1}\right) \leq 1 - \frac{5}{8}\times\frac{71}{1920}\abs*{k}^2,
    \end{equation}
    where we have used $\frac{213}{640}<\frac{1}{3}$. Therefore we have constants $b,c_1>0$ such that $\abs*{k}\leq b$ implies that $\fconnf(k) \leq 1 - c_1\abs*{k}^2$.

    From \eqref{eqn:H-CubeFourier} it is clear that $\fconnf(k)$ is radially decreasing and non-negative on the set $\left\{k\in\Rd\colon \max_{i}\abs*{k_i}\leq 2\pi\right\}$. On the other hand if there exists $i^*\in\left\{1,2,\ldots,d\right\}$ such that $\abs*{k_{i^*}}>2\pi$, then $\abs*{\fconnf(k)}<\frac{1}{\pi}$. Therefore if $\abs*{k}>3$ we can bound
    \begin{equation}
        \fconnf(k) \leq 1 - \frac{5}{3}\times\frac{71}{1920}\times3^2 < \frac{1}{2}.
    \end{equation}
    We have therefore proven that \ref{Assump:QuadraticBound} holds with $b=3$, $c_1= \frac{5}{3}\times\frac{71}{1920}$, and $c_2=\frac{1}{2}$.

    Lemma~\ref{lem:HyperCubeCalculations} and our above observation that we can have $g(d)=\left(\frac{3}{4}\right)^d$ ensures that Assumption~\ref{AssumptionBeta} holds.
\end{proof}

\subsection{Gaussian Calculations}
Recall that for $\sigma^2>0$ and $0<\Acal\leq \left(2\pi\sigma^2\right)^\frac{d}{2}$, the Gaussian RCM is defined by having
        \begin{equation}
            \connf\left(x\right) = \frac{\Acal}{\left(2\pi \sigma^2\right)^\frac{d}{2}}\exp\left(-\frac{1}{2\sigma^2}\abs*{x}^2\right).
        \end{equation}

\begin{lemma}
\label{lem:GaussianCalculations}
    For the Gaussian RCM,
    \begin{align}
        \connf^{\star n}\left(\orig\right) &= \Acal^n\left(2n\pi\sigma^2\right)^{-\frac{d}{2}}\qquad\forall n\geq 1,\\
        \connf^{\star n_1 \star n_2 \cdot n_3}\left(\orig\right) &= \Acal^{n_1+n_2+n_3}\left(\left(n_1n_2 + n_1n_3 + n_2n_3\right)\left(2\pi\sigma^2\right)^2\right)^{-\frac{d}{2}}\qquad\forall n_1,n_2,n_3\geq 1.
    \end{align}
    In particular,
    \begin{align}
        \connf^{\star 1\star 2\cdot 2}\left(\orig\right) &= \Acal^5\left(32\pi^2\sigma^4\right)^{-\frac{d}{2}} = \Acal^5\left(8\times \left(2\pi\sigma^2\right)^2\right)^{-\frac{d}{2}}\\
        \connf^{\star 1 \star 2\cdot 3}\left(\orig\right) &= \Acal^6\left(44\pi^2\sigma^4\right)^{-\frac{d}{2}} = \Acal^6\left(11\times \left(2\pi\sigma^2\right)^2\right)^{-\frac{d}{2}}\\
        \connf^{\star 2\star 2\cdot 2}\left(\orig\right) &= \Acal^6\left(48\pi^2\sigma^4\right)^{-\frac{d}{2}} = \Acal^6\left(12\times \left(2\pi\sigma^2\right)^2\right)^{-\frac{d}{2}}.
    \end{align}
\end{lemma}

\begin{proof}
    Without loss of generality, we scale space so that $\phiint = \Acal = 1$.
    
    First we note that the convolution of two unit-mass Gaussian functions is itself a unit-mass Gaussian function whose ``variance'' parameter is the sum of the variance parameters of the two initial Gaussian functions:
    \begin{multline}
        \int_{\Rd}\frac{1}{\left(2\pi\sigma_1^2\right)^\frac{d}{2}}\exp\left(-\frac{1}{2\sigma_1^2}\abs*{x-y}^2\right)\frac{1}{\left(2\pi\sigma_2^2\right)^\frac{d}{2}}\exp\left(-\frac{1}{2\sigma_2^2}\abs*{y}^2\right) \dd y \\= \frac{1}{\left(2\pi\left(\sigma_1^2 + \sigma_2^2\right)\right)^\frac{d}{2}}\exp\left(-\frac{1}{2\left(\sigma_1^2+\sigma_2^2\right)}\abs*{x}^2\right).
    \end{multline}
    It therefore follows that
    \begin{equation}
        \connf^{\star n}(x) = \frac{1}{\left(2\pi n\sigma^2\right)^\frac{d}{2}}\exp\left(-\frac{1}{2n\sigma^2}\abs*{x}^2\right),
    \end{equation}
    and $\connf^{\star n}\left(\orig\right)=\left(2\pi n\sigma^2\right)^{-\frac{d}{2}}$.

    For the remaining expressions we write the pointwise product of two unit-mass Gaussian functions as a constant multiple of a unit-mass Gaussian function:
    \begin{multline}
        \frac{1}{\left(2\pi\sigma_1^2\right)^\frac{d}{2}}\exp\left(-\frac{1}{2\sigma_1^2}\abs*{x}^2\right)\frac{1}{\left(2\pi\sigma_2^2\right)^\frac{d}{2}}\exp\left(-\frac{1}{2\sigma_2^2}\abs*{x}^2\right)=\frac{1}{\left(4\pi^2\sigma_1^2\sigma_2^2\right)^\frac{d}{2}}\exp\left(-\frac{\sigma_1^2+\sigma_2^2}{2\sigma_1^2\sigma_2^2}\abs*{x}^2\right) \\= \frac{1}{\left(2\pi\left(\sigma_1^2+\sigma_2^2\right)\right)^\frac{d}{2}}\left(\frac{\sigma_1^2+\sigma_2^2}{2\pi\sigma_1^2\sigma_2^2}\right)^{\frac{d}{2}}\exp\left(-\frac{\sigma_1^2+\sigma_2^2}{2\sigma_1^2\sigma_2^2}\abs*{x}^2\right).
    \end{multline}
    Using this expression, we find
    \begin{multline}
        \connf^{\star n_1 \star n_2 \cdot n_3}\left(\orig\right) = \left(2\pi\sigma^2\left(n_2+n_3\right)\right)^{-\frac{d}{2}}\left(2\pi\sigma^2\left(n_1 + \frac{n_2n_3}{n_2+n_3}\right)\right)^{-\frac{d}{2}} \\= \left(4\pi^2\sigma^4\left(n_1n_2 + n_1n_3 + n_2n_3\right)\right)^{-\frac{d}{2}}.
    \end{multline}
    This produces the results.
\end{proof}

\begin{lemma}
\label{lem:GaussianAssumptions}
    The Gaussian RCM with $\liminf_{d\to\infty}\connf\left(\orig\right)^{\frac{1}{d}}>0$ satisfies Assumptions~\ref{Assumption} and \ref{AssumptionBeta}.
\end{lemma}

\begin{proof}
    For this proof we make the scaling choice that the total mass of the adjacency function in each dimension is set to be equal to $1$. Clearly this maps $\Acal \mapsto \widetilde{\Acal}\equiv1$, but since $\connf\left(\orig\right)=\Acal\left(2\pi\sigma^2\right)^{-\frac{d}{2}}$ is left invariant, we also have $\sigma\mapsto\widetilde{\sigma}=\sigma\Acal^{-\frac{1}{d}}$. The condition that $\liminf \connf\left(\orig\right)^{\frac{1}{d}}>0$ now means that $\limsup\widetilde{\sigma}<\infty$, and the trivial condition that $\connf\left(\orig\right)\leq 1$ means that $\widetilde{\sigma}^2\geq \sfrac
    {1}{2\pi}$.
    
    The results of Lemma~\ref{lem:GaussianCalculations} proves that \ref{Assump:DecayBound} holds with the choice $g(d)= \left(4\pi\widetilde{\sigma}^2\right)^{-\frac{d}{2}} = 2^{-\frac{d}{2}}\connf\left(\orig\right)$ and therefore $\beta(d) = 2^{-\frac{d}{8}}\connf\left(\orig\right)^{\frac{1}{4}}$ (here we use $\limsup\widetilde{\sigma}<\infty$ to get the appropriate form of $\beta$ from \eqref{eqn:betafromgfunction}). Now observe that the Fourier transform of $\connf(x)$ is given by
    \begin{equation}
        \fconnf(k) = \exp\left(-\frac{1}{2}\widetilde{\sigma}^2\norm*{k}_2^2\right) \leq \exp\left(-\frac{1}{4\pi}\norm*{k}_2^2\right),
    \end{equation}
    where the inequality follows from $\widetilde{\sigma}^2\geq \sfrac
    {1}{2\pi}$. Therefore \ref{Assump:QuadraticBound} holds.

    For Assumption~\ref{AssumptionBeta}, we first use Lemma~\ref{lem:GaussianCalculations} to see that $\connf^{\star 6}\left(\orig\right) = 6^{-\frac{d}{2}}\phiint^5\connf\left(\orig\right)$. Therefore \ref{Assump:ExponentialDecay} can be seen to hold with $\rho=6^{-\frac{1}{2}}\liminf \connf\left(\orig\right)^{\frac{1}{d}}>0$. This also provides a lower bound on $h(d)$. After noting that $\log \beta(d)<0$, we have
    \begin{equation}
        \frac{\log h(d)}{\log \beta(d)} \leq \frac{-\frac{d}{2}\log 6 + \log \connf\left(\orig\right)}{-\frac{d}{8}\log 2 + \frac{1}{4}\log\connf\left(\orig\right)} \leq \frac{4\log 6  - 8\log \connf\left(\orig\right)^{\frac{1}{d}}}{\log 2 - 2 \log \connf\left(\orig\right)^{\frac{1}{d}}}.
    \end{equation}
    Note that $\log \connf\left(\orig\right)^{\frac{1}{d}}\leq0$. By taking the derivative of the map $x\mapsto \frac{4\log 6 - 8x}{\log 2-2x}$ for $x\leq 0$ we can find that it is maximised at $x=0$. Therefore
    \begin{equation}
        \frac{\log h(d)}{\log \beta(d)} \leq 4\frac{\log 6}{\log 2} = 4\log_2 6.
    \end{equation}
    Since this is finite, we have proven Assumption~\ref{Assump:NumberBound}.
\end{proof}

\subsection{Coordinate-Cauchy Calculations}
Recall that for $\gamma>0$ and $0<\Acal\leq \left(\gamma\pi\right)^d$, the Coordinate-Cauchy RCM is defined by having
        \begin{equation}
            \connf(x) = \frac{\Acal}{\left(\gamma\pi\right)^d}\prod^d_{j=1}\frac{\gamma^2}{\gamma^2+x^2_j},
        \end{equation}
        where $x=\left(x_1,\ldots,x_d\right)\in\Rd$.
        
\begin{lemma}
    For the Coordinate-Cauchy RCM, 
    \begin{align}
        \connf^{\star n}\left(\orig\right) &= \Acal^n\left(\frac{1}{n\gamma\pi}\right)^d, \qquad\forall n\geq 1,\\
        \connf^{\star n_1\star n_2 \cdot n_3}\left(\orig\right) &= \Acal^{n_1+n_2+n_3}\left(\frac{n_1+n_2+n_3}{\left(n_1+n_2\right)\left(n_1+n_3\right)\left(n_2+n_3\right)\gamma^2\pi^2}\right)^d, \qquad\forall n_1,n_2,n_3\geq 1.
    \end{align}
    In particular,
    \begin{equation}
        \connf^{\star 1\star 2\cdot 2}\left(\orig\right) = \Acal^5\left(\frac{5}{36\gamma^2\pi^2}\right)^d,\qquad
        \connf^{\star1\star2\cdot3}\left(\orig\right) = \Acal^6\left(\frac{1}{10\gamma^2\pi^2}\right)^d,\qquad
        \connf^{\star2\star2\cdot2}\left(\orig\right) = \Acal^6\left(\frac{3}{32\gamma^2\pi^2}\right)^d.
    \end{equation}
\end{lemma}

\begin{proof}
    We begin with the simplification that $\phiint=\Acal$ is set to be equal to $1$ (by a spatial scaling choice).
    
    Like for the Hyper-Cubic model, the factorisable structure of the adjacency function means that we only need to evaluate the answers for the $1$-dimensional model, and then we can take the result to the power $d$ to get the $d$-dimensional answer. Let the $1$-dimensional adjacency function be denoted
    \begin{equation}
        \connf_1(x) = \frac{\gamma}{\pi\left(\gamma^2+x^2\right)}.
    \end{equation}
    By well-known complex analysis techniques, the Fourier transform of this function is given by
    \begin{equation}
        \fconnf_1(k) = \e^{-\gamma\abs*{k}}
    \end{equation}
    for $k\in\R$. Then by the Fourier inversion formula, for $n\geq 1$,
    \begin{equation}
        \connf_1^{\star n}\left(\orig\right) = \frac{1}{2\pi}\int^\infty_{-\infty}\e^{-n\gamma\abs*{k}}\dd k = \frac{1}{\gamma\pi}\int^\infty_{0}\e^{-n k}\dd k = \frac{1}{n\gamma\pi}.
    \end{equation}

    The calculation is a little more complicated for the remaining objects. For $n_1,n_2,n_3\geq 1$,
    \begin{multline}
        \connf_1^{\star n_1 \star n_2 \cdot n_3}\left(\orig\right) =  \frac{1}{\left(2\pi\right)^2}\int^\infty_{-\infty}\int^\infty_{-\infty}\e^{-n_1 \gamma\abs{k} - n_2\gamma\abs{k-l} - n_3\gamma\abs{l}}\dd k\dd l \\= \frac{1}{\left(2\gamma\pi\right)^2}\int^\infty_{-\infty}\int^\infty_{-\infty}\e^{-n_1 \abs{k} - n_2\abs{k-l} - n_3\abs{l}}\dd k\dd l.
    \end{multline}
    For $l\geq 0$, the $k$-integral can then be partitioned as
    \begin{align}
        \int^\infty_{-\infty}\e^{-n_1 \abs{k} - n_2\abs{k-l}}\dd k &= \int^\infty_{l}\e^{-n_1k - n_2k + n_2l}\dd k + \int^{l}_{0}\e^{-n_1k + n_2k-n_2l}\dd k + \int^0_{-\infty}\e^{n_1k + n_2k-n_2l}\dd k\nonumber\\
        &= \frac{1}{n_1+n_2}\left(\e^{-n_1l}+\e^{-n_2l}\right) + \begin{cases}
            l\e^{-n_1l} &\colon n_1=n_2\\
            \frac{1}{n_1-n_2}\left(\e^{-n_2l}-\e^{-n_1l}\right) &\colon n_1\ne n_2
        \end{cases}\nonumber\\
        &= \begin{cases}
            \left(\frac{1}{n_1} + l\right)\e^{-n_1l} &\colon n_1=n_2\\
            \frac{2n_1}{n_1^2- n_2^2}\e^{-n_2l} - \frac{2n_2}{n_1^2-n_2^2}\e^{-n_1l} &\colon n_1\ne n_2.
        \end{cases}
    \end{align}
    The calculation is performed similarly for $l<0$, and we get
    \begin{equation}
        \int^\infty_{-\infty}\e^{-n_1 \abs{k} - n_2\abs{k-l}}\dd k = \begin{cases}
            \left(\frac{1}{n_1} + \abs*{l}\right)\e^{-n_1\abs*{l}} &\colon n_1=n_2\\
            \frac{2n_1}{n_1^2- n_2^2}\e^{-n_2\abs*{l}} - \frac{2n_2}{n_1^2-n_2^2}\e^{-n_1\abs*{l}} &\colon n_1\ne n_2
        \end{cases}
    \end{equation}
    for all $l\in\R$.

    For $n_1=n_2$ we then get
    \begin{align}
        \connf_1^{\star n_1\star n_1\cdot n_3}\left(\orig\right) &= \frac{1}{4\gamma^2\pi^2}\int^\infty_{-\infty}\left(\frac{1}{n_1}+\abs*{l}\right)\e^{-\left(n_1+n_3\right)\abs*{l}}\dd l\nonumber\\
        &= \frac{1}{2\gamma^2\pi^2}\left(\frac{1}{n_1\left(n_1+n_3\right)} + \frac{1}{\left(n_1+n_3\right)^2}\right) \nonumber\\
        &=\frac{2n_1 + n_3}{2n_1\left(n_1+n_3\right)^2}\frac{1}{\gamma^2\pi^2}.\label{eqn:n1=n2Case}
    \end{align}
    Using $n_1=n_2=2$ and $n_3=1$, and $n_1=n_2=2$ and $n_3=2$ gives us two of our desired results. We are only left with $\connf^{\star1\star2\cdot 3}\left(\orig\right)$.

    For $n_1\ne n_2$ we get
    \begin{align}
        \frac{1}{\left(2\gamma\pi\right)^2}\int^\infty_{-\infty}\int^\infty_{-\infty}\e^{-n_1 \abs{k} - n_2\abs{k-l} - n_3\abs{l}}\dd k\dd l &= \frac{1}{2\gamma^2\pi^2}\int^\infty_{0}\left(\frac{2n_1}{n_1^2- n_2^2}\e^{-n_2l} - \frac{2n_2}{n_1^2-n_2^2}\e^{-n_1l}\right)\e^{-n_3l}\dd l\nonumber\\
        &= \frac{1}{\gamma^2\pi^2}\left(\frac{n_1}{\left(n_1^2-n_2^2\right)\left(n_2+n_3\right)} - \frac{n_2}{\left(n_1^2-n_2^2\right)\left(n_1+n_3\right)}\right)\nonumber\\
        &= \frac{n_1+n_2+n_3}{\left(n_1+n_2\right)\left(n_1+n_3\right)\left(n_2+n_3\right)}\frac{1}{\gamma^2\pi^2}.
    \end{align}
    Note that this expression reduces to the case \eqref{eqn:n1=n2Case} if $n_1=n_2$.
\end{proof}

\begin{lemma}
    The Coordinate-Cauchy RCM with $\liminf_{d\to\infty}\connf\left(\orig\right)^{\frac{1}{d}}>0$ satisfies Assumptions~\ref{Assumption} and \ref{AssumptionBeta}.
\end{lemma}

\begin{proof}
    For simplicity we scale space so that $\phiint=\Acal=1$. As argued analogously for the Gaussian RCM in Lemma~\ref{lem:GaussianAssumptions}, the condition $\liminf\connf\left(\orig\right)^{\frac{1}{d}}>0$ then becomes $\limsup \gamma <\infty$, and $\connf\left(\orig\right)\leq 1$ becomes $\gamma\geq \sfrac{1}{\pi}$.

    Since $\fconnf(k) = \e^{-\gamma\norm*{k}_1}\geq 0$, we know that $\esssup_{x\in\Rd}\connf^{\star m}(x) = \connf^{\star m}\left(\orig\right)$ for all $m\geq 1$. Therefore
    \begin{equation}
       \esssup_{x\in\Rd}\connf^{\star m}(x) = \left(m\gamma\pi\right)^{-d}.
    \end{equation}
    Since $\gamma\pi\geq 1$, this approaches zero for all $m\geq 2$. Therefore $\ref{Assump:DecayBound}$ holds with the choice $g(d) = \left(2\gamma\pi\right)^{-d}= 2^{-d}\connf\left(\orig\right)$ and $\beta(d)=2^{-\frac{d}{4}}\connf\left(\orig\right)^{\frac{1}{4}}$ (here we use $\limsup \gamma <\infty$ to get the appropriate form of $\beta$ from \eqref{eqn:betafromgfunction}). Furthermore, $\gamma$ cannot approach $0$ and therefore our expression for $\fconnf(k)$ implies \ref{Assump:QuadraticBound} holds too.

    From our prior calculations we have $\connf^{\star 6}\left(\orig\right) =6^{-d}\phiint^5\connf\left(\orig\right)$ and therefore \ref{Assump:ExponentialDecay} can be seen to hold with $\rho=6^{-1}\liminf \connf\left(\orig\right)^{\frac{1}{d}}>0$. It also allows us to lower bound $h(d)$. Noting that $\log\beta(d)<0$, this implies that
    \begin{equation}
        \frac{\log h(d)}{\log \beta(d)} \leq \frac{-d\log 6 + \log \connf\left(\orig\right)}{-\frac{d}{4}\log 2 + \frac{1}{4}\log\connf\left(\orig\right)} \leq \frac{4\log 6  - 4\log \connf\left(\orig\right)^{\frac{1}{d}}}{\log 2 - \log \connf\left(\orig\right)^{\frac{1}{d}}}.
    \end{equation}
    Note that $\log \connf\left(\orig\right)^{\frac{1}{d}}\leq0$. By taking the derivative of the map $x\mapsto \frac{4\log 6 - 4x}{\log 2-x}$ for $x\leq 0$ we can find that it is maximised at $x=0$. Therefore
    \begin{equation}
        \frac{\log h(d)}{\log \beta(d)} \leq 4\frac{\log 6}{\log 2} = 4\log_2 6.
    \end{equation}
    Since this is finite, we have proven Assumption~\ref{Assump:NumberBound}.
\end{proof}

\end{appendix}

\bibliography{bibliography}{}
\bibliographystyle{alpha}

\paragraph{Acknowledgements.} This work is supported by  
\textit{Deutsche Forschungsgemeinschaft} (project number 443880457) through priority program ``Random Geometric Systems'' (SPP 2265). 
The authors thank Sabine Jansen and Günter Last for their inspiring discussions. 

\end{document}